\documentclass[12p]{amsart}
\usepackage{amssymb}
\usepackage{amsmath}
\usepackage{amsfonts,color}
\usepackage{geometry}
\usepackage{graphicx}
\usepackage{mathrsfs,amssymb}

\usepackage{hyperref}
\usepackage{cleveref}

\theoremstyle{plain}

\newtheorem{claim}{Claim}

\newtheorem{definition}{Definition}

\newtheorem{lemma}{Lemma}

\newtheorem{proposition}{Proposition}
\newtheorem{remark}{Remark}

\newtheorem{theorem}{Theorem}
\numberwithin{equation}{section}

\begin{document}
\title{Schr{\"o}dinger maps to a K{\"a}hler manifold in two dimensions}

\author{Benjamin Dodson and Jeremy L. Marzuola}
\date{\today}

\begin{abstract}
We prove a global well--posedness and scattering result for Schr{\"o}dinger maps to a general K{\"a}hler manifold with small initial data in a Besov space.
\end{abstract}

\maketitle

\section{Introduction}

Let $(\mathcal N, J, h)$ denote a K{\"a}hler manifold, where $\mathcal N$ is the manifold, $J$ is the almost complex structure, and $h$ is the metric on $\mathcal N$. A Schr{\"o}dinger map from a Euclidean space to a K{\"a}hler manifold is a map
\begin{equation}\label{1.1}
u : \mathbb{R} \times \mathbb{R}^{d} \rightarrow \mathcal N,
\end{equation}
that satisfies the equation
\begin{equation}\label{1.2}
u_{t} = J(\sum_{k = 1}^{d} \nabla_{k} \partial_{k} u), \qquad u|_{t = 0} = u_{0}.
\end{equation}
Here $\nabla$ denotes the induced covariant derivative in the pullback bundle from the metric $h$, $u^{\ast} T \mathcal N$.\medskip

Frequently, research has focused on the study of $(\ref{1.1})$ when $\mathcal N$ is two--dimensional and of constant curvature. When $\mathcal N = \mathbb{C}$, $J = i$, and $h$ is the usual Euclidean metric, $(\ref{1.1})$ reduces to the usual linear Schr{\"o}dinger equation
\begin{equation}\label{1.3}
u_{t} = i \Delta u, \qquad u|_{t = 0} = u_{0}, \qquad u : \mathbb{R} \times \mathbb{R}^{d} \rightarrow \mathbb{C}.
\end{equation}
This equation is well understood, as solutions have the form
\begin{equation}\label{1.4}
u(t, x) = (4 \pi t)^{-d/2} e^{-i \frac{d \pi}{4}} \int e^{i \frac{|x - y|^{2}}{4t}} u_{0}(y) dy,
\end{equation}
(see for example Chapter six of \cite{taylorpartial}).\medskip

When $\mathcal N$ has constant curvature $\pm 1$, $\mathcal N = \mathbb{S}^{2}$ or $\mathbb{H}^{2}$, then $(\ref{1.1})$ is a nonlinear equation. Let $J$ be the usual complex structure on $\mathbb{S}^{2}$ given by $u \times \cdot$ and let $h$ be the usual Riemannian metric on $\mathbb{S}^{2}$. Then $(\ref{1.2})$ becomes
\begin{equation}\label{1.5}
u_{t} = u \times \Delta u, \qquad u|_{t = 0} = u_{0}.
\end{equation}
The small data global well--posedness theory for $(\ref{1.3})$ has been well--studied for $u : \mathbb{R} \times \mathbb{R}^{d} \rightarrow \mathbb{S}^{2} \hookrightarrow \mathbb{R}^{3}$ in different dimensions $d$. In general, equation $(\ref{1.5})$ is $\dot{H}^{\frac{d}{2}}$--critical. Indeed, a solution to $(\ref{1.1})$ obeys the scaling symmetry
\begin{equation}\label{1.6}
u(t, x) \mapsto u(\lambda^{2} t, \lambda x), \qquad \text{for any} \qquad \lambda > 0.
\end{equation}
Since the $\dot{H}^{d/2}$ norm of the initial data is invariant under the scaling $(\ref{1.6})$, $(\ref{1.1})$ is called $\dot{H}^{d/2}$--critical.
\begin{remark}
Observe that the critical norm is independent of $\mathcal N$.
\end{remark}
\begin{theorem}[Small data global well--posedness]\label{t1}
For data that is small in $\dot{H}^{d/2}$ norm, global well--posedness for $(\ref{1.5})$ has been proved in all dimensions $d \geq 2$.
\end{theorem}
Theorem $\ref{t1}$ was proved in \cite{bejenaru2007global} for dimensions $d \geq 4$, \cite{ionescu2006low} for dimensions $d \geq 3$ and \cite{bejenaru2011global} for dimensions $d \geq 2$. The proof of global well--posedness in \cite{bejenaru2007global},  \cite{ionescu2006low}, and \cite{bejenaru2011global} relies on the local well--posedness result of \cite{mcgahagan2007approximation}. Specifically, $(\ref{1.1})$ is locally well--posed for sufficiently smooth initial data. Therefore, for smooth data (which will be precisely defined later) with sufficiently small $\| u_{0} \|_{\dot{H}^{d/2}} \ll 1$,
\begin{equation}
\label{1.7}
\| u(t) \|_{H^{\sigma}} \leq C(\sigma) \| u_{0} \|_{H^{\sigma}},
\end{equation}
for any $\sigma \in \mathbb{N}$, which implies global well--posedness.\medskip

Dimension $d = 2$ is particularly important due to the existence of a conserved energy for $(\ref{1.5})$,
\begin{equation}\label{1.8}
E(t) = \int_{\mathbb{R}^{d}} \sum_{m = 1}^{d} |\partial_{m} u(t, x)|^{2} dx.
\end{equation}
Since $(\ref{1.6})$ controls the $\dot{H}^{1}$--norm, it is reasonable to conjecture that the local well--posedness result can be extended to a global well--posedness result for solutions below the energy of the soliton solution. This problem remains open, although the result has been obtained for solutions to $(\ref{1.1})$ with equivariant symmetry, see \cite{bejenaru2013equivariant}.   In the $2$-equivariant case, Besov regularity recently played a role in the analysis of \cite{bejenaru2024near}.  \medskip

When $\mathcal N = \mathbb{H}^{2}$, global well--posedness for equivariant solutions to $(\ref{1.1})$ is known, see \cite{bejenaru2016equivariant}. There is no soliton when the target manifold is $\mathcal N = \mathbb{H}^{2}$, so global well-posedness holds for any initial data with finite energy. The proof in \cite{bejenaru2013equivariant} and \cite{bejenaru2016equivariant} utilizes Morawetz estimates and the local energy estimates in \cite{grillakis2002lagrangian} adapted to the equivariant problem.\medskip

The above results were generalized to the case when $\mathcal N$ is a general two dimensional compact K{\"a}hler manifold in Theorem $1.1$ of \cite{li2018global}.  We further generalize this result to initial data lying in a Besov space that contains small energy initial data. There are three reasons why such a result is useful. The first is that the Besov space strictly contains small energy, and thus is a broader class of initial data. The second reason is that to generalize the results in \cite{bejenaru2013equivariant} and \cite{bejenaru2016equivariant} to all initial data with finite energy below the soliton such a result will probably be needed in order to utilize the concentration compactness argument. The third reason is that the proof here is substantially simpler than the arguments in \cite{bejenaru2011global} and  \cite{li2018global}. In particular, the proof uses fairly ``ordinary" function spaces, mostly Strichartz spaces and bilinear Strichartz spaces.\medskip

Before stating the result of this paper, it is useful to explain a few of the reasons why the result is stated in the way that it is. When the target $\mathcal N$ is a curved manifold, assigning function spaces to contain a map $u : \mathbb{R}^{d} \rightarrow \mathcal N$ is necessarily complicated by the curvature of the target, and there are several issues to consider. One such issue is the lack of a canonical zero. Recalling $(\ref{1.4})$ in the case when $\mathcal N = \mathbb{C}$ and $J = i$, observe that if $u_{0} \in L^{1}(\mathbb{R}^{d})$, $\| u \|_{L^{\infty}} \lesssim t^{-d/2} \| u_{0} \|_{L^{1}}$. However, analyzing the dispersive nature of $u : \mathbb{R} \times \mathbb{R}^{d} \rightarrow \mathcal N = \mathbb{S}^{2}$ is necessarily complicated by the fact that if $\mathcal N = \mathbb{S}^{2}$ (for example), $|u(t, x)| \equiv 1$ everywhere. Following \cite{chang2000schrodinger}, we choose $Q \in \mathcal N$ and show that if $u_{0} - Q$ lies in some function space, then the solution $u(t,x)$ tends toward $Q$ in some space.\medskip

A second issue is that Banach spaces are closed under multiplication by scalars. However, again if $\mathcal N = \mathbb{S}^{2}$ (for example), $|u(t, x)| \equiv 1$, and thus, $\lambda u(t, x)$ does not lie in $\mathbb{S}^{2}$ for any $\lambda$ such that $|\lambda| \neq 1$. If $u$ lies close to $Q$ in $L^{\infty}$ norm, one option is to use a coordinate projection in a neighborhood of $Q$. This is exactly what \cite{chang2000schrodinger} does. However, for initial data lying in the function space $\dot{H}^{d/2}(\mathbb{R}^{d})$, the coordinate projection is no longer as useful, even for small data. This is because $\dot{H}^{d/2}$ does not embed into $L^{\infty}(\mathbb{R}^{d})$, and thus, if $u$ has small $\dot{H}^{d/2}$ norm, $u$ need not lie close to $Q$. A second option is to use the extrinsic formulation of the function space. That is, since $\mathcal N \hookrightarrow \mathbb{R}^{N}$ for some $N$, we define the function space
\begin{equation}\label{1.9}
X = \{ f : \mathbb{R}^{d} \rightarrow \mathbb{R}^{N} : \| f - Q \|_{X} < \infty \},
\end{equation}
where $X$ is a Banach norm on functions mapping from $\mathbb{R}^{d}$ to $\mathbb{R}^{N}$. This is the space that we shall use. 

We will define the Sobolev spaces and Littlewood--Paley theory using the extrinsic definition.
\begin{definition}\label{d1.1}
Suppose $\mathcal N$ is a compact K{\"a}hler manifold, and suppose $Q \in \mathcal N$ and $\mathcal N \subset \mathbb{R}^{N}$ for some integer $N$.

Then define the extrinsic Sobolev space $H_{Q}^{k}$ by
\begin{equation}\label{1.10}
H_{Q}^{k} = \{ u : \mathbb{R}^{d} \rightarrow \mathbb{R}^{N} : u(x) \in \mathcal N, \qquad a.e. \qquad \text{in} \qquad \mathbb{R}^{d}, \qquad \| u - Q \|_{H^{k}(\mathbb{R}^{d})} < \infty \},
\end{equation}
which is equipped with the metric $d_{Q}(f, g) = \| f - g \|_{H^{k}}$. Define $\mathcal H_{Q}$ to be
\begin{equation}\label{1.11}
\mathcal H_{Q} = \cap_{k = 1}^{\infty} H_{Q}^{k}.
\end{equation}
\end{definition}
This definition was used in \cite{bejenaru2011global} for the sphere $S^{2} \hookrightarrow \mathbb{R}^{3}$. This definition was also used in \cite{li2021global} for a general compact K{\"a}hler manifold.

We will utilize the spaces $B_{p, q}^{s}$, the usual Besov spaces with norm
\begin{equation}\label{1.12}
\| f \|_{B_{p, q}^{s}} = (\sum_{j}  2^{jps}  \| P_{j} f \|_{L^{q}}^{p})^{1/p},
\end{equation}
with the obvious modifications when $p = \infty$, $q = \infty$, or both. Here, $P_{j}$ is the standard Littlewood--Paley projection operator.\medskip

\begin{theorem}\label{t1.1}
Let $d = 2$.  For $n \geq 1$, let $\mathcal N$ be a $2n$-dimensional compact K{\"a}hler manifold that is isometrically embedded into $\mathbb{R}^{N}$, and let $Q \in \mathcal N$ be a given point. There exists a sufficiently small constant $\epsilon_{\ast}(\| u_{0} - Q \|_{\dot{H}^{1}}) \ll 1$ such that if $u_{0} \in \mathcal H_{Q}$ satisfies
\begin{equation}\label{1.13}
\| u_{0} - Q \|_{B_{\infty, 2}^{1}} \leq \epsilon_{\ast} \ll 1,
\end{equation}
then $(\ref{1.1})$ with initial data $u_{0}$ evolves into a global unique solution $u \in C(\mathbb{R}; \mathcal H_{Q})$. Moreover, as $|t| \rightarrow \infty$ the solution $u$ converges to a constant map in the sense that
\begin{equation}\label{1.14}
\lim_{|t| \rightarrow \infty} \| u(t) - Q \|_{L_{x}^{\infty}} = 0.
\end{equation}
Furthermore, in the energy space properly interpreted using the complex structure on $\mathcal N$,  we also have 
\begin{equation}\label{1.15}
\lim_{t \rightarrow \infty} \| u(t) -  \sum_{j=1}^n (Re(e^{it \Delta} h_{+}^{j}) -  Im(e^{it \Delta} g_{+}^{j})) \|_{\dot{H}^{1}} = 0,
\end{equation}
for some functions $h_{+}^{j}$, $g_{+}^{j} : \mathbb{R}^{2} \rightarrow \mathbb{C}^{n}$ belonging to $\dot{H}^{1}$ with $j = 1,\dots,n$.
\end{theorem}

\begin{remark}
Our scattering result \eqref{1.15} will actually be proven as an $L^2$ scattering result on a gauged system of equations that will be introduced shortly, then the result \eqref{1.15} can follow as in \cite{li2018global}, Section $7.5$, in particular the arguments leading to equation $(7.29)$.
\end{remark}

Notice that Theorem $\ref{t1.1}$ is a generalization of Theorem $1.1$ in \cite{li2018global} and also Theorem $1.1$ \cite{bejenaru2011global}. Both \cite{bejenaru2011global} and \cite{li2018global} studied Schr{\"o}dinger maps for small initial data in $B_{2, 2}^{1} = \dot{H}^{1}$. Equivalently,
\begin{equation}\label{1.16}
\| \nabla u \|_{L^{2}} \leq \epsilon_{\ast} \ll 1.
\end{equation}
However, while $(\ref{1.16})$ has a natural intrinsic formulation, $(\ref{1.14})$ does not, and therefore must be formulated extrinsically.

\subsection{Outline of argument}
The argument in \cite{bejenaru2011global} and \cite{li2018global} studied the derivative of the solution to $(\ref{1.1})$, which satisfies the equation
\begin{equation}\label{1.17}
(i \partial_{t} + \Delta_{x}) \psi_{m} = -2i \sum_{l = 1}^{d} A_{l} \partial_{l} \psi_{m} + (A_{d + 1} + \sum_{l = 1}^{d} (A_{l}^{2} - i \partial_{l} A_{l})) \psi_{m} - i \sum_{l = 1}^{d} Im(\bar{\psi}_{l} \psi_{m}) \psi_{l},
\end{equation}
under a suitable gauge condition (the caloric gauge). This is what we will do as well. Here, $(v, w)$ is a suitable frame in $T_{u} \mathcal N$,
\begin{equation}\label{1.18}
\psi_{m} = v \cdot \partial_{m} u + i w \cdot \partial_{m} u,
\end{equation}
and
\begin{equation}\label{1.19}
A_{m} = w \cdot \partial_{m} v.
\end{equation}
\begin{remark}
For convenience let $\psi_{x}(t)$ denote the vector with components $\psi_{m}(t)$, $m = 1, 2$.
\end{remark}

The next step of the proof is to use the result of \cite{mcgahagan2007approximation} to prove local well--posedness.
\begin{theorem}[Local in Time $H^3$ result from \cite{mcgahagan2007approximation}]\label{t8.1}
For $u_{0} \in H^{3}$, there exists $$T(\| u_{0} \|_{H^{3}}) > 0$$ such that $(\ref{1.2})$ is locally well-posed on the interval $[-T, T]$. Furthermore, the solution $$u \in C([-T, T]; H^{3}) \cap L^{\infty}([-T, T]; H^{3})$$ and $\| u \|_{L_{t}^{\infty} H_{x}^{3}([-T, T] \times \mathbb{R}^{2})} \leq 2 \| u_{0} \|_{H^{3}}$.
\end{theorem}

Since $u \in C([-T, T]; H^{3})$, there exists some $0 < \delta(\| u_{0} \|_{H^{3}}) \leq T$ such that
\begin{equation}\label{Besov}
\| \psi_{x}(t) \|_{B_{\infty, 2}^{0}} \leq 2 \epsilon^{\ast}, \qquad \text{for all} \qquad t \in [-\delta, \delta].
\end{equation}

\noindent By the Sobolev embedding theorem and Bernstein's inequality,
\begin{equation}\label{8.1}
\| P_{k} \psi_{x}(t) \|_{L_{x}^{4}(\mathbb{R}^{2})} \lesssim 2^{k/2} \| P_{k} \psi_{x}(t) \|_{L_{x}^{2}(\mathbb{R}^{2})} \lesssim 2^{k/2} \inf \{ 1, 2^{-3k} \} \| \psi_{x}(t) \|_{H_{x}^{3}(\mathbb{R}^{2})}.
\end{equation}
Therefore, by H{\"o}lder's inequality,
\begin{equation}\label{8.2}
\aligned
& \| P_{k} \psi_{x}(t) \|_{L_{t,x}^{4}([-T, T] \times \mathbb{R}^{2})} \lesssim 2^{k/2} T^{1/4} \inf \{ 1, 2^{-3k} \} \| \psi_{x}(t) \|_{H_{x}^{3}(\mathbb{R}^{2})} \\
&  \lesssim 2^{k/2} T^{1/4} \inf \{ 1, 2^{-3k} \} \| u_{0} \|_{H_{x}^{3}(\mathbb{R}^{2})}.
\endaligned
\end{equation}
Hence,
\begin{equation}\label{8.3}
(\sum_{k} \| P_{k} \psi_{x}(t) \|_{L_{t,x}^{4}([-T, T] \times \mathbb{R}^{2})}^{2})^{1/2} \lesssim T^{1/4} \| u_{0} \|_{H_{x}^{3}(\mathbb{R}^{2})}.
\end{equation}

For $T(\| u_{0} \|_{H^{3}})$ sufficiently small,
\begin{equation}\label{8.4}
(\sum \| P_{k} \psi_{x}(t) \|_{L_{t,x}^{4}([-T, T]}^{2})^{1/2} \leq \epsilon \ll 1.
\end{equation}
Therefore, we can use Theorem $1.2$ from \cite{dodsonsmith}, which essentially states that if 
$$\| \nabla \phi \|_{\ell^2 L^4_{t,x} ( I \times \mathbb{R}^2)} \leq \epsilon \ll 1$$ 
on some interval $I$, then the interval of existence can be extended.
\begin{remark}
\label{rem:DS15}
Two issues need to be addressed here. The first is that the result of \cite{dodsonsmith} is on $\mathcal M \in \{ \mathbb{S}^{2}, \mathbb{H}^{2} \}$, while here the manifold $\mathcal M$ is a more general K{\"a}hler manifold. Secondly, the result in \cite{dodsonsmith} requires that the initial data lie below the soliton, that is $E < E_{crit}$. We do not require that here either.

This is because the small Besov norm $(\ref{Besov})$ guarantees well-posedness of the harmonic map heat flow, which in turn implies that a caloric gauge exists. This is proved in Section \ref{sec:Gauges}. As a result, the computations in \cite{dodsonsmith} can be carried over to this setting as well.
\end{remark}

Theorem $4.3$ of \cite{dodsonsmith} states that if there exists $I$ such that $\| \psi_x \|_{\ell^2 L^4_{t,x} ( I \times \mathbb{R}^2) } \leq \epsilon \ll 1$, then for $|j-k|>10$, 
\begin{equation}
\label{eqn:DSthm43}
\| ( P_j  \psi_x (s)  )(  P_k \psi_x (s') )  \|_{L^2_{t,x} ( I \times \mathbb{R}^2)} \lesssim  2^{- \frac{|j-k|}{2} } v_j v_k (1 + s 2 ^{2j})^{-4} (1 + s'  2 ^{2k})^{-4} .
\end{equation} 
In particular, if we have smallness of the $L^4$ norm on $I$, then we have a bilinear estimate on that interval similar to the linear one.  Similar modifications as mentioned in Remark \ref{rem:DS15} hold for this estimate.

Applying Theorem $4.3$ of \cite{dodsonsmith} as is at the moment, let
\begin{equation}\label{8.5}
\alpha_{k} = \sup_{j \in \mathbb{Z}} 2^{-\delta |k - j|} \| P_{j} \psi_{x} \|_{L_{t,x}^{4}([-T, T])}, \qquad \beta_{k} = \sup_{j \in \mathbb{Z}} 2^{-\delta |k - j|} \| P_{j} \psi_{x}(0) \|_{L^{2}}, \qquad v_{k} = \alpha_{k} + \beta_{k},
\end{equation}
and for any $\sigma \in \mathbb{N}$,
\begin{align}\label{8.5.1}
\alpha_{k}(\sigma) & = \sup_{j \in \mathbb{Z}} 2^{-\delta |k - j|} 2^{j \sigma} \| P_{j} \psi_{x} \|_{L_{t,x}^{4}([-T, T])}, \ \ \beta_{k} = \sup_{j \in \mathbb{Z}} 2^{-\delta |k - j|} 2^{j \sigma} \| P_{j} \psi_{x}(0) \|_{L^{2}}, \\
& \ v_{k}(\sigma) = \alpha_{k}(\sigma) + \beta_{k}(\sigma). \notag
\end{align}
Theorem $4.3$ of \cite{dodsonsmith} implies
\begin{equation}\label{8.6}
\| (P_{j} \psi_{x}(s)) (P_{k} \psi_{x}(s')) \|_{L_{t,x}^{2}([-T, T] \times \mathbb{R}^{2})} \lesssim 2^{-\frac{|j - k|}{2}} v_{j} v_{k} (1 + s 2^{2j})^{-3} (1 + s' 2^{2k})^{-3}
\end{equation}
and
\begin{equation}\label{8.6.1}
\| (P_{j} \psi_{x}(s)) (P_{k} \psi_{x}(s')) \|_{L_{t,x}^{2}([-T, T] \times \mathbb{R}^{2})} \lesssim 2^{-\frac{|j - k|}{2}} 2^{-k \sigma} v_{j} v_{k}(\sigma) (1 + s 2^{2j})^{-3} (1 + s' 2^{2k})^{-3}.
\end{equation}
From $(\ref{8.6})$ and $(\ref{8.6.1})$,
\begin{equation}\label{8.8}
\sup_{t \in [-\delta, \delta]} \| P_{k} \psi_{x}(t) \|_{L^{2}} \leq \| P_{k} \psi_{x}(0) \|_{L^{2}} + O(\epsilon^{3}) \leq \frac{3}{2} \epsilon^{\ast}
\end{equation}
and
\begin{equation}\label{8.7}
\| \psi_{x}(t) \|_{L_{t}^{\infty} H^{3}([-\delta, \delta] \times \mathbb{R}^{3})} \leq \frac{3}{2} \| \psi_{x}(0) \|_{H^{3}}.
\end{equation}
Therefore, by continuity, $(\ref{8.8})$ and $(\ref{8.7})$ can be extended to $[-T, T]$.\medskip

Now suppose that for some interval $I \subset \mathbb{R}$, $0 \in I$, $(\ref{1.2})$ has a solution on $I \subset \mathbb{R}$ that satisfies
\begin{equation}\label{8.9}
\aligned
\| (P_{k} \psi_{x}(s))(P_{j} \psi_{x}(s')) \|_{L_{t,x}^{2}(I \times \mathbb{R}^{2})} \lesssim 2^{-\frac{|j - k|}{2}} v_{j} v_{k} (1 + s 2^{2k})^{-3} (1 + s' 2^{2j})^{-3}, \\
\| (P_{k} \psi_{x}(s))(P_{j} \psi_{x}(s')) \|_{L_{t,x}^{2}(I \times \mathbb{R}^{2})} \lesssim 2^{-\frac{|j - k|}{2}} 2^{-k \sigma} v_{j} v_{k}(\sigma) (1 + s 2^{2k})^{-3} (1 + s' 2^{2j})^{-3}
\endaligned
\end{equation}
where
\begin{equation}\label{8.10}
\aligned
\alpha_{k} = \sup_{j \in \mathbb{Z}} 2^{-\delta |k - j|} \| P_{j} \psi_{x} \|_{L_{t,x}^{4}(I \times \mathbb{R}^{2})}, \quad \beta_{k} = \sup_{j \in \mathbb{Z}} 2^{-\delta |k - j|} \| P_{j} \psi_{x}(0) \|_{L^{2}}, \\
 \alpha_{k}(\sigma) = \sup_{j \in \mathbb{Z}} 2^{-\delta |k - j|} 2^{j \sigma} \| P_{j} \psi_{x} \|_{L_{t,x}^{4}(I \times \mathbb{R}^{2})}, \quad \beta_{k}(\sigma) = \sup_{j \in \mathbb{Z}} 2^{-\delta |k - j|} 2^{j \sigma} \| P_{j} \psi_{x}(0) \|_{L^{2}}.
\endaligned
\end{equation}
We take
\begin{equation}
\label{eqn:vdef}
v_{k} := \alpha_{k} + \beta_{k},
\end{equation} 
which satisfies the bounds
\begin{equation}\label{8.11}
\sup_{k} v_{k} \lesssim \epsilon, \qquad \sum_{k} v_{k}^{2} \lesssim E^{4}.
\end{equation}
 We will similarly take $ v_{k} (\sigma) := \alpha_{k} (\sigma)+ \beta_{k} (\sigma)$.
\begin{remark}
Here we do not require that $\sum \alpha_{k}^{2} \lesssim \epsilon^{2}$.
\end{remark}

Applying the result of \cite{mcgahagan2007approximation} to the endpoint of $I = [-t_{1}, t_{2}]$, and using $(\ref{8.1})$--$(\ref{8.8})$, there exists some $\eta(\| u_{0} \|_{H^{3}}) > 0$ and a $C > 1$ such that $(\ref{8.9})$ holds on the interval $[-t_{1} - \eta, t_{2} + \eta]$, with
\begin{equation}\label{8.13}
\aligned
\| (P_{k} \psi_{x}(s))(P_{j} \psi_{x}(s')) \|_{L_{t,x}^{2}([-t_{1} - \eta, t_{2} + \eta] \times \mathbb{R}^{2})} \lesssim 2^{-\frac{|j - k|}{2}} C v_{j} v_{k} (1 + s 2^{2k})^{-3} (1 + s' 2^{2j})^{-3}, \\
\| (P_{k} \psi_{x}(s))(P_{j} \psi_{x}(s')) \|_{L_{t,x}^{2}([-t_{1} - \eta, t_{2} + \eta] \times \mathbb{R}^{2})} \lesssim 2^{-\frac{|j - k|}{2}} C v_{j} v_{k}(\sigma) 2^{-k \sigma} (1 + s 2^{2k})^{-3} (1 + s' 2^{2j})^{-3},
\endaligned
\end{equation}
\begin{equation}\label{8.13.0}
\sum_{j} v_{j}^{2} \lesssim C E, \qquad \sup_{j} v_{j} \lesssim C E^{4} \epsilon^{\ast},
\end{equation}
\begin{equation}\label{8.13.1}
\sup_{t \in [-t_{1} - \eta, t_{2} + \eta]} \| \psi_{x}(t) \|_{B_{\infty, 2}^{0}} \leq 2 \epsilon^{\ast},
\end{equation}
\begin{equation}\label{8.13.2}
\| \psi_{x}(t) \|_{L_{t}^{\infty} H^{3}([-t_{1} - \eta, t_{2} + \eta] \times \mathbb{R}^{2})} \leq 2 \| \psi_{x}(0) \|_{H^{3}}.
\end{equation}
Note, here the $\lesssim C (\cdot)$ means that the bound on the right hand side of our bootstrapping assumption is taken to be a large constant multiplied by the implicit constant we will get in the bilinear bound eventually.

Bootstrapping the bounds $(\ref{8.13})$--$(\ref{8.13.2})$ with the constant $C >1$, we can then prove
\begin{equation}\label{8.14}
\aligned
\| (P_{k} \psi_{x}(s))(P_{j} \psi_{x}(s')) \|_{L_{t,x}^{2}([-t_{1} - \eta, t_{2} + \eta] \times \mathbb{R}^{2})} \lesssim 2^{-\frac{|j - k|}{2}} v_{j} v_{k} (1 + s 2^{2k})^{-3} (1 + s' 2^{2j})^{-3}, \\
\| (P_{k} \psi_{x}(s))(P_{j} \psi_{x}(s')) \|_{L_{t,x}^{2}([-t_{1} - \eta, t_{2} + \eta] \times \mathbb{R}^{2})} \lesssim 2^{-\frac{|j - k|}{2}} v_{j} v_{k}(\sigma) 2^{-k \sigma} (1 + s 2^{2k})^{-3} (1 + s' 2^{2j})^{-3},
\endaligned
\end{equation}
\begin{equation}\label{8.14.0}
\sum_{j} v_{j}^{2} \lesssim E^{4}, \qquad \sup_{j} v_{j} \lesssim E^{4} \epsilon^{\ast},
\end{equation}
\begin{equation}\label{8.14.1}
\sup_{t \in [-t_{1} - \eta, t_{2} + \eta]} \| \psi_{x}(t) \|_{B_{\infty, 2}^{0}} \leq \frac{3}{2} \epsilon^{\ast},
\end{equation}
\begin{equation}\label{8.14.2}
\| \psi_{x}(t) \|_{L_{t}^{\infty} H^{3}([-t_{1} - \eta, t_{2} + \eta] \times \mathbb{R}^{2})} \leq \frac{3}{2} \| \psi_{x}(0) \|_{H^{3}}.
\end{equation}
By a standard continuity argument, this implies that $I$ can be extended to $\mathbb{R}$, giving scattering through our control of the $L^4$ norm using the $\alpha$ frequency envelope. 
The proof of $(\ref{8.14})$--$(\ref{8.14.2})$ uses the interaction Morawetz estimate to obtain bilinear Strichartz estimates for quasilinear Schr{\"o}dinger equations, see \cite{ifrim2023global}, \cite{ifrim2023long}, and \cite{ifrim2023global2}.  

We proceed as follows.  In Section \ref{sec:warmup}, we demonstrate how the interaction Morawetz estimate simplifies scattering results for the standard cubic NLS equation when working in Besov spaces.   We will demonstrate the argument to establish scattering in Section \ref{sec:wuscat} in the presence of a global $L^4$ norm, so in the remainder of the proof focusing on the Schr\"odinger map, we will focus more on getting the global $L^4$ bounds.  

Starting in Section \ref{sec:HMHF}, we introduce the appropriate Besov spaces in the geometric setting and use them to prove that the Harmonic Map Heat Flow is well defined in these spaces.  With the Harmonic Map Heat Flow established, in Section \ref{sec:Gauges}, we can define the caloric gauge and set up the problem in covariant coordinates.  We also demonstrate that smallness in the Besov norm translates to smallness in the initial data for the gauged equations.  

The crucial bootstrap bounds will be done by proving the following chain of Lemmas proved in Section \ref{sec:SchMap}.
\begin{lemma}[Change in mass]\label{lmass}
Under the bootstrap assumptions \eqref{8.13}-\eqref{8.13.2}, we have
\begin{equation}
\| P_{k} \psi_{x}(t) \|_{L^{2}}^{2} \leq \| P_{k} \psi_{x}(0) \|_{L^{2}}^{2} + c_{k}^{2} \epsilon^{2} \leq \beta_{k}^{2} + v_{k}^{2} \epsilon^{2}.
\end{equation}
This proves $(\ref{8.14.1})$.

Also, for any $\sigma \in \mathbb{N}$,
\begin{equation}
\| P_{k} \psi_{x}(t) \|_{L^{2}}^{2} \leq 2^{-2k \sigma} \beta_{k}(\sigma)^{2} + \epsilon^{2} 2^{-2k \sigma} v_{k}(\sigma)^{2}.
\end{equation}
This proves $(\ref{8.14.2})$.
\end{lemma}

\begin{lemma}[Bilinear estimate]\label{lbilinear}
Under the bootstrap assumptions \eqref{8.13}-\eqref{8.13.2}, we have that for $k - j > 10$,
\begin{equation}\label{bilinear}
	\| (P_{k} \psi_{x}(s)) (P_{j} \psi_{x}(s')) \|_{L_{t,x}^{2}} \lesssim 2^{\frac{j - k}{2}} \{ \beta_{k}\beta_{j} + C \epsilon v_{j} v_{k} \}  (1 + s 2^{2k})^{-3} (1 + s' 2^{2j})^{-3}  ,
\end{equation}
and

\begin{equation}
\| (P_{k} \psi_{x}(s)) (P_{j} \psi_{x}(s')) \|_{L_{t,x}^{2}} \lesssim 2^{\frac{j - k}{2}} 2^{-k \sigma} \{ \beta_{k}(\sigma) \beta_{j} + C \epsilon v_{j} v_{k}(\sigma) \}   (1 + s 2^{2k})^{-3} (1 + s' 2^{2j})^{-3}.
\end{equation}
This proves $(\ref{8.14})$.
\end{lemma}

\begin{lemma}\label{lL4}
Under the bootstrap assumptions \eqref{8.13}-\eqref{8.13.2}, we have
\begin{equation}
\| P_{k} \psi_{x}(t) \|_{L_{t,x}^{4}} \lesssim \beta_{k} \sum_{j} \beta_{j}^{2} + \epsilon v_{k} \sum_{j} v_{j}^{2}.
\end{equation}
Therefore,
\begin{equation}
\| P_{k} \psi_{x}(t) \|_{L_{t,x}^{4}} \lesssim \beta_k E + v_k (C \epsilon E) E^{4}
\end{equation}
and
\begin{equation}
\sum_{k} \| P_{k} \psi_{x}(t) \|_{L_{t,x}^{4}}^{2} \lesssim E^{2} + \epsilon C E^{6}.
\end{equation}
Also,
\begin{equation}
\| P_{k} \psi_{x}(t) \|_{L_{t,x}^{4}} \lesssim \beta_{k}(\sigma) E + \epsilon v_{k}(\sigma) E^{4}.
\end{equation}
This proves $(\ref{8.14.0})$.
\end{lemma}

Finally, in Section \ref{sec:thm2pf}, we are able to apply the bilinear estimates in the gauged equations to finish the proof of Theorem $\ref{t1.1}$.  

\section{A warmup result}
\label{sec:warmup}

As a warmup, we will prove a global well-posedness and scattering result for the nonlinear Schr{\"o}dinger equation
\begin{equation}\label{2.1}
i u_{t} + \Delta u = F(u) = \mu |u|^{2} u, \qquad \mu = \pm 1, \qquad u(0,x) = u_{0},
\end{equation}
where $u_{0} \notin L^{2}(\mathbb{R}^{2})$, but $\| u_{0} \|_{B_{2, \infty}^{0}} \leq \epsilon_{0}$, for some $\epsilon_{0} > 0$ sufficiently small. The space $B_{p, q}^{s}$ is the Besov space defined by the norm
\begin{equation}\label{2.2}
\| f \|_{B_{p, q}^{s}(\mathbb{R}^{2})} = (\sum_{j = -\infty}^{\infty} 2^{jqs} \| P_{j} f \|_{L^{p}(\mathbb{R}^{2})}^{q})^{1/q}
\end{equation}
when $q < \infty$, and
\begin{equation}\label{2.2inf}
\| f \|_{B_{p, \infty}^{s}(\mathbb{R}^{2})} = \sup_{j \in \mathbb{Z}} 2^{js} \| P_{j} f \|_{L^{p}(\mathbb{R}^{2})}.
\end{equation}
Here, $P_{j}$ is the Littlewood--Paley projection operator
\begin{equation}\label{2.3}
P_{j} f = \mathcal F^{-1}(\psi(2^{-j} \xi) \hat{f}(\xi)),
\end{equation}
where
\begin{equation}\label{2.4}
\psi(\xi) = \phi(\frac{\xi}{2}) - \phi(\xi), \qquad \phi(\xi) \in C_{0}^{\infty}(\mathbb{R}^{2}), \qquad \phi(\xi) = 1, \qquad \text{for} \qquad |\xi| \leq 1.
\end{equation}

\begin{theorem}\label{t2.1}
There exists $\epsilon_{0} > 0$ such that if $\| u_{0} \|_{B_{2, \infty}^{0}} \leq \epsilon_{0}$, then $(\ref{2.1})$ is globally well-posed. Furthermore, there exist $u_{-}$, $u_{+} \in B_{2, \infty}^{0}$ such that if $u$ is the solution to $(\ref{2.1})$,
\begin{equation}\label{2.5}
e^{-it \Delta} u(t, \cdot) \rightharpoonup u_{+},
\end{equation}
as $t \nearrow \infty$, weakly in $B_{2, \infty}^{0}$, and
\begin{equation}\label{2.6}
e^{-it \Delta} u(t, \cdot) \rightharpoonup u_{-},
\end{equation}
as $t \searrow -\infty$, weakly in $B_{2, \infty}^{0}$.
\end{theorem}
 Theorem $\ref{t2.1}$ is a small data result, and does not depend on the sign of the nonlinearity. We do not require that $u_{0} \in L^{2}$. It is known from \cite{dodson2016global} that Theorem $\ref{t2.1}$ holds for $u_{0} \in L^{2}$ and $\mu = +1$, and from \cite{dodson2015global} that Theorem $\ref{t2.1}$ holds for $\| u_{0} \|_{L^{2}} < \| Q \|_{L^{2}}$ and $\mu = -1$.
 \begin{proof}
 Theorem $\ref{t2.1}$ follows from Duhamel's principle and the following two estimates:
 \begin{equation}\label{2.7}
 \sup_{j, k \in \mathbb{Z}} 2^{\frac{|j - k|}{2}} \| (P_{j} u)(P_{k} u) \|_{L_{t,x}^{2}(\mathbb{R} \times \mathbb{R}^{2})} \lesssim \epsilon_{0}^{2}
 \end{equation}
 and
 \begin{equation}\label{2.8}
 \sup_{j, k \in \mathbb{Z}, \| u_{0} \|_{L^{2}} = 1} \| (P_{j} u)(P_{k} e^{it \Delta} u_{0}) \|_{L_{t,x}^{2}(\mathbb{R} \times \mathbb{R}^{2})} \lesssim \epsilon_{0}.
 \end{equation}
 
 When $|j - k| \leq 10$, $(\ref{2.7})$ and $(\ref{2.8})$ will follow from ordinary Strichartz estimates. When $|j - k| > 10$, the proof of $(\ref{2.7})$ and $(\ref{2.8})$ will utilize bilinear Strichartz estimates. Such estimates were proved in \cite{bourgain1998refinements} using Fourier analysis, and they are of the form
 \begin{equation}\label{2.9}
 \| (e^{it \Delta} P_{j} u_{0})(e^{it \Delta} P_{k} v_{0}) \|_{L_{t,x}^{2}(\mathbb{R} \times \mathbb{R}^{2})} \lesssim 2^{\frac{j - k}{2}} \| P_{j} u_{0} \|_{L^{2}} \| P_{k} v_{0} \|_{L^{2}}.
 \end{equation}
It is useful to prove $(\ref{2.7})$ using the interaction Morawetz estimate.
 
\begin{theorem}\label{t3.1}
If $j \leq k - 10$, then
\begin{equation}\label{2.11}
\| (e^{it \Delta} P_{j} u_{0})(e^{it \Delta} P_{k} v_{0}) \|_{L_{t,x}^{2}(\mathbb{R} \times \mathbb{R}^{2})} \lesssim 2^{\frac{j - k}{2}} \| P_{j} u_{0} \|_{L^{2}} \| P_{k} v_{0} \|_{L^{2}}.
\end{equation}
\end{theorem}
\begin{proof}
Let $v = e^{it \Delta} P_{k} v_{0}$ and $u = e^{it \Delta} P_{j} u_{0}$. Then define the interaction Morawetz potential,
\begin{equation}\label{2.12}
M(t) = \int |u(t, y)|^{2} \frac{(x - y)_{1}}{|(x - y)_{1}|} Im[\bar{v} \partial_{1} v](t, x) dx dy + \int |v(t, y)|^{2} \frac{(x - y)_{1}}{|(x - y)_{1}|} Im[\bar{u} \partial_{1} u](t, x) dx dy.
\end{equation}
Since $e^{it \Delta}$ is a unitary operator on $L^{2}$-based Sobolev spaces and $j \leq k$,
\begin{equation}\label{2.13}
\sup_{t \in \mathbb{R}} |M(t)| \lesssim 2^{k} \| P_{j} u_{0} \|_{L^{2}}^{2} \| P_{k} v_{0} \|_{L^{2}}^{2}.
\end{equation}
Then by the product rule,
\begin{equation}\label{2.14}
\aligned
\frac{d}{dt} M(t) & = -2 \int \nabla \cdot Im[\bar{u} \nabla u] \frac{(x - y)_{1}}{|(x - y)_{1}|} Im[\bar{v} \partial_{1} v](t, x) dx dy \\
& -2 \int \nabla \cdot Im[\bar{v} \nabla v] \frac{(x - y)_{1}}{|(x - y)_{1}|} Im[\bar{u} \partial_{1} u](t, x) dx dy \\
& + \int |u(t, y)|^{2} \frac{(x - y)_{1}}{|(x - y)_{1}|} Re[\bar{v} \partial_{1} \Delta v](t, x) dx dy - \int |u(t, y)|^{2} \frac{(x - y)_{1}}{|(x - y)_{1}|} Re[\Delta \bar{v} \partial_{1}  v](t, x) dx dy \\
& + \int |v(t, y)|^{2} \frac{(x - y)_{1}}{|(x - y)_{1}|} Re[\bar{u} \partial_{1} \Delta u](t, x) dx dy - \int |v(t, y)|^{2} \frac{(x - y)_{1}}{|(x - y)_{1}|} Re[\Delta \bar{u} \partial_{1}  u](t, x) dx dy.
\endaligned
\end{equation}
Integrating by parts,
\begin{equation}\label{2.15}
\frac{d}{dt} M(t) = \int |\partial_{1}(u(t, x_{1}, x_{2}) \overline{v(t, x_{1}, y_{2})})|^{2} dx_{1} dx_{2} dy_{2},
\end{equation}
so by the fundamental theorem of calculus,
\begin{equation}\label{2.16}
\int \int |\partial_{1}(u(t, x_{1}, x_{2}) \overline{v(t, x_{1}, y_{2})})|^{2} dx_{1} dx_{2} dy_{2} dt \lesssim 2^{k} \| P_{j} u_{0} \|_{L^{2}}^{2} \| P_{k} v_{0} \|_{L^{2}}^{2}.
\end{equation}
If $\psi$ is the Littlewood--Paley kernel,
\begin{equation}\label{2.17}
\aligned
& \int \int |\partial_{1}(u(t, x_{1}, y_{2}) \overline{v(t, x_{1}, y_{2})})|^{2} dx_{1} dy_{2} dt \\
& \hspace{1cm} \lesssim 2^{j} \int \int |\psi(2^{j}(x_{2} - y_{2}))| |\partial_{1}(u(t, x_{1}, x_{2}) \overline{v(t, x_{1}, y_{2})})|^{2} dx_{1} dx_{2} dy_{2} dt \\
& \lesssim 2^{j} \int \int |\partial_{1}(u(t, x_{1}, x_{2}) \overline{v(t, x_{1}, y_{2})})|^{2} dx_{1} dx_{2} dy_{2} dt \lesssim 2^{j + k} \| P_{j} u_{0} \|_{L^{2}}^{2} \| P_{k} v_{0} \|_{L^{2}}^{2}.
\endaligned
\end{equation}
The same argument also implies that $(\ref{2.17})$ also holds for $\partial_{1}$ replaced by $\partial_{2}$, and therefore,
\begin{equation}\label{2.18}
\int \int |\nabla (u(t, x_{1}, y_{2}) \overline{v(t, x_{1}, y_{2})})|^{2} dx_{1} dy_{2} dt \lesssim 2^{j + k} \| P_{j} u_{0} \|_{L^{2}}^{2} \| P_{k} v_{0} \|_{L^{2}}^{2}.
\end{equation}
Then by Bernstein's inequality,
\begin{equation}\label{2.19}
\| v \bar{u} \|_{L_{t,x}^{2}(\mathbb{R} \times \mathbb{R}^{2})}^{2} = \| vu \|_{L_{t,x}^{2}(\mathbb{R} \times \mathbb{R}^{2})}^{2} \lesssim 2^{j - k} \| P_{j} u_{0} \|_{L^{2}}^{2} \| P_{k} v_{0} \|_{L^{2}}^{2}.
\end{equation}
\end{proof}

The main advantage in proving the bilinear estimates in this way is that the computations are quite robust under a forcing term. Indeed, suppose $u$ solves
\begin{equation}\label{2.20}
i u_{t} + \Delta u = F_{1}, \qquad u(0,x) = u_{0},
\end{equation}
and that $v$ solves
\begin{equation}\label{2.21}
i v_{t} + \Delta v = F_{2}, \qquad v(0,x) = v_{0}.
\end{equation}
Then following the computations in the proof of Theorem $\ref{t3.1}$,
\begin{equation}\label{2.22}
\aligned
\| (P_{j} u)(P_{k} v) \|_{L_{t,x}^{2}([0, T] \times \mathbb{R}^{2})}^{2} \lesssim 2^{j - k} \sup_{t \in [0, T]} \| P_{j} u(t) \|_{L^{2}}^{2} \| P_{k} v(t) \|_{L^{2}}^{2} \\
+ 2^{j - 2k + 1} \int_{0}^{T} \int Im[\overline{P_{j} u} P_{j} F_{1}](t, y) \frac{x - y}{|x - y|} \cdot Im[\overline{P_{k} v} \nabla P_{k} v](t, x) dx dy dt \\
+ 2^{j - 2k + 1} \int_{0}^{T} \int Im[\overline{P_{k} v} P_{k} F_{2}](t, y) \frac{x - y}{|x - y|} \cdot Im[\overline{P_{j} u} \nabla P_{j} u](t, x) dx dy dt \\
- 2^{j - 2k} \int_{0}^{T} \int |P_{j} u(t,y)|^{2} \frac{x - y}{|x - y|} \cdot Re[\overline{P_{k} v} \nabla P_{k} F_{2}](t,x) dx dy dt \\
+ 2^{j - 2k} \int_{0}^{T} \int |P_{j} u(t,y)|^{2} \frac{x - y}{|x - y|} \cdot Re[\overline{P_{k} F_{2}} \nabla P_{k} v](t,x) dx dy dt \\
- 2^{j - 2k} \int_{0}^{T} \int |P_{k} v(t,y)|^{2} \frac{x - y}{|x - y|} \cdot Re[\overline{P_{j} u} \nabla P_{j} F_{1}](t,x) dx dy dt \\
+ 2^{j - 2k} \int_{0}^{T} \int |P_{k} v(t,y)|^{2} \frac{x - y}{|x - y|} \cdot Re[\overline{P_{j} F_{1}} \nabla P_{j} u](t,x) dx dy dt.
\endaligned
\end{equation}
Taking $F_{1} = |u|^{2} u$, $F_{2} = |v|^{2} v$, if $(\ref{2.7})$ and $(\ref{2.8})$ hold for $u$ and $v$,
\begin{equation}\label{2.23}
 \aligned
 & \| (P_{j} u)(P_{j - 5 \leq \cdot \leq j + 5} u)(P_{\leq j - 5} u)^{2} \|_{L_{t, x}^{1}(\mathbb{R} \times \mathbb{R}^{2})} \\
 & \lesssim \sum_{j_{1} \leq j_{2} \leq j - 5} \| (P_{j_{1}} u)(P_{j - 5 \leq \cdot \leq j + 5} u) \|_{L_{t,x}^{2}(\mathbb{R} \times \mathbb{R}^{2})} \| P_{j_{2}} u \|_{L_{t, x}^{4}(\mathbb{R} \times \mathbb{R}^{2})} \| P_{j} u \|_{L_{t,x}^{4}(\mathbb{R} \times \mathbb{R}^{2})}  \\
 & \lesssim \epsilon_{0}^{4}. 
 \endaligned
 \end{equation}
 Next, if $(\ref{2.7})$ and $(\ref{2.8})$ hold for $u$ and $v$,
  \begin{equation}\label{2.24}
 \aligned
 & \| (P_{j} u)(P_{\leq j - 5} u)(P_{\geq j - 5} u)^{2} \|_{L_{t, x}^{1}(\mathbb{R} \times \mathbb{R}^{2})} \\
 & \lesssim \sum_{j_{1} \leq j - 5} \sum_{j - 5 \leq j_{2} \leq j_{3}} \| P_{j} u \|_{L_{t,x}^{4}(\mathbb{R} \times \mathbb{R}^{2})} \| (P_{j_{1}} u)(P_{j_{3}} u) \|_{L_{t,x}^{2}(\mathbb{R} \times \mathbb{R}^{2})} \| P_{j_{2}} u \|_{L_{t,x}^{4}}  \lesssim \epsilon_{0}^{4}.
 \endaligned
 \end{equation}
Finally, if $(\ref{2.7})$ and $(\ref{2.8})$ hold for $u$ and $v$,
 \begin{equation}\label{2.25}
 \aligned
& \| (P_{j} u) P_{j}((P_{\geq j - 5} u)^{3}) \|_{L_{t,x}^{1}(\mathbb{R} \times \mathbb{R}^{2})} \\
& \lesssim \sum_{j - 5 \leq j_{1} \leq j_{2} \leq j_{3}} \| (P_{j} u)(P_{j_{3}} u) \|_{L_{t,x}^{2}} \| P_{j_{2}} u \|_{L_{t,x}^{4}} \| P_{j_{1}} u \|_{L_{t,x}^{4}} \lesssim \epsilon_{0}^{4}.
 \endaligned
 \end{equation}
 Similar estimates hold when $u$ and $F_{1}$ are replaced by $v$ and $F_{2}$. Plugging $(\ref{2.23})$--$(\ref{2.25})$ into $(\ref{2.22})$, for $k - j > 10$,
 \begin{equation}\label{2.26}
 \| (P_{j} u)(P_{k} v) \|_{L_{t,x}^{2}}^{2} \lesssim 2^{j - k} \sup_{t} \| P_{j} u(t) \|_{L^{2}}^{2} \| P_{k} v(t) \|_{L^{2}}^{2} + 2^{j - k} (\epsilon_{0}^{4} \sup_{t} \| P_{j} u(t) \|_{L^{2}}^{2} + \epsilon_{0}^{4} \sup_{t} \| P_{k} v(t) \|_{L^{2}}^{2}).
 \end{equation}
 Now then,
 \begin{equation}\label{2.27}
 \frac{d}{dt} \| P_{j} u \|_{L^{2}}^{2} \lesssim \| (P_{j} u) P_{j} F_{1} \|_{L^{1}}, \qquad \frac{d}{dt} \| P_{k} v \|_{L^{2}}^{2} \lesssim \| (P_{k} v) P_{k} F_{2} \|_{L^{2}},
 \end{equation}
so plugging $(\ref{2.23})$--$(\ref{2.25})$ into $(\ref{2.27})$ implies that when $k - j > 10$,
\begin{equation}\label{2.28}
\| (P_{j} u)(P_{k} v) \|_{L_{t,x}^{2}(\mathbb{R} \times \mathbb{R}^{2})}^{2} \lesssim 2^{j - k}\| P_{j} u_{0} \|_{L^{2}}^{2} \| P_{k} v_{0} \|_{L^{2}}^{2} + 2^{j-k} \epsilon_{0}^{6}.
\end{equation}
Taking $u = v$ gives $(\ref{2.7})$ when $|j - k| > 10$ and taking $v = e^{it \Delta} u_{0}$ gives $(\ref{2.8})$ when $|j - k| > 10$.\medskip

When $|j - k| \leq 10$, use the Strichartz estimates in \cite{tao2006nonlinear}, \cite{keel1998endpoint}, and \cite{hadac2009well}. Indeed, using the $U^{p}$ and $V^{p}$ spaces, see for example \cite{hadac2009well},
 \begin{equation}\label{2.29}
 \| P_{j} u \|_{U_{\Delta}^{4}(\mathbb{R} \times \mathbb{R}^{2})} \lesssim \| P_{j} u_{0} \|_{L^{2}} + \sup_{\| v \|_{V_{\Delta}^{4/3}} = 1} \int_{\mathbb{R}} (P_{j}(|u|^{2} u), v)_{L^{2}} dt,
 \end{equation}
 where $(\cdot, \cdot)_{L^{2}}$ is the $L^{2}$ inner product,
 \begin{equation}\label{2.30}
 (f, g)_{L^{2}} = \int f(x) \overline{g(x)} dx.
 \end{equation}
Since $V_{\Delta}^{4/3} \subset U_{\Delta}^{2}$, applying $(\ref{2.23})$ with $P_{j} u$ replaced by $P_{j} v$,
\begin{equation}\label{2.31}
\int (P_{j}(|u|^{2} u), P_{j} v)_{L^{2}} dt \lesssim \epsilon_{0}^{3}.
\end{equation}
Plugging $(\ref{2.31})$ into $(\ref{2.29})$,
\begin{equation}\label{2.32}
\| P_{j} u \|_{L_{t,x}^{4}} \lesssim \| P_{j} u_{0} \|_{L^{2}} + \epsilon_{0}^{3}.
\end{equation}
Therefore, making a standard bootstrap argument, $(\ref{2.7})$ and $(\ref{2.8})$ hold.

 \subsection{Scattering}
 \label{sec:wuscat}
 
 To make the proof of Theorem $\ref{t2.1}$ completely rigorous, fix $N \in \mathbb{Z}$ and set
 \begin{equation}\label{2.33}
 u_{0, N} = \sum_{j \leq N} P_{j} u_{0}.
 \end{equation}
 Then since $u_{0, N} \in \dot{H}^{1}(\mathbb{R}^{2})$, we know from the arguments in \cite{cazenave1988cauchy} that $(\ref{2.1})$ is locally well-posed. One could furthermore choose an interval $[-T_{N}, T_{N}]$, where $T_{N}$ will likely go to zero as $N \nearrow \infty$, where $\| u \|_{L_{t,x}^{4}([-T_{N}, T_{N}] \times \mathbb{R}^{2})} \leq \epsilon_{0}$. Plugging this fact into $(\ref{2.23})$, $(\ref{2.24})$, and $(\ref{2.26})$ would therefore imply $(\ref{2.22})$ on the interval $[-T_{N}, T_{N}]$. Furthermore, we would continue to have $\| u(T_{N}) \|_{B_{2, \infty}^{0}}$ and $\| u(-T_{N}) \|_{B_{2, \infty}^{0}} \lesssim \epsilon_{0}$ and $\| u(\pm T_{N}) \|_{\dot{H}^{1}} \lesssim \| u_{0, N} \|_{\dot{H}^{1}}$, so the same computations could be extended to the intervals $[T_{N}, 2T_{N}]$ and $[-2 T_{N}, -T_{N}]$. Furthermore, $(\ref{2.23})$, $(\ref{2.24})$, $(\ref{2.26})$ would imply $(\ref{2.28})$ on the interval $[-2 T_{N}, 2 T_{N}]$. Iterating, the estimate $(\ref{2.28})$ would be obtained on $\mathbb{R}$. 
 
 Furthermore, the bounds in $(\ref{2.28})$ would be independent of $N$. The computations in $(\ref{2.23})$, $(\ref{2.24})$, and $(\ref{2.26})$, we can show that $P_{j} u_{N}(t)$ converges in $L_{t,x}^{4}(\mathbb{R} \times \mathbb{R}^{2})$, where $u_{N}(t)$ is the solution to $(\ref{2.1})$ with initial data $u_{0, N}$.
 
 The bounds in $(\ref{2.22})$ imply
 \begin{equation}\label{2.35}
 e^{-it \Delta} P_{k} u(t, \cdot)
 \end{equation}
 converge in $L^{2}(\mathbb{R}^{2})$ as $t \rightarrow \pm \infty$, which implies $(\ref{2.6})$ and $(\ref{2.7})$ using standard methods (see for instance Corollary 2.3.5 in \cite{dodson2019defocusing}).
 \end{proof}
 
 \begin{remark}
 The weak convergence in $(\ref{2.5})$ and $(\ref{2.6})$ could not be upgraded to strong convergence in $B_{2, \infty}^{0}$. To see why, suppose $\psi \in L^{2}$ is supported in Fourier space on $1 \leq |\xi| \leq 2$. Then the initial data
 \begin{equation}
 u_{0}(x) = \epsilon_{0} \sum_{j \in \mathbb{Z}} 2^{j} \psi(2^{j} x),
 \end{equation}
 satisfies the bounds $\| u_{0} \|_{B_{2, \infty}^{0}} \lesssim \epsilon_{0}$, as does the initial data,
 \begin{equation}\label{2.36}
 u_{0}(x) = \epsilon_{0} \sum_{j \in \mathbb{Z}} 2^{j} e^{i t_{j} \Delta} \psi(2^{j} x),
 \end{equation}
 where $t_{j}$ has subsequences converging to $+\infty$ and $-\infty$. Then $(\ref{2.5})$ and $(\ref{2.6})$ will not hold if weak convergence is replaced by convergence in $B_{2, \infty}^{0}$.
 
 If $B_{2, \infty}^{0}$ were replaced by $B_{2, q}^{0}$ for some $1 \leq q < \infty$, then $(\ref{2.5})$ and $(\ref{2.6})$ would hold in norm.
 \end{remark}

\section{Harmonic Map Heat Flow}
\label{sec:HMHF}

 In this section, we will focus on the case $d = 2$ throughout.  We will use the derivative formulation of the Schr{\"o}dinger maps problem. Since $d = 2$, if $\phi : \mathbb{R}^{2} \rightarrow \mathcal N \hookrightarrow \mathbb{R}^{N}$, then
\begin{equation}\label{3.1}
\nabla \phi : \mathbb{R}^{2} \rightarrow \mathbb{R}^{2N}.
\end{equation}
We will impose the Besov space condition on the initial data $\phi_{0}$ that
\begin{equation}\label{3.2}
\sup_{j \in \mathbb{Z}} \| P_{j} (\nabla \phi_{0}) \|_{L^{2}} \leq \epsilon.
\end{equation}
The extrinsic form of the harmonic map heat flow is given by
\begin{equation}\label{3.3}
\frac{\partial u}{\partial t} = \Delta u - \Gamma(u)(\nabla u, \nabla u), \qquad u(0) = \phi_{0},
\end{equation}
where $\Gamma(u) (\nabla u , \nabla u)$ is the second fundamental form, $\Gamma : T \mathcal N \times T \mathcal N \to N \mathcal N$.  
\begin{lemma}\label{l3.1}
The harmonic map heat flow $(\ref{3.3})$ is well-defined for initial data satisfying $(\ref{3.2})$.
\end{lemma}
\begin{proof}
The argument proving Lemma $\ref{l3.1}$ is almost identical to the argument proving Theorem $\ref{t2.1}$. We seek a solution that satisfies Duhamel's principle,
\begin{equation}\label{3.4}
u(t) = e^{t \Delta} u_{0} + \int_{0}^{t} e^{(t - \tau) \Delta} \Gamma(u)(\nabla u, \nabla u) d\tau.
\end{equation}
By direct computation,
\begin{equation}\label{3.5}
\| e^{t \Delta} P_{k} u_{0} \|_{L_{t}^{1} L_{x}^{2}([0, \infty) \times \mathbb{R}^{2})} \lesssim 2^{-2k} \| P_{k} u_{0} \|_{L^{2}},
\end{equation}
and
\begin{equation}\label{3.6}
\| e^{t \Delta} P_{k} u_{0} \|_{L_{t}^{\infty} L_{x}^{2}([0, \infty) \times \mathbb{R}^{2})} \lesssim \| P_{k} u_{0} \|_{L^{2}},
\end{equation}
so by the Sobolev embedding theorem,
\begin{equation}\label{3.7}
\| e^{t \Delta} P_{k} u_{0} \|_{L_{t}^{1} L_{x}^{\infty}([0, \infty) \times \mathbb{R}^{2})} \lesssim 2^{-k} \| P_{k} u_{0} \|_{L^{2}},
\end{equation}
and
\begin{equation}\label{3.8}
\| e^{t \Delta} P_{k} u_{0} \|_{L_{t,x}^{\infty}([0, \infty) \times \mathbb{R}^{2})} \lesssim 2^{k} \| P_{k} u_{0} \|_{L^{2}}.
\end{equation}

Now then,
\begin{align}\label{3.9}
& P_{k}(\Gamma(u)(\nabla u, \nabla u)) \\
& = P_{k}(\Gamma(u)(\nabla u_{\geq k}, \nabla u_{\geq k}) + 2 \Gamma(u)(\nabla u_{\leq k}, \nabla u_{\geq k})) + P_{k}(\Gamma(u)(\nabla u_{\leq k}, \nabla u_{\leq k})). \notag
\end{align}
Applying the Sobolev embedding theorem and that $\|P_k u \|_{L^2_x} \leq 2^k \| P_k u \|_{L^1}$, we get
\begin{align}\label{3.10}
& \| \nabla P_{k}(\Gamma(u)(\nabla u_{\geq k}, \nabla u_{\geq k})) \|_{L_{t}^{1} L_{x}^{2}}   \lesssim  2^{2k}  \| \nabla u_{\geq k} \|_{L_{t,x}^{2}}^{2} \\
\label{3.11}
& \hspace{1cm} \lesssim   2^{2k} (\sum_{j \geq k} 2^{-j} \| \nabla^{2} P_{j} u \|_{L_{t,x}^{2}})^{2}.
\end{align}
Next,
\begin{equation}\label{3.12}
\aligned
\| \nabla P_{k} (\Gamma(u)(\nabla u_{\geq k}, \nabla u_{\leq k})) \|_{L_{t}^{1} L_{x}^{2}} \lesssim 2^{k} \| \nabla  u_{\geq k} \|_{L_{t}^{1} L_{x}^{2}} \| \nabla  u_{\leq k} \|_{L_{t,x}^{\infty}} \\ \lesssim 2^{k} (\sum_{j \geq k} 2^{-2j} \| \nabla^{3} P_{j} u \|_{L_{t}^{1} L_{x}^{2}})(\sum_{j \leq k} 2^{j} \| P_{j} \nabla u \|_{L_{t}^{\infty} L_{x}^{2}}).
\endaligned
\end{equation}
Now, using that $\nabla \Gamma (u) = \Gamma'(u) \nabla u$, we have
\begin{align}\label{3.13}
& \| \nabla P_{k}(\Gamma(u)(\nabla u_{\leq k}, \nabla u_{\leq k})) \|_{L_{t}^{1} L_{x}^{2}}  \\
& \hspace{1cm} \lesssim \| P_{k}(\Gamma(u)(\nabla^{2} u_{\leq k}, \nabla u_{\leq k})) \|_{L_{t}^{1} L_{x}^{2}} + \| P_{k}(\Gamma'(u)(\nabla u_{\leq k}, \nabla u_{\leq k}) \nabla u) \|_{L_{t}^{1} L_{x}^{2}}. \notag
\end{align}
By Bernstein's inequality,
\begin{equation}\label{3.14}
\aligned
& \| P_{k}(\Gamma(u)(\nabla^{2} u_{\leq k}, \nabla u_{\leq k})) \|_{L_{t}^{1} L_{x}^{2}}  \lesssim 2^{-k} \| \nabla P_{k}(\Gamma(u)(\nabla^{2} u_{\leq k}, \nabla u_{\leq k})) \|_{L_{t}^{1} L_{x}^{2}} \\
& \hspace{.5cm} \lesssim 2^{-k} \| P_{k}(\Gamma(u)(\nabla^{3} u_{\leq k}, \nabla u_{\leq k})) \|_{L_{t}^{1} L_{x}^{2}} + 2^{-k} \| P_{k}(\Gamma(u)(\nabla^{2} u_{\leq k}, \nabla^{2} u_{\leq k})) \|_{L_{t}^{1} L_{x}^{2}} \\ 
&  \hspace{1cm} + 2^{-k} \| P_{k}(\Gamma'(u)(\nabla^{2} u_{\leq k}, \nabla u_{\leq k}) \nabla u) \|_{L_{t}^{1} L_{x}^{2}} \\
&  \hspace{.5cm}  \lesssim 2^{-k} \| \nabla^{3} u_{\leq k} \|_{L_{t}^{1} L_{x}^{4}} \| \nabla u_{\leq k} \|_{L_{t}^{\infty} L_{x}^{4}}  + 2^{-k} \| \nabla^{2} u_{\leq k} \|_{L_{t}^{2} L_{x}^{4}}^{2} \\
&  \hspace{1cm} + 2^{-k} \| \nabla u_{\geq k} \|_{L_{t,x}^{2}} \| \nabla^{2} u_{\leq k} \|_{L_{t}^{2} L_{x}^{\infty}} \| \nabla u_{\leq k} \|_{L_{t,x}^{\infty}}  + 2^{-k} \| \nabla u_{\leq k} \|_{L_{t}^{4} L_{x}^{8}}^{2} \| \nabla^{2} u_{\leq k} \|_{L_{t}^{2} L_{x}^{4}} \\
&  \hspace{.5cm} \lesssim 2^{-k} (\sum_{j \leq k} 2^{j/2} \| \nabla^{3} P_{j} u \|_{L_{t}^{1} L_{x}^{2}})(\sum_{j \leq k} 2^{j/2} \| \nabla P_{j} u \|_{L_{t}^{\infty} L_{x}^{2}}) + 2^{-k} (\sum_{j \leq k} 2^{j/2} \| \nabla^{2} P_{j} u \|_{L_{t,x}^{2}})^{2} \\
&  \hspace{1cm}  + 2^{-k} (\sum_{j \geq k} 2^{-j} \| \nabla^{2} P_{j} u \|_{L_{t,x}^{2}})(\sum_{j \leq k} 2^{j} \| \nabla P_{j} u \|_{L_{t}^{2} L_{x}^{\infty}})(\sum_{j \leq k} 2^{j} \| \nabla P_{j} u \|_{L_{t,x}^{\infty}}) \\ 
&  \hspace{1cm} + 2^{-k}(\sum_{j \leq k} 2^{j/2} \| \nabla^{2} P_{j} u \|_{L_{t,x}^{2}})(\sum_{j \leq k} 2^{j/4} \| \nabla P_{j} u \|_{L_{t,x}^{4}})^{2}.
\endaligned
\end{equation}
Also,
\begin{equation}\label{3.15}
\| P_{k}(\Gamma'(u)(\nabla u_{\leq k}, \nabla u_{\leq k}) \nabla u_{\geq k}) \|_{L_{t}^{1} L_{x}^{2}} \lesssim (\sum_{j \geq k} 2^{-j} \| \nabla^{2} P_{j} u \|_{L_{t,x}^{2}})(\sum_{j \leq k} 2^{j/2} \| \nabla P_{j} u \|_{L_{t,x}^{4}})^{2}.
\end{equation}
Therefore, it remains to bound
\begin{equation}\label{3.16}
\| P_{k}(\Gamma'(u)(\nabla u_{\leq k}, \nabla u_{\leq k}) \nabla u_{\leq k}) \|_{L_{t}^{1} L_{x}^{2}}.
\end{equation}
By Bernstein's inequality,
\begin{equation}\label{3.17}
\aligned
& (\ref{3.16}) \lesssim 2^{-k} \| \nabla P_{k}(\Gamma'(u)(\nabla u_{\leq k}, \nabla u_{\leq k}) \nabla u_{\leq k}) \|_{L_{t}^{1} L_{x}^{2}}  \\
& \lesssim 2^{-k} \| P_{k}(\Gamma(u)(\nabla^{2} u_{\leq k}, \nabla u_{\leq k}) \nabla u_{\leq k}) \|_{L_{t}^{1} L_{x}^{2}} + 2^{-k} \| P_{k}(\Gamma'(u)(\nabla u_{\leq k}, \nabla u_{\leq k}) \nabla^{2} u_{\leq k}) \|_{L_{t}^{1} L_{x}^{2}} \\
& \hspace{1cm}  + 2^{-k} \|  P_{k}(\Gamma''(u)(\nabla u_{\leq k}, \nabla u_{\leq k}) \nabla u_{\leq k} \nabla u) \|_{L_{t}^{1} L_{x}^{2}} \\
&  \lesssim  2^{-k}(\sum_{j \leq k} 2^{j/2} \| \nabla^{2} P_{j} u \|_{L_{t,x}^{2}})(\sum_{j \leq k} 2^{j/4} \| \nabla P_{j} u \|_{L_{t,x}^{4}})^{2} \\
& \hspace{1cm} + 2^{-k}(\sum_{j \leq k} 2^{j/4} \| \nabla P_{j} u \|_{L_{t,x}^{4}})^{4}+ 2^{-k}(\sum_{j \geq k} 2^{-j} \| \nabla^{2} P_{j} u \|_{L_{t,x}^{2}})(\sum_{j \leq k} 2^{j/3} \| \nabla P_{j} u \|_{L_{t}^{6} L_{x}^{3}})^{3}.
\endaligned
\end{equation}
Therefore, define the norm
\begin{equation}\label{3.18}
a_{j} = \sup_{j}(\| \nabla P_{j} u \|_{L_{t}^{\infty} L_{x}^{2}} + \| \nabla^{3} P_{j} u \|_{L_{t}^{1} L_{x}^{2}}).
\end{equation}
Then by $(\ref{3.10})$--$(\ref{3.17})$, $(\ref{3.2})$, and $(\ref{3.5})$--$(\ref{3.8})$,
\begin{equation}\label{3.19}
\sup_{j \in \mathbb{Z}} a_{j} \lesssim \epsilon + (\sup_{j \in \mathbb{Z}} a_{j}^{2}).
\end{equation}
Therefore,
\begin{equation}\label{3.20}
\sup_{j \in \mathbb{Z}} a_{j} \lesssim \epsilon.
\end{equation}

To make these computations rigorous, we know that the harmonic map heat flow has a local solution. Furthermore, this solution obeys $(\ref{3.20})$ on the interval of existence $[0, T]$. Since $(\ref{3.20})$ holds on this interval, $[0, T]$, we can use local well-posedness and show that $(\ref{3.20})$ holds on a slightly larger interval. Since the maximal interval of existence is open and closed in $[0, \infty)$, it must be all of $[0, \infty)$.
\end{proof}

\section{The Gauged Schr\"odinger Map Equation}
\label{sec:Gauges}

We give a brief overview here of the possible gauge choices for Schr\"odinger Maps from $\mathbb{R}^d$ to a $2n$ dimensional K{\"a}hler manifold target, denoted $(\mathcal N, J, h)$, for $J$ the symplectic form on $\mathcal N$ and $h$ the metric.  Begin with a smooth function $\phi: \mathbb{R}^d \times (-T,T) \to \mathcal N$.  At each point $\phi (x,t)$, there are tangent vectors to $\mathcal N$ at $\phi(x,t)$ that can be taken of the form
\begin{equation}
\label{4.1}
E := \{ e_1 (t,x), J e_1 (t,x), \dots, e_n(t,x), J e_n (t,x) \},
\end{equation}
which in the case $\mathcal N = \mathbb{S}^2$ is of the form given by an orthonormal frame $(v(t,x), w(t,x))$.  It is useful to consider the equations for the derivatives of $\phi$, so consider functions
\begin{equation}
\label{4.2}
{\tilde \psi}_m^\alpha = \langle \partial_m u, e_\alpha \rangle, \ \ {\tilde \psi}_m^{\bar \alpha} = \langle \partial_m u, J e_\alpha \rangle
\end{equation}
for $0 \leq m \leq d$ ($\partial_0 = \partial_t$ throughout) and $0 \leq \alpha \leq n$.  Set 
\begin{equation}
\label{4.4}
\psi_m^\alpha = {\tilde \psi}_m^\alpha + i {\tilde \psi}_m^{\bar \alpha}.
\end{equation}

Again, for $\mathcal N = \mathbb{S}^2$, this is typically considered using the complex coordinate
\begin{equation}
\label{4.5}
\psi_m = \langle v(t,x), \partial_m \phi \rangle + i \langle w(t,x), \partial_m \phi \rangle.
\end{equation}

Define the $d+1$ connection coefficient matrices (of size  $n\times n$) as 
\begin{equation}
\label{4.6}
[A_m]_{p,q} = \langle \partial_m e_p, e_q \rangle,
\end{equation}
which again for our example case is simply 
\begin{equation}
\label{4.7}
A_m = \langle \partial_m v, w \rangle.
\end{equation}
Letting $D_m = \partial_m + A_m$, we have the ability to reframe the Schr\"odinger map flow as
\begin{equation}
\label{4.8}
\psi_t = i \sum_{m=1}^d D_m \psi_m.
\end{equation}

Combining a large number of components of the structure of the equation and the frame, one arrives at the equations
\begin{equation}
\label{4.9}
(i \partial_t + \Delta) \psi_m = - 2 i \sum_{\ell=1}^d A_\ell \partial_\ell  \psi_m + ( A_0 + \sum_{\ell = 1}^d (A_\ell^2 - i \partial_\ell A_\ell) ) \psi_m - i \sum_{\ell = 1}^d \psi_\ell {\rm Im} (\overline{\psi_\ell} \psi_m).
\end{equation}

However, throughout here, $A$ is chosen up to setting the orthonormal frame, and selecting a particularly useful frame through the use of a gauge will be the discussion of the remainder of this section.  

\subsection{The Coulomb Gauge}

The Coulomb gauge tends to be reasonably simple and useful for $d \geq 3$.  In this case, we assume that
\begin{equation}
\label{4.10}
\sum_{m = 1}^d \partial_m A_m = 0,
\end{equation}
giving
\begin{equation}
\label{4.11}
A_m = \Delta^{-1} \sum_{\ell = 1}^d \partial_\ell  {\rm Im} (\overline{\psi_\ell} \psi_m)
\end{equation}
for $m=1,\dots, d$ and
 \begin{equation}
\label{4.11alt}
A_0 = \Delta^{-1} \sum_{\ell , m = 1}^d \Delta^{-1} \partial_\ell \partial_m   {\rm Re} (\overline{\psi_\ell} \psi_m) - \frac12 \sum_{m=1}^d |\psi_m|^2.
\end{equation}

This gauge works well in high dimensions, in particular in Sobolev based spaces it is effective for $d \geq 3$. We recall it here only to give insight into choices of gauge that are possible and to mention that for $d\geq 3$, many of the results below can be applied more easily using this gauge, see \cite{dodson2012bilinear}.

\subsection{The Caloric Gauge}

Since our main interest lies in the case $d=2$, we will need to use the caloric gauge that was used in \cite{bejenaru2011global} for the small data map to a sphere in two dimensions.  This gauge was also used in \cite{li2021global} to study the small data Schr\"odinger map equation into a compact K{\"a}hler manifold.  The caloric gauge was first used in \cite{tao2008global} in the case of wave maps.  To proceed, note that we will use the calculations in \cite{smith2013schrodinger} relating to caloric gauges very heavily. 

Let us recall the key components of the Caloric gauge and the structure of the Schr\"odinger map equation in such a setting.  Consider the harmonic map heat flow $(\ref{3.3})$, which by Lemma \ref{l3.1} is well-posed in Besov spaces.  Given the dissipative nature of the equation, it will be shown that as $s \to \infty$, $u(s,x)$ approaches equilibrium state $Q$.  Hence, at $\infty$, we can select a frame, say 
\begin{equation}
\label{4.12}
E_\infty := \{ e_{1,\infty} , J e_{1,\infty}, \dots, e_{n,\infty} , J e_{n,\infty}  \},
\end{equation}
an arbitrary orthonormal basis to $T_Q K$ that is independent to $x$ and $t$.  To define the frame for all $s \geq 0$, we pull back the frame $E_\infty$ using the backward heat flow.  A key relation we must establish is that that
\[
\langle e_1 (s,x) , \partial_s (Je_1) (x,s) \rangle = 0.
\]
In the setting of the Kahler manifolds, Definition $2.1$, \cite{li2021global} gives that for $\phi (x,t) : [-T,T] \times \mathbb{R}^d \to K$ and a given orthonormal frame $E_\infty$, then the {\it caloric gauge} is a map $z: \mathbb{R}^+ \times [-T,T] \times \mathbb{R}^d \to K$ solving
\begin{equation}\label{4.14}
\frac{\partial z}{\partial s} = \Delta z - \Gamma(z)(\nabla z, \nabla z), \qquad z(0,x;t) = \phi (t,x).
\end{equation}
and corresponding frames
 \begin{equation}
\label{4.15}
 E(v(s,x;t))  := \{ e_{1} (s,x;t)  , J e_{1} (s,x;t) , \dots, e_{n} (s,x;t)  , J e_{n} (s,x;t)   \},
\end{equation}
such that
\begin{equation}
\label{4.16}
\partial_s e_k (s,x;t) = 0, \ \ \lim_{s \to \infty} e_k (s,x;t) = e_{k,\infty}
\end{equation}
converging in the correct sense.  The existence of such frames has been established in the setting of Sobolev spaces in \cite{smith2013schrodinger} and in \cite{bejenaru2011global} using atomic-like spaces for $\mathcal{N} = \mathbb{S}^2$. Then, in this gauge, we can see that for $m = 0,\dots,d$, we have
\begin{equation}
\label{4.17}
[A_m]_{p,q} (s,x;t)  = -\int_s^\infty    \langle R( \partial_s v, \partial_m v) e_p, e_q \rangle d s'
\end{equation}
for $R$ the Riemannian curvature tensor on $\mathcal{N}$.  
 
The main issue is going from the Besov space smallness norm on the Schr\"odinger map to the Besov space smallness in gauge.  In particular, we will prove the following.

\begin{proposition}
\label{p5.1}
Let $\phi$ satisfy the smallness condition \eqref{3.2}.  There exists a unique, smooth $z$ satisfying \eqref{4.14} and smooth field functions as in \eqref{4.15}-\eqref{4.16} such that for $F \in \{ z, e_1, \dots, e_n \}$ we have
\begin{equation}
\label{4.18}
\| P_k F(s) \|_{L^\infty_t L^2_x} \leq \gamma_k (\sigma) (1 + s 2^{2k})^{-20} 2^{- \sigma k}
\end{equation}
and 
\begin{equation}
\label{4.19}
\sup_{k \in \mathbb{Z}} \sup_{s \in [0,\infty)} (s+1)^{\sigma/2} 2^{\sigma k} \| P_k \partial_t^\rho F(s) \|_{L^\infty_t L^2_x} < \infty.
\end{equation}
\end{proposition}

However, we need bounds on the associated functions $\psi_m$ and $A_m$ in the Besov norm. 

\subsection{Smallness in Besov norm leads to smallness in the gauge}

We begin with the assumption \eqref{3.2} for $\epsilon$ sufficiently small. Using the caloric gauge as defined above for $d \geq 2$, we will thus prove that $\sup_{j \in \mathbb{Z}} \| P_j \psi_m \|_{L^2} \leq C \epsilon$ and $\sup_{j \in \mathbb{Z}} \| P_j A_m \|_{L^2} \leq C \epsilon$.  A similar result holds for the Coulomb gauge for $d \geq 3$.  To see this, we begin by recalling some key results from \cite{bejenaru2011global} relating to frequency localized versions of \eqref{4.14}.  We will focus on the case $\mathbb{S}^2$ for simplicity of exposition to begin. Thus, we have
\begin{equation}
\label{eqn:Amdef}
A_m = - \sum_{\ell = 1}^d \int_0^\infty \Im ( \overline{\psi_m} ( \partial_\ell \psi_\ell + i A_\ell  \psi_\ell )) dr
\end{equation}
for the caloric gauge.

For both the Coulomb gauge and the caloric gauge, the covariant derivative coefficients are quadratic in their dependence upon $\psi$.  In both models, given a smooth frame $u,w$, we have
\begin{equation}
\label{4.20}
\psi_m = \langle v, \partial_m \phi \rangle + i \langle w, \partial_m \phi \rangle.
\end{equation}
Hence, 
\begin{equation}
\label{4.21}
\sup_{j \in \mathbb{Z}} \| P_j \psi_m \|_{H^{\frac{d-2}{2} } } \leq C \epsilon
\end{equation}
holds provided that such a smooth frame exists.  In the case of the Coulomb gauge, the existence of such a frame is now relatively standard and can be seen for instance in Section $2$ of \cite{bejenaru2007global}, with the properties that $v,\partial_t v, \partial_m v, w, \partial_t w, \partial_m w \in C([0,T] \times H^\infty)$.  
Thus, utilizing \eqref{4.11}, the model calculation of interest for the Coulomb gauge is for instance to prove
\begin{align}
\label{4.22}
& \| P_j A_1 \|_{H^{\frac{d-2}{2} } }  \lesssim \\
 & \hspace{.2cm} 2^{-j}  \| \sum_{\ell < j} (P_j \overline{\psi}_2) (P_\ell \psi_1) +  \sum_{\ell < j} (P_j \psi_1) (P_\ell \overline{\psi}_2) +  \sum_{ k \geq j, |k-\ell|< 10} P_j [(P_\ell \psi_1) (P_k\overline{\psi}_2) ]\|_{L^{\frac{d-2}{2}} } \notag \\
 & \hspace{.5cm}  \lesssim  \sum_{k \geq j}  2^{j \left(\frac{d-2}{2} \right) }  2^{-j} 2^{j  \left(\frac{d}{2} \right) }  2^{-k (d-2) } \| P_k \phi \|_{L^2}^2 =  \sum_{k \geq j}  2^{(j-k)(d-2)} \| P_k \phi \|_{H^\frac{d}{2}}^2 \notag
\end{align}
using Sobolev embeddings and Bernstein from $L^2$ to $L^1$.  This is why such a method is restricted to higher dimensions, even in case small Besov norms.  See for instance the treatment in the pre-print of the first author \cite{dodson2017global}.

In the case of the Caloric Gauge, we need to establish the existence of such a frame and make rigorous the observation from \cite{bejenaru2011global} that we have made the high-high interactions stronger in that
\begin{align}
\label{4.23}
& \| P_j A_1 \|_{H^{\frac{d-2}{2} } }  \lesssim \\
 & \hspace{.2cm} 2^{-j}  \| \sum_{\ell < j} (P_j \overline{\psi}_2) (P_\ell \psi_1) +  \sum_{\ell < j} (P_j \psi_1) (P_\ell \overline{\psi}_2) +  2^{j-k} \sum_{ k \geq j, |k-\ell|< 10} P_j [(P_\ell \psi_1) (P_k\overline{\psi}_2) ]\|_{H^{\frac{d-2}{2} } } \notag \\
 & \hspace{.5cm}  \lesssim \sum_{k \geq j} 2^{(j-k)} \| P_k \psi \|_{H^\frac{d}{2}}^2  . \notag
\end{align}
First of all, as we have established the existence of a solution to the Harmonic Map Heat Flow under our small Besov norm condition in Lemma \ref{l3.1}, let us notationally take
\begin{equation}
\label{4.24}
z = ({\rm HMHF})(s) (\phi(x,t)), \ \ \lim_{s \to \infty} z(\cdot,s) = Q
\end{equation}
for $\phi \in H^{\infty,\infty}$.  We are thus able to apply the proof of Lemma $8.3$ from \cite{bejenaru2011global} to establish $L^\infty_t L^2_x$ bounds on $P_j z$.  Similarly, we can construct a function $v: \mathbb{R}^d \to \mathbb{S}^2$ such that $v \cdot z = 0$ and
\begin{equation}
\label{4.25}
\partial_s v = [ (\partial_s z) z^T - z (\partial_s z)^T] v,
\end{equation}
as well as $w = v \times z$.  The key difference between the estimates that we require and the statements in Section $8$ of \cite{bejenaru2011global} is the use of frequency envelopes instead of Besov norm bounds, but the proofs are identical otherwise.

\subsection{Bounds on terms in the Caloric Gauge}

To proceed, we have to adapt some of the bounds for the Harmonic Map Heat Flow from \cite{dodsonsmith} to the setting of small Besov norm.   In order to prove our main theorem, we need some preliminary results using \cite{dodson2012bilinear,dodsonsmith,dodson2017global} adapted to our Besov norm.  To that end, define the appropriate frequency envelope $\alpha_k$ in the Besov norm for the solution $\Phi$.  First, for $d=2$, we let $\beta_{k}$ majorize 
\begin{equation}
 \| P_k \psi (0) \|_{L^2_{x}}
\end{equation}
and
\begin{equation}
 \| P_k \psi \|_{L^2_{x}} \leq \beta_{k}, \ \ \beta_{k} \leq 2^{\delta |k-\ell|} \beta (\ell), \ \ \sup_k \beta_{k} \leq  \epsilon (\| \phi \|_{\dot{H}^1}).
 \end{equation}
 In particular, for $d=2$, we let $\alpha_k$ majorize
\begin{equation}
 \| P_k \psi \|_{L^2_{x,t}}
\end{equation}
and
\begin{equation}
 \| P_k \psi \|_{L^2_{x,t}} \leq \alpha_k, \ \ \alpha_k \leq 2^{\delta |k-\ell|} \alpha (\ell), \ \ \sup_k \alpha_k \leq  \epsilon (\| \phi \|_{\dot{H}^1}).
 \end{equation}

To bound the Caloric gauge covariant terms, we recall that
\begin{equation}
P_k A_m (s) = - P_k \int_s^\infty \Im (\overline{\psi}_m ( \partial_\ell \psi_\ell + i A_\ell \psi_\ell) ) ds'.
\end{equation}

We start with the gauged harmonic map heat flow bounds proven above, that gives bounds on $\psi (s),A(s)$.  Using the corresponding decay bounds on $P_k \phi(s), P_k A(s)$, we are then ready to consider the gauged Schr\"odinger map flow.  To begin, the most difficult term to handle in the bootstrap is the quadratic term, $A \nabla \psi$.  

In order to prove this estimate, we consider the interactions $P_\ell A$, $P_k \psi$ and when $|\ell-k| > 10$, we will apply a bilinear estimate directly.  In particular, we have from \cite{dodson2012bilinear}, Theorem $1.3$ it was proven in the setting of smallness in the energy space that we have that for $|\ell -k| > 10$,
\begin{equation}
\| (P_k \overline{\psi  (s)} )( P_\ell \psi (s')) \|_{L^2 (I \times \mathbb{R}^2)} \lesssim 2^{-|\ell-k|/2} v_k v_\ell (1 + s 2^{2k} )^{-4} (1+ s' 2^{2 \ell})^{-4}.
\end{equation} 
We will prove a slightly modified result for smallness in the critical Besov space, namely that
\begin{equation}
\| (P_k \overline{\psi (s)} )( P_\ell \psi (s')) \|_{L^2 (I \times \mathbb{R}^2)} \lesssim 2^{-|\ell-k|/2} v_k v_\ell (1 + s 2^{2k} )^{-3} (1+ s' 2^{2 \ell})^{-3}.
\end{equation} 
The slightly weaker decay in $s$ arises due to the fact that since we are working in Besov spaces, we must sacrifice some decay in $s$ to make summation possible, as will be seen in the proofs below.  
Once we have such a result, we then have that
\begin{equation}
|P_k \sum_{| \ell - k | > 10} P_\ell A 2^k P_k \psi | \leq c(E_0) \varepsilon v_k^2 .
\end{equation}

First, let us note that as in Corollary 3.2 of \cite{dodson2012bilinear}, if $u$ solves the {\it free} Schr\"odinger equation 
\[ 
i \partial_t u + \Delta u = 0, \ \ u(x,0) = u_0, \ \  u_0 \in L^2 (\mathbb{R}^2),
\]
then for $M \ll N$,
\[
\| ( P_M u(t,x)) (P_N \bar{u} (t, x+ x_0)) \|_{L^2_{t,x} } \lesssim \frac{ M^{\frac12} }{N^{\frac12}} \| P_M u_0 \|_{L^2} \| P_N u_0 \|_{L^2}.
\]

To establish the bilinear estimate in the Caloric Gauge Schr\"odinger map equation, we follow closely the analysis in \cite{dodson2012bilinear} but make modifications when necessary to account for the smallness Besov norm modifications we require.  The primary methods of the proof involve using the Duhamel representation of the harmonic map heat flow to bound the elements in the gauged equations appropriately.    Hence, as in \cite{dodson2012bilinear}, we also have to recall some results from \cite{smith2012geometric}, Thm. 7.4 and Cor. 7.5.  We have that
\begin{align}
\label{eqn:Aheat1} 
\sup_{s>0} s^{(k+1)/2} \| \partial_x^k A_x (s) \|_{L^\infty_x} \lesssim_{E_{0,k}} \varepsilon, \\
\label{eqn:Aheat2} 
\sup_{s>0} s^{k/2} \| \partial_x^k A_x (s) \|_{L^2_x} \lesssim_{E_{0,k}} 1
\end{align}
for all $k \geq 0$ and $s > 0$.  In addition, we have
\begin{align}
\label{eqn:Aheat3} 
\int_0^\infty s^{(k-1)/2} \| \partial_x^k A_x (s) \|_{L^\infty_x} \lesssim_{E_{0,k}} 1, \\ 
\label{eqn:Aheat4} 
\int_0^\infty s^{(k-1)/2} \| \partial_x^{k+1} A_x (s) \|_{L^2_x} \lesssim_{E_{0,k}} 1.
\end{align}
We also have for $k \geq 1$, 
\begin{align}
\label{eqn:psiheat1} 
\sup_{s>0} s^{k/2} \| \partial_x^{k-1} \psi_x (s) \|_{L^\infty_x} \lesssim_{E_{0,k}} \varepsilon, \\
\label{eqn:psiheat2} 
\sup_{s>0} s^{(k-1)/2} \| \partial_x^{k-1} \psi_x (s) \|_{L^2_x} \lesssim_{E_{0,k}} 1, \\
\label{eqn:psiheat3} 
\int_0^\infty s^{k-1} \| \partial_x^k \psi_x (s) \|^2_{L^2_x} \lesssim_{E_{0,k}} 1, \\
\label{eqn:psiheat4} 
\int_0^\infty s^{k-1} \| \partial_x^{k-1} \psi_x (s) \|^2_{L^\infty_x} \lesssim_{E_{0,k}}  1 .  
\end{align}
Note, when $k > 1$, we have $\varepsilon$ improvements in the $L^2$ norms as well.

We want a slightly stronger version of these estimates using frequency envelopes to get a power of $\varepsilon$. First, we can demonstrate using the linear heat flow that \eqref{eqn:psiheat1} holds for $\psi = e^{s \Delta} \psi(0)$.  By translation invariance, it is easily seen that
 \begin{equation}
 \label{eqn:linheatbd1}
 \| P_k e^{s \Delta} \psi(0)  \|_{L^\infty_t L^2_{x}} \lesssim v_k (1 + s 2^{2k})^{-4}.
 \end{equation} 
 Note, we can replace the power $4$ above with any integer power and the estimate holds.
 Applying the Sobolev embedding theorem gives \eqref{eqn:psiheat1} for the linear heat flow.  We must then bootstrap this gain into the Duhamel term, which follows as in Lemma $5.3$ of \cite{dodsonsmith}.  The results for $A$ follow similarly.  The remaining bounds hold as in the work of \cite{smith2012geometric}.  Interpolating between $L^\infty$ and $L^1$, we can get a $\sqrt{\epsilon}$ gain as in \cite{dodsonsmith}, so let us verify if these are sufficient.


\begin{lemma}
\label{lem:lem5}
For any $k \geq 0$,
\begin{equation}
\sup_{s > 0} s^{\frac{k + 1}{2}} \| \partial_{x}^{k} A_{x}(s) \|_{L^{\infty}} \lesssim_{E_{0}, k} \epsilon,
\end{equation}
and for any $k \geq 1$,
\begin{equation}
\sup_{s > 0} s^{k/2} \| \partial_{x}^{k - 1} \psi_{x}(s) \|_{L^{\infty}} \lesssim_{E_{0}, k} \epsilon.
\end{equation}
\end{lemma}
\begin{proof}
We know from \cite{smith2012geometric} that $(4.50)$--$(4.52)$ and $(4.54)$--$(4.56)$ hold. Furthermore, we know from \cite{smith2012geometric} that for $k \geq 0$,
\begin{equation}
	\sup_{s > 0} s^{\frac{k + 1}{2}} \| \partial_{x}^{k} A_{x}(s) \|_{L^{\infty}} \lesssim_{E_{0}, k} 1,
\end{equation}
and for any $k \geq 1$,
\begin{equation}
	\sup_{s > 0} s^{k/2} \| \partial_{x}^{k - 1} \psi_{x}(s) \|_{L^{\infty}} \lesssim_{E_{0}, k} 1.
\end{equation}

Since $\psi_{x}(s)$ solves the harmonic map heat flow,
\begin{align}
 \psi_{x}(s) = e^{s \Delta} \psi_{x}(0) & + \int_{0}^{s} e^{(s - s') \Delta} \nabla \cdot (A \psi_{x}(s')) ds' + \int_{0}^{s} e^{(s - s') \Delta} (\nabla \cdot A) \psi_{x}(s') ds' \\
&  + \int_{0}^{s} e^{(s - s') \Delta} (A_{x}^{2} + \psi_{x}^{2}) \psi_{x}(s') ds'. \notag
\end{align}
We know from \eqref{eqn:Aheat1} that $\| \nabla \cdot A \|_{L_{s}^{1} L_{x}^{\infty}} \lesssim_{E_{0}} 1$. We also know from \eqref{eqn:Aheat3}  and \cite{smith2012geometric} that
\begin{equation}
\label{eqn:ALinfbds}
\int_{0}^{\infty} \| A_{x} \|_{L^{\infty}}^{2} ds \leq \sup_{s >0} s^{1/2} \| A_{x}(s) \|_{L^{\infty}} \cdot \int_{0}^{\infty} s^{-1/2} \| A_{x} \|_{L^{\infty}} ds \lesssim_{E_{0}} 1.
\end{equation}

Also, by  \eqref{eqn:psiheat3},
\begin{equation}
\label{eqn:psiLinfbds}
\int_{0}^{\infty} \| \psi_{x}(s) \|_{L^{\infty}}^{2} ds \lesssim_{E_{0}} 1.
\end{equation}
Therefore, we can partition $[0, \infty)$ into $\leq L(\eta)$ subintervals $I_{j}$ such that
\begin{equation}
\int_{I_{j}} \| \psi_{x}(s) \|_{L^{\infty}}^{2} + \| A_{x}(s) \|_{L^{\infty}}^{2} + \| \nabla \cdot A \|_{L^{\infty}} ds \leq \eta,
\end{equation}
for some $0 < \eta \ll 1$ to be specified later. Here we label the intervals $I_{0} = [0, s_{1})$, $I_{1} = [s_{1}, s_{2})$, ..., $I_{L(\eta) - 1} = [s_{L(\eta) - 1}, \infty)$.\medskip

We will start by considering only $ \in I_0$.  First observe that for all $s \in [0, \infty)$,
\begin{equation}
s^{1/2} \| e^{s \Delta} \psi_{x}(0) \|_{L^{\infty}} \lesssim \epsilon.
\end{equation}
Next, split
\begin{equation}
\aligned
s^{1/2} \| \int_{0}^{s} e^{(s - s') \Delta} (A_{x}^{2} + \psi_{x}^{2}) \psi_{x}(s') ds' \|_{L^{\infty}} \leq s^{1/2} \| \int_{0}^{\eta^{1/2} s} e^{(s - s') \Delta} (A_{x}^{2} + \psi_{x}^{2}) \psi_{x}(s') ds' \|_{L^{\infty}} \\ + s^{1/2} \| \int_{\eta^{1/2} s}^{s} e^{(s - s') \Delta} (A_{x}^{2} + \psi_{x}^{2}) \psi_{x}(s') ds' \|_{L^{\infty}}.
\endaligned
\end{equation}
Since $\| e^{(s - s') \Delta} \|_{L^{1} \rightarrow L^{\infty}} \lesssim \frac{1}{|s - s'|}$, by \eqref{eqn:Aheat3},  \eqref{eqn:psiheat3} we have
\begin{equation}
	\aligned
s^{1/2} \| \int_{0}^{\eta^{1/2} s} e^{(s - s') \Delta} (A_{x}^{2} + \psi_{x}^{2}) \psi_{x}(s') ds' \|_{L^{\infty}} & \lesssim s^{1/2} \int_{0}^{\eta^{1/2} s} \frac{1}{s - s'} (s')^{-1/2} ds' \cdot (\sup_{s >0} s^{1/2} \| \psi_{x}(s) \|_{L^{\infty}}) \\
& \lesssim \eta^{1/4} (\sup_{s >0} s^{1/2} \| \psi_{x}(s) \|_{L^{\infty}}).
\endaligned
\end{equation}
Meanwhile, since $\| e^{(s - s') \Delta} \|_{L^{\infty} \rightarrow L^{\infty}} \lesssim 1$ and we have assumed $s \in I_0$
\begin{equation}
\aligned
s^{1/2} \| \int_{\eta^{1/2} s}^{s} e^{(s - s') \Delta} (A_{x}^{2} + \psi_{x}^{2}) \psi_{x}(s') ds' \|_{L^{\infty}} &\lesssim (\| A_{x} \|_{L_{s}^{2} L_{x}^{\infty}}^{2} + \| \psi_{x} \|_{L_{s}^{2} L_{x}^{\infty}}^{2}) (\sup_{s' > \eta^{1/2} s} s^{1/2} \| \psi_{x}(s') \|_{L^{\infty}}) \\
&\lesssim \eta^{3/4} (\sup_{s > 0} s^{1/2} \| \psi_{x}(s) \|_{L^{\infty}}).
\endaligned
\end{equation}

Now turn to $\int_{0}^{s} e^{(s - s') \Delta} (\nabla \cdot A) \psi_{x}(s') ds'$. Using the bound $\| e^{(s - s') \Delta} \|_{L^{4/3} \rightarrow L^{\infty}} \lesssim \frac{1}{(s - s')^{3/4}}$, along with the bound $\| (\nabla \cdot A) \|_{L^{4/3}} \lesssim_{E_{0}} s^{-1/4}$,
\begin{equation}
	\aligned
s^{1/2} & \| \int_{0}^{\eta^{1/2} s} e^{(s - s') \Delta} (\nabla \cdot A) \psi_{x}(s') ds' \|_{L^{\infty}} \\
&  \lesssim_{E_{0}} s^{1/2} \int_{0}^{\eta^{1/2} s} \frac{1}{(s - s')^{3/4}} (s')^{-1/2} ds' \cdot \sup_{s > 0} s^{1/2} \| \psi_{x}(s) \|_{L^{\infty}} \\ 
& \lesssim \eta^{1/4} \cdot \sup_{s > 0} s^{1/2} \| \psi_{x}(s) \|_{L^{\infty}}.
\endaligned
\end{equation}
\begin{claim}
We have
\begin{equation}
\| \nabla \cdot A \|_{L^{4/3}} \lesssim_{E_{0}} s^{-1/4}.
\end{equation}
\end{claim}
\begin{proof}[Proof of claim]
To see why this is true, recall the formula for $A_{x}$,
\begin{equation}
A_{x}(s) = \int_{s}^{\infty} Im(\bar{\psi} (\partial_{l} + i A_{l}) \psi_{x}) ds'.
\end{equation}
Taking a derivative and using the product rule, and interpolating $(4.49)$--$(4.56)$ with the $\epsilon$'s all replaced by $1$'s,
\begin{equation}
\aligned
\| \nabla A \|_{L^{4/3}} \lesssim \int_{s}^{\infty} \| \psi_{x} \|_{L^{4}} \| \partial_{x}^{2} \psi_{x} \|_{L^{2}} + \| \partial_{x} \psi \|_{L^{8/3}}^{2} + \| A_{x} \|_{L^{4}} \| \partial_{x} \psi_{x} \|_{L^{4}} \| \psi_{x} \|_{L^{4}} + \| \partial_{x} A_{x} \|_{L^{4}} \| \psi_{x} \|_{L^{4}}^{2} ds' \\
\lesssim_{E_{0}} \int_{s}^{\infty} (s')^{-5/4} ds' \lesssim s^{-1/4}.
\endaligned
\end{equation}
\end{proof}
Meanwhile,
\begin{align}
& s^{1/2} \| \int_{\eta^{1/2} s}^{s} e^{(s - s') \Delta} (\nabla \cdot A) \psi_{x}(s') ds' \|_{L^{\infty}} \lesssim  \\
& \hspace{1cm} s^{1/2} \| \nabla \cdot A \|_{L_{s}^{1} L_{x}^{\infty}} (\sup_{s' > \eta^{1/2} s} \| \psi_{x}(s') \|_{L^{\infty}}) \lesssim  \\
& \hspace{2cm} \eta^{3/4} (\sup_{s > 0} s^{1/2} \| \psi_{x}(s) \|_{L^{\infty}}).
\end{align}

Finally turn to $\| \int_{0}^{s} e^{(s - s') \Delta} \nabla (A \psi_{x}(s')) ds' \|_{L^{\infty}}$. Since $\| \nabla e^{(s - s') \Delta} \|_{L^{2} \rightarrow L^{\infty}} \lesssim \frac{1}{s - s'}$, by $(4.50)$,
\begin{equation}
	\aligned
s^{1/2} \| \int_{0}^{\eta^{1/10} s} e^{(s - s') \Delta} \nabla (A \psi_{x}) ds' \|_{L^{\infty}} \lesssim_{E_{0}} s^{1/2} \int_{0}^{\eta^{1/10} s} \frac{1}{s - s'} (s')^{-1/2} ds' \cdot \sup_{s > 0} s^{1/2} \| \psi_{x}(s) \|_{L^{\infty}} \\
\lesssim \eta^{1/20} \cdot \sup_{s > 0} s^{1/2} \| \psi_{x}(s) \|_{L^{\infty}}.
\endaligned
\end{equation}
Next, interpolating $\| A \|_{L_{s}^{2} L_{x}^{\infty}} \leq \eta^{1/2}$ with $(4.53)$, gives $\| s^{1/4} A_{x} \|_{L_{s}^{4} L_{x}^{\infty}} \lesssim_{E_{0}} \eta^{1/4}$. Therefore, since $\| \nabla e^{(s - s') \Delta} \|_{L^{\infty} \rightarrow L^{\infty}} \lesssim \frac{1}{(s - s')^{1/2}}$,
\begin{equation}
	\aligned
s^{1/2} & \| \int_{\eta^{1/10} s}^{s} e^{(s - s') \Delta} \nabla (A \psi_{x}) ds' \|_{L^{\infty}} \\
\lesssim & s^{1/2} \int_{\eta^{1/10} s}^{s} \frac{1}{(s - s')^{1/2}} (s')^{-1/4} \| (s')^{1/4} A_{x} \|_{L^{\infty}} (s')^{-1/2} \| (s')^{1/2} \psi_{x} \|_{L^{\infty}} ds' \\
& \lesssim_{E_{0}}  \eta^{1/4} \eta^{-3/40}  \cdot \sup_{s > 0} s^{1/2} \| \psi_{x}(s) \|_{L^{\infty}}.
\endaligned
\end{equation}
Putting this all together,
\begin{equation}
\sup_{s \in I_{0}} s^{1/2} \| \psi_{x}(s) \|_{L^{\infty}} \lesssim \epsilon + \eta^{1/20} \sup_{s \in I_{0}} s^{1/2} \| \psi_{x}(s) \|_{L^{\infty}},
\end{equation}
which implies, for $\eta \ll 1$ sufficiently small, independent of $\epsilon > 0$,
\begin{equation}
\sup_{s \in I_{0}} s^{1/2} \| \psi_{x}(s) \|_{L^{\infty}} \lesssim \epsilon.
\end{equation}
Furthermore, observe that we could apply the same computations to $s > s_{1}$ if the interval over which the Duhamel term was computed was contained in $I_{0}$. Therefore, arguing by induction, we have proved
\begin{equation}
s^{1/2} \| \psi_{x}(s) \|_{L^{\infty}} \lesssim_{L(\eta)} \epsilon.
\end{equation}
Since $\eta$ does not depend on $\epsilon > 0$, and $L(\eta)$ depends only on $E_{0}$, the proof is complete.\medskip

Plugging this fact into the formula for $A_{x}(s)$,
\begin{equation}\label{4.67}
\| A_{x}(s) \|_{L^{\infty}} \lesssim \int_{s}^{\infty} \| \psi_{x}(s') \|_{L^{\infty}} \| \partial_{x} \psi_{x}(s') \|_{L^{\infty}} + \| \psi_{x}(s') \|_{L^{\infty}}^{2} \| A_{x}(s') \|_{L^{\infty}} \lesssim_{E_{0}} \epsilon s^{-1/2}.
\end{equation}
The estimates of terms with $\partial_{x}^{k}$ is identical. In this case, split the Duhamel integral into
\begin{equation}
\int_{0}^{s/2} e^{(s - s') \Delta} F(s') ds' + \int_{s/2}^{s} e^{(s - s') \Delta} [(\nabla \cdot A) \psi_{x} + (A_{x}^{2} + \psi_{x}^{2}) \psi_{x}] ds' + \int_{s/2}^{s} e^{(s - s') \Delta} \nabla(A \psi_{x}) ds',
\end{equation}
where $F = \nabla (A \psi_{x}) + (\nabla \cdot A) \psi_{x} + (A_{x}^{2} + \psi_{x}^{2}) \psi_{x}$. Since 
\begin{equation}
\| \partial_{x}^{k} e^{(s - s') \Delta} \|_{L^{q} \rightarrow L^{p}} \lesssim \frac{1}{(s - s')^{k/2 + \frac{1}{q} - \frac{1}{p}}} 
\end{equation} 
for $q \leq p$,
applying the above calculations for $k = 0$ to higher $k$'s gives
\begin{equation}
\| \partial_{x}^{k} \int_{0}^{s/2} e^{(s - s') \Delta} F(s') ds' \|_{L^{\infty}} \lesssim_{E_{0}, k} s^{-(k + 1)/2}.
\end{equation}
For $s/2 < s' < s$, using the bound $\| \partial_{x} e^{(s - s') \Delta} \|_{L^{\infty} \rightarrow L^{\infty}} \lesssim \frac{1}{(s - s')^{1/2}}$, for $k \geq 1$,
\begin{align}
& \| \partial_{x}^{k} \int_{s/2}^{s} e^{(s - s') \Delta} [(\nabla \cdot A) \psi_{x} + (A_{x}^{2} + \psi_{x}^{2}) \psi_{x}] ds' \|_{L^{\infty}} \lesssim \\
& \hspace{1cm}   \| \int_{s/2}^{s} \frac{1}{|s - s'|^{1/2}} \| \partial_{x}^{k -1} [(\nabla \cdot A) \psi_{x} + (A_{x}^{2} + \psi_{x}^{2}) \psi_{x}] ds' \|_{L^{\infty}}. \notag
\end{align}
Using $(4.49)$ and $(4.53)$ up to $k -1$, $(4.50)$--$(4.52)$, $(4.54)$--$(4.56)$, and arguing by induction and using the product rule,
\begin{equation}
\lesssim_{k, E_{0}} \int_{s/2}^{s} \frac{1}{|s - s'|^{1/2}} \epsilon (s')^{-k/2} (s')^{-1} ds' \lesssim_{E_{0}, k} s^{-(k + 1)/2}.
\end{equation}
Finally, for $\| \partial_{x}^{k + 1} \int_{s/2}^{s} e^{(s - s') \Delta} (A \psi_{x}) ds' \|_{L^{\infty}}$,
\begin{equation}
\aligned
\| \partial_{x}^{k + 1} \int_{s/2}^{s} e^{(s - s') \Delta} (A \psi_{x}) ds' \|_{L^{\infty}} \lesssim_{k} \int_{s/2}^{s} \frac{1}{|s - s'|^{1/2}} \| \partial_{x}^{k} (A_{x} \psi_{x}) \|_{L^{\infty}} ds' \\ \lesssim_{k} \sum_{j = 0}^{k} \int_{s/2}^{s} \frac{1}{|s - s'|^{1/2}} \| \partial_{x}^{j} A_{x} \|_{L^{\infty}} \| \partial_{x}^{k - j} \psi_{x} \|_{L^{\infty}} ds'.
\endaligned
\end{equation}
Since either $j < k$ or $k - j < k$, arguing by induction using $(4.59)$ up to $k - 1$ or $(4.53)$ up to $k - 1$, along with $(4.50)$--$(4.52)$ and $(4.54)$--$(4.56)$,
\begin{equation}
\lesssim_{E_{0}, k} \int_{s/2}^{s} \frac{\epsilon}{|s - s'|^{1/2}} (s')^{-k - 1} ds' \lesssim_{E_{0}, k} \epsilon s^{-(k + 1)/2}.
\end{equation}
This proves the lemma. Plugging the bounds on $\partial_{x}^{k} \psi_{x}(s)$ into the formula for $A$ proves the bounds for $\partial_{x}^{k} A$ as well.
\end{proof}

\subsection{$L^4$ bounds}

 We need to make a bootstrapping assumption on the solution to the harmonic map heat flow, namely that for any $k$,
\begin{equation}\label{4.76}
\| P_{k} \psi_{x}(s) \|_{L^4_{t,x}} \leq v_{k} (1 + s 2^{2k})^{-4}, \qquad |v_{k}| \leq \epsilon \ll 1.
\end{equation}


\begin{lemma}
\label{lem:Adecay}
Under the conditions of $(\ref{4.76})$,
\begin{equation}\label{4.77}
\| P_{k} A_{x}(s) \|_{L^4_{x,t}} \lesssim \epsilon v_{k} (1 + s 2^{2k})^{-7/2}.
\end{equation}
\end{lemma}
\begin{proof}
Using 
\begin{equation}
A_{x}(s) = \int_{s}^{\infty} \psi_{x}(s') \nabla \psi_{x}(s') ds' + \int_s^\infty A_x \psi_x^2 \ (s') d s'
\end{equation}
and applying \eqref{eqn:Aheat1}-\eqref{eqn:psiheat4}, we have
\begin{align}
\label{eqn:AL4bd1}
\| P_k A_x (s) \|_{L^4_{t,x}} & \lesssim \int_s^\infty  \| P_{\geq k} \nabla \psi_x (s') \|_{L^4_{t,x}} \| \psi_x (s') \|_{L^\infty_{t,x}} ds' \\
\label{eqn:AL4bd2}
& +  \int_s^\infty  \| P_{\leq k-5} \nabla \psi_x (s') \|_{L^\infty_{t,x}} \| P_k \psi_x (s') \|_{L^4_{t,x}} ds' \\
\label{eqn:AL4bd3}
& + \int_s^\infty  \| P_{ k-5 \leq \cdot \leq k +5 } A_x (s') \|_{L^\infty_{t,x}} \| P_{\leq k-5} \psi_x (s') \|_{L^8_{t,x}}^2 ds' \\
\label{eqn:AL4bd4}
& +   \int_s^\infty  \| A_x (s') \|_{L^\infty_{t,x}} \|  \psi_x (s') \|_{L^\infty_{t,x}}  \|  P_{\geq k -5} \psi_x (s') \|_{L^4_{t,x}} ds' .
\end{align}

For a fixed $s$, $2^{-2k} \leq s \leq 2^{-2(k - 1)}$,
\begin{equation}\label{4.28}
\| \psi_{x}(s) \|_{L^{\infty}} \lesssim \sum_{l \leq k} \| P_{l} \psi_{x}(s) \|_{L^{\infty}} + \sum_{l > k} \| P_{l} \psi_{x}(s) \|_{L^{\infty}} \lesssim \epsilon \sum_{l \leq k} 2^{l} + \epsilon \sum_{l > k} 2^{l} 2^{-8l} s^{-4} \epsilon s^{-1/2},
\end{equation}
and
\begin{equation}\label{4.29}
\| \nabla P_{\leq k + 5} \psi_{x}(s) \|_{L^{\infty}} \lesssim \epsilon \sum_{l \leq k + 5} 2^{2l} \lesssim \epsilon 2^{k} s^{-1/2}.
\end{equation}
Since $v_{k}$ is a frequency envelope,
\begin{equation}\label{4.30}
\aligned
& \int_{s}^{\infty} \| \psi_{x}(s') \|_{L^{\infty}} \| \nabla P_{k - 5 \leq \cdot \leq k + 5} \psi_{x}(s') \|_{L^4_{t,x}} ds' \\
& + \int_{s}^{\infty} \| P_{k - 5 \leq \cdot \leq \cdot k + 5} \psi_{x}(s') \|_{L^4_{t,x}} \| \nabla P_{\leq k + 5} \psi_{x}(s') \|_{L^{\infty}} ds' \\
& \lesssim \int_{s}^{\infty} \frac{2^{k} \epsilon v_{k}}{(s')^{1/2}} (1 + s' 2^{2k})^{-4} ds' \lesssim \epsilon v_{k} (1 + s 2^{2k})^{-7/2}.
\endaligned
\end{equation}
Using the Sobolev embedding theorem,
\begin{equation}\label{4.85}
\aligned
& \sum_{j > k - 5} 2^{k} \int_{s}^{\infty} \| P_{j} \psi_{x}(s') \|_{L^{2}} \| \nabla P_{j} \psi_{x}(s') \|_{L^4_{t,x}} ds' \\
& \hspace{1cm} \lesssim \sum_{j > k - 5} 2^{k} \int_{s}^{\infty} 2^{j} v_{j} \epsilon (1 + s 2^{2k})^{-4} (1 + s' 2^{2j})^{-4} ds' \lesssim \epsilon v_{k} (1 + s 2^{2k})^{-4}.
\endaligned
\end{equation}
We have
\begin{align}
 |\text{\eqref{eqn:AL4bd1}, \eqref{eqn:AL4bd2}, \eqref{eqn:AL4bd4} }| & \lesssim \epsilon \int_s^\infty v_k \sum_{k' \geq k-5} 2^{k'} \alpha_{k'} (1 + 2^{2 k'} s')^{-4} \frac{1}{(s')^{\frac12}} d s' \\
 & \lesssim \epsilon v_k (1 + 2^{2 k'} s)^{-\frac72}.
\end{align}
Using $\| P_k \psi_x (s) \|_{L^4_{t,x}} \lesssim v_k$ and $\| P_k \psi_x (s) \|_{L^\infty_{t,x}} \lesssim 2^k$, we have
\begin{equation}
\| P_k \psi_x (s) \|_{L^8_{t,x}} \lesssim \sqrt{v_k} 2^{\frac{k}{2}}.
\end{equation}
We also have by \eqref{eqn:Aheat1} that
\begin{equation}
\| P_k A_x (s) \|_{L^\infty_{t,x}} \lesssim \epsilon (1 + s^{-\frac{8}{2}} 2^{-8k }) s^{-\frac{1}{2}}.
\end{equation}
Hence,
\begin{align}
 |\text{\eqref{eqn:AL4bd3}}| & \lesssim \epsilon \int_s^\infty (s')^{-\frac92}  2^{-7k} \left( 2^{-\frac{k}{2}} \sum_{k' \leq k} \sqrt{\alpha_{k'}} 2^{\frac{k'}{2}}  \right)^2  ds' \\
 & \lesssim \epsilon v_k (1 + 2^{2 k'} s)^{-\frac72}.
\end{align}

\end{proof}

We wish to compute
\begin{equation}
\| \psi_{x} \|_{L_{t,x}^{4}}^{4} \lesssim \sum_{k_{1} \leq k_{2} \leq k_{3} \leq k_{4}} \| (P_{k_{1}} \psi_{x})(P_{k_{4}} \psi_{x}) \|_{L_{t,x}^{2}} \| P_{k_{2}} \psi_{x} \|_{L_{t,x}^{4}} \| P_{k_{3}} \psi_{x} \|_{L_{t,x}^{4}}.
\end{equation}

We follow the approach that we developed in the (admittedly simpler) setting of the proof of Theorem \ref{t2.1}.  In particular, similar to \eqref{2.7},\eqref{2.8}, we make the bootstrap assumptions,
\begin{equation}
\label{eqn:bootstrap_assumptions}
\aligned
\sup_{|k_{1}- k_{4}|\geq10} 2^{\frac{|k_{1} - k_{4}|}{2}} \| (P_{k_{1}} \psi_{x})(P_{k_{4}} \psi_{x}) \|_{L_{t,x}^{2}} \lesssim \epsilon^{2}, \\
 \| P_{k} \psi_{x} \|_{L_{t,x}^{4}} \lesssim v_k E,
\endaligned
\end{equation}
where $E$ is the energy.

Then we have the estimate
\begin{equation}
\label{eq:L4norm}
\aligned
\| \psi_x \|_{L^4_{t,x}}^4 & \lesssim \sum_{k_{1} \leq k_{2} \leq k_{3} \leq k_{4}} \| (P_{k_{1}} \psi_{x})(P_{k_{4}} \psi_{x}) \|_{L_{t,x}^{2}} \| P_{k_{2}} \psi_{x} \|_{L_{t,x}^{4}} \| P_{k_{3}} \psi_{x} \|_{L_{t,x}^{4}} \\ & \lesssim \epsilon^{2} \sum_{k_{1} \leq k_{2} \leq k_{3} \leq k_{4}} 2^{\frac{k_{1} - k_{2}}{2}} 2^{\frac{k_{2} - k_{3}}{2}} 2^{\frac{k_{3} - k_{4}}{2}} \| P_{k_{2}} \psi_{x} \|_{L_{t,x}^{4}} \| P_{k_{3}} \psi_{x} \|_{L_{t,x}^{4}} \\
& \lesssim \epsilon^{2} \sum_{k_{2} \leq k_{3}} 2^{\frac{k_{2} - k_{3}}{2}} \| P_{k_{2}} \psi_{x} \|_{L_{t,x}^{4}} \| P_{k_{3}} \psi_{x} \|_{L_{t,x}^{4}} \lesssim \epsilon^{2} E.
 \endaligned
 \end{equation}

Then, using this implicit bound on the $L^4_{x,t}$ norm, following the proof of Theorem $6.3$ of \cite{dodson2012bilinear}, we can observe that
\begin{equation}
\label{eqn:psiL4epsbd}
\| \psi_x \|^2_{L^4_{t,x}} \lesssim  \epsilon \sqrt{E}.
 \end{equation}
Indeed, computing verbatim as in the proof of Lemma \ref{lem:lem5}, we get that for any $k\geq 0$, 
\begin{equation}\label{4.98}
\| \partial_x^k \psi_x (s) \|_{L^4_x} \lesssim_k s^{-k/2} \| \psi_x (0) \|_{L^4_x}
 \end{equation}
and
\begin{equation}\label{4.99}
\| \partial_x \psi_x (s) - \partial_{x} e^{s \Delta} \psi_{x}(0) \|^2_{L^2_s L^4_{t.x}}  \lesssim  \|\psi_x (0) \|_{L^4_{t,x}}^2 \lesssim \epsilon \sqrt{E}.
 \end{equation}
 \begin{remark}
 We can only say that
 \begin{equation}
 \| \partial_{x} e^{s \Delta} \psi_{x}(0) \|_{L_{s}^{2} L_{t, x}^{4}} \lesssim E,
 \end{equation}
 which comes from $(\ref{eqn:bootstrap_assumptions})$.
 \end{remark}

\begin{proof}[Proof of $(\ref{4.98})$ and $(\ref{4.99})$]
Again recall the formula for $\psi_{x}(s)$ from the harmonic map heat flow:
\begin{align}
	\psi_{x}(s) = e^{s \Delta} \psi_{x}(0) & + \int_{0}^{s} e^{(s - s') \Delta} \nabla \cdot (A \psi_{x}(s')) ds' + \int_{0}^{s} e^{(s - s') \Delta} (\nabla \cdot A) \psi_{x}(s') ds' \\
	&  + \int_{0}^{s} e^{(s - s') \Delta} (A_{x}^{2} + \psi_{x}^{2}) \psi_{x}(s') ds'. \notag
\end{align}
Again recall that $\| \nabla \cdot A \|_{L_{s}^{1} L_{x}^{\infty}} \lesssim_{E_{0}} 1$, 
\begin{equation}
	\int_{0}^{\infty} \| A_{x} \|_{L^{\infty}}^{2} ds \leq \sup_{s >0} s^{1/2} \| A_{x}(s) \|_{L^{\infty}} \cdot \int_{0}^{\infty} s^{-1/2} \| A_{x} \|_{L^{\infty}} ds \lesssim_{E_{0}} 1,
\end{equation}
and
\begin{equation}
	\int_{0}^{\infty} \| \psi_{x}(s) \|_{L^{\infty}}^{2} ds \lesssim_{E_{0}} 1.
\end{equation}
Therefore, we can partition $[0, \infty)$ into $\leq L(\eta)$ subintervals $I_{j}$ such that
\begin{equation}
	\int_{I_{j}} \| \psi_{x}(s) \|_{L^{\infty}}^{2} + \| A_{x}(s) \|_{L^{\infty}}^{2} + \| \nabla \cdot A \|_{L^{\infty}} ds \leq \eta,
\end{equation}
for some $0 < \eta \ll 1$ to be specified later. Here we label the intervals $I_{0} = [0, s_{1})$, $I_{1} = [s_{1}, s_{2})$, ..., $I_{L(\eta) - 1} = [s_{L(\eta) - 1}, \infty)$.\medskip

Now then, since the heat kernel is an $L^{1}$ function, for any $k \geq 0$,
\begin{equation}
	s^{k/2}  \| \partial_{x}^{k} e^{s \Delta} \psi_{x}(0) \|_{L_{x}^{4}} \lesssim_{k} \| \psi_{x}(0) \|_{L_{x}^{4}}.
\end{equation}
Next, using the fact that $\| e^{(s - s') \Delta} \|_{L^{4} \rightarrow L^{4}} \lesssim 1$,
\begin{equation}
	\aligned
	 \| \int_{0}^{s} e^{(s - s') \Delta} (A_{x}^{2} + \psi_{x}^{2}) \psi_{x}(s') ds' \|_{L_{x}^{4}} & \lesssim (\| A_{x} \|_{L_{s}^{2} L_{x}^{\infty}}^{2} + \| \psi_{x} \|_{L_{s}^{2} L_{x}^{\infty}}^{2}) (\sup_{s' > 0} \| \psi_{x}(s') \|_{L_{x}^{4}})  \\
	 &\lesssim \eta (\sup_{s' > 0} \| \psi_{x}(s') \|_{L_{x}^{4}}).
	\endaligned
\end{equation}
The computation is similar for $\int_{0}^{s} e^{(s - s') \Delta} (\nabla \cdot A) \psi_{x}(s') ds'$. This time we do not need the bound $\| (\nabla \cdot A) \|_{L^{4/3}} \lesssim_{E_{0}} s^{-1/4}$, but rather simply use the fact that $\| e^{(s - s') \Delta} \|_{L^{4} \rightarrow L^{4}}$ combined with $\| \nabla \cdot A \|_{L_{s}^{1} L_{x}^{\infty}} \leq \eta$ to prove
\begin{equation}
	\aligned
	 & \| \int_{0}^{\eta^{1/2} s} e^{(s - s') \Delta} (\nabla \cdot A) \psi_{x}(s') ds' \|_{L_{x}^{4}}
	\lesssim \eta \sup_{s > 0} \| \psi_{x}(s) \|_{L_{x}^{4}}.
	\endaligned
\end{equation}

Finally, let us turn to $\| \int_{0}^{s} e^{(s - s') \Delta} \nabla (A \psi_{x}(s')) ds' \|_{L_{t,x}^{4}}$. Since 
\[\| \nabla e^{(s - s') \Delta} \|_{L^{4} \rightarrow L^{4}} \lesssim \frac{1}{|s - s'|^{1/2}} \] 
and $\| A \|_{L_{s}^{2} L_{x}^{\infty}} \leq \eta^{1/2}$, using the fact that $\int_{0}^{(1 - \eta) s} \frac{1}{|s - s'|} ds' \lesssim \log(\frac{1}{\eta})$,
\begin{equation}
	\aligned
	\| \int_{0}^{(1 - \eta) s} e^{(s - s') \Delta} \nabla (A \psi_{x}) ds' \|_{L_{x}^{4}} \lesssim \eta^{1/2} \log(\frac{1}{\eta}) \sup_{s > 0} \| \psi_{x}(s) \|_{L_{t,x}^{4}}.
	\endaligned
\end{equation}
On the other hand, by $(4.53)$,
\begin{equation}
	\aligned
	 \| \int_{(1 - \eta) s}^{s} e^{(s - s') \Delta} \nabla (A \psi_{x}) ds' \|_{L_{x}^{4}}
	& \lesssim  \int_{(1 - \eta) s}^{s} \frac{1}{(s - s')^{1/2}} (s')^{-1/2}  \| \psi_{x}(s') \|_{L_{x}^{4}} ds'  \\
	& \lesssim  \eta^{1/2}  \cdot \sup_{s > 0} \| \psi_{x}(s) \|_{L_{x}^{4}}.
	\endaligned
\end{equation}
Putting this all together,
\begin{equation}
	\sup_{s \in I_{0}}  \| \psi_{x}(s) \|_{L_{x}^{4}} \lesssim \| \psi_{x}(0) \|_{L_{x}^{4}} + \eta^{1/3} \sup_{s \in I_{0}} \| \psi_{x}(s) \|_{L_{x}^{4}},
\end{equation}
which implies, for $\eta \ll 1$ sufficiently small, 
\begin{equation}
	\sup_{s \in I_{0}} \| \psi_{x}(s) \|_{L_{x}^{4}} \lesssim \| \psi_{x}(0) \|_{L_{x}^{4}}.
\end{equation}
Furthermore, observe that we could apply the same computations to $s > s_{1}$ if the interval over which the Duhamel term was computed was contained in $I_{0}$. Therefore, arguing by induction, we have proved
\begin{equation}
	\| \psi_{x}(s) \|_{L_{x}^{4}} \lesssim_{L(\eta)} \| \psi_{x}(0) \|_{L_{x}^{4}}.
\end{equation}
Since $L(\eta)$ depends only on $E_{0}$, the proof is complete for $k = 0$. \medskip

The estimates of terms with $\partial_{x}^{k}$ is identical. In this case, split the Duhamel integral into
\begin{equation}
	\int_{0}^{(1 - \eta)s} e^{(s - s') \Delta} F(s') ds' + \int_{(1 - \eta)s}^{s} e^{(s - s') \Delta} [(\nabla \cdot A) \psi_{x} + (A_{x}^{2} + \psi_{x}^{2}) \psi_{x}] ds' + \int_{(1 - \eta)s}^{s} e^{(s - s') \Delta} \nabla(A \psi_{x}) ds',
\end{equation}
where $F = \nabla (A \psi_{x}) + (\nabla \cdot A) \psi_{x} + (A_{x}^{2} + \psi_{x}^{2}) \psi_{x}$. Since 
\begin{equation}
	\| \partial_{x}^{k} e^{(s - s') \Delta} \|_{L^{q} \rightarrow L^{p}} \lesssim \frac{1}{(s - s')^{k/2 + \frac{1}{q} - \frac{1}{p}}} 
\end{equation} 
for $q \leq p$,
applying the above calculations for $k = 0$ to higher $k$'s gives
\begin{equation}
	\| \partial_{x}^{k} \int_{0}^{(1 - \eta) s} e^{(s - s') \Delta} F(s') ds' \|_{L_{x}^{4}} \lesssim_{E_{0}, k} s^{-k/2} \| \psi_{x}(0) \|_{L_{x}^{4}}.
\end{equation}
For $(1 - \eta)s < s' < s$, using the bound $\| \partial_{x} e^{(s - s') \Delta} \|_{L_{x}^{4} \rightarrow L_{x}^{4}} \lesssim \frac{1}{|s - s'|^{1/2}}$, for $k \geq 1$,
\begin{align}
\label{eqn:4pt116}
	& \| \partial_{x}^{k} \int_{(1 - \eta)s}^{s} e^{(s - s') \Delta} [(\nabla \cdot A) \psi_{x} + (A_{x}^{2} + \psi_{x}^{2}) \psi_{x}] ds' \|_{L_{x}^{4}} \lesssim \\
	& \hspace{1cm}   \| \int_{(1 - \eta)s}^{s} \frac{1}{|s - s'|^{1/2}} \| \partial_{x}^{k -1} [(\nabla \cdot A) \psi_{x} + (A_{x}^{2} + \psi_{x}^{2}) \psi_{x}] ds' \|_{L_{x}^{4}}. \notag
\end{align}
Using $(4.49)$ and $(4.53)$ up to $k -1$, $(4.50)$--$(4.52)$, $(4.54)$--$(4.56)$, and arguing by induction and using the product rule,
\begin{equation}
 \text{\eqref{eqn:4pt116}}	\lesssim_{k, E_{0}} \| \psi_{x}(0) \|_{L_{x}^{4}} \int_{(1 - \eta)s}^{s} \frac{1}{|s - s'|^{1/2}} (s')^{-(k - 1)/2} (s')^{-1} ds' \lesssim_{E_{0}, k} s^{-k/2} \| \psi_{x}(0) \|_{L_{x}^{4}}.
\end{equation}
Finally, for $\| \partial_{x}^{k + 1} \int_{s/2}^{s} e^{(s - s') \Delta} (A \psi_{x}) ds' \|_{L^{4}}$,
\begin{equation}
\label{eqn:4pt118}
	\aligned
	\| \partial_{x}^{k + 1} \int_{(1 - \eta)s}^{s} e^{(s - s') \Delta} (A \psi_{x}) ds' \|_{L_{x}^{4}} \lesssim_{k} \int_{(1 - \eta)s}^{s} \frac{1}{|s - s'|^{1/2}} \| \partial_{x}^{k} (A_{x} \psi_{x}) \|_{L_{x}^{4}} ds' \\ \lesssim_{k} \sum_{j = 1}^{k} \int_{(1 - \eta)s}^{s} \frac{1}{|s - s'|^{1/2}} \| \partial_{x}^{j} A_{x} \|_{L^{\infty}} \| \partial_{x}^{k - j} \psi_{x} \|_{L_{x}^{4}} ds' + \int_{(1 - \eta)s}^{s} \frac{1}{|s - s'|^{1/2}} \| A_{x} \|_{L_{x}^{\infty}} \| \partial_{x}^{k} \psi_{x}(s') \|_{L_{x}^{4}}
	\endaligned
\end{equation}
Arguing by induction up to $k - 1$,
\begin{equation}
\aligned
	\text{\eqref{eqn:4pt118}} \lesssim_{E_{0}, k} s^{-k/2} \| \psi_{x}(0) \|_{L_{x}^{4}} + \int_{(1 - \eta)s}^{s} \frac{1}{|s - s'|^{1/2}} (s')^{-1/2} \| \partial_{x}^{k} \psi_{x}(s') \|_{L_{x}^{4}} \\ \lesssim_{E_{0}, k} s^{-k/2} \| \psi_{x}(0) \|_{L_{x}^{4}} + s^{-k/2} \eta^{1/2} (\sup_{s > 0} s^{k/2} \| \partial_{x}^{k} \psi_{x}(s) \|_{L_{x}^{4}}).
\endaligned
\end{equation}
Iterating over a finite number of subintervals of $[0, \infty)$, where the number of subintervals depends on $E_{0}$ and $k$ proves $(\ref{4.98})$.\medskip

The Duhamel representation and $(\ref{4.98})$ gives $(\ref{4.99})$. First, by $(\ref{eqn:Aheat1})$--$(\ref{eqn:psiheat4})$,
\begin{equation}
\aligned
\| A_{x}^{2} + \psi_{x}^{2} + (\nabla \cdot A) \|_{L_{s}^{4/3} L_{x}^{4}} & \lesssim \| A_{x}^{2} + \psi_{x}^{2} + (\nabla \cdot A) \|_{L_{s}^{1} L_{x}^{\infty}}^{1/2} \| A_{x}^{2} + \psi_{x}^{2} + (\nabla \cdot A) \|_{L_{s}^{2} L_{x}^{2}}^{1/2} \\
& \lesssim_{E_{0}} 1,
\endaligned
\end{equation}
so by the Hardy--Littlewood Sobolev inequality and $\| \nabla e^{(s - s') \Delta} \|_{L^{2} \rightarrow L^{4}} \lesssim \frac{1}{|s - s'|^{3/4}}$,
\begin{equation}\label{4.120}
\aligned
 & \| \nabla \int_{0}^{s} e^{(s - s') \Delta} \nabla \cdot (A \psi_{x}(s')) ds'  + \nabla \int_{0}^{s} e^{(s - s') \Delta} (\nabla \cdot A) \psi_{x}(s') ds' \|_{L_{s}^{2} L_{x}^{4}}\\
& \hspace{2cm} \lesssim_{E_{0}} \sup_{s > 0} \| \psi_{x}(s) \|_{L^{4}}.
\endaligned
 \end{equation}
 Next, since $\| \nabla^{2} e^{(s - s') \Delta} \|_{L^{4} \rightarrow L^{4}} \lesssim \frac{1}{|s - s'|}$ and Young's inequality,
\begin{equation}
\aligned
\| \nabla \int_{0}^{s/2} e^{(s - s') \Delta} \nabla \cdot (A \psi_{x}(s')) ds' \|_{L_{s}^{2} L_{x}^{4}} & \lesssim \| A \|_{L_{s}^{2} L_{x}^{\infty}} \| \psi_{x}(s) \|_{L_{s}^{\infty} L_{x}^{4}} \\
& \lesssim_{E_{0}} \sup_{s > 0} \| \psi_{x}(s) \|_{L^{4}}.
\endaligned
\end{equation}
Meanwhile, by the product rule, $(\ref{4.98})$, and $(\ref{4.120})$,
\begin{equation}
\aligned
\| \nabla \int_{s/2}^{s} e^{(s - s') \Delta} \nabla \cdot (A \psi_{x}(s')) ds' \|_{L_{s}^{2} L_{x}^{4}} \lesssim_{E_{0}} \| \psi_{x} \|_{L_{s}^{\infty} L_{x}^{4}} + \| \int_{s/2}^{s} e^{(s - s') \Delta} A \nabla^{2} \psi_{x}(s') ds' \|_{L_{s}^{2} L_{x}^{4}} \\ 
\lesssim_{E_{0}} \| \psi_{x} \|_{L_{s}^{\infty} L_{x}^{4}} + \| \int_{s/2}^{s} \frac{1}{s'} \| A_{x}(s') \|_{L^{\infty}} (s' \| \partial_{x}^{2} \psi_{x}(s') \|_{L^{4}}) ds' \lesssim \| \psi_{x}(0) \|_{L_{x}^{4}}.
\endaligned
\end{equation}
\end{proof}

From Lemma \ref{lem:Adecay}, we have
\begin{equation}
\label{eqn:Abdassumps_alt}
\| A_x (s) \|_{L^4_{t,x}} \lesssim \sqrt{\epsilon}.
\end{equation}
We can then continue and establish that 
\begin{equation}\label{bound:A}
  \| P_k A_x (s) \|_{L^1_x} \lesssim 2^{-k}
 \end{equation}
using 
\begin{equation}
A_x(s) = - \int_s^\infty \Im (\overline{\psi}_x (\partial_\ell \psi_\ell + i A_\ell \psi_\ell))  (r) dr
 \end{equation}
and \eqref{eqn:Aheat1}-\eqref{eqn:psiheat4} as in Corollary $6.5$ and Theorem $6.6$ of \cite{dodson2012bilinear}.

Continuing, we can then establish that
\begin{equation}\label{bound:At}
\| \partial_x \psi_t (s) - \partial_{x} e^{s \Delta} \psi_{t}(0) \|_{L^2_s L^4_{t.x}}  \lesssim_{E}  \sqrt{\epsilon}, \qquad \text{for any } k \geq 1, \qquad \ \ \| \partial_{x}^{k} \psi_t (s) \|_{L^4_{t,x}} \lesssim_{E, k} s^{-k/2} \sqrt{\epsilon},
 \end{equation}
 by applying \eqref{eqn:psiL4epsbd}, the Harmonic Map Flow formula for $\psi_t(s)$. Also by applying the caloric gauge formula for $A_t$ as in Lemma $6.7$, Corollary $6.8$ of \cite{dodson2012bilinear},
 \begin{equation}\label{bound:At12}
 A_{t} = A_{t}^{(1)} + A_{t}^{(2)}, \qquad A_{t}^{(1)} = \int_{s}^{\infty} |e^{r \Delta} \nabla \cdot \psi_{x}|_{s = 0}|^{2} dr, \qquad \| A_{t}^{(2)} \|_{L_{t,x}^{2}} \lesssim_{E} \sqrt{\epsilon},
 \end{equation}
 and for any $k \geq 1$,
 \begin{equation}\label{bound:At1}
 \| \partial_{x}^{k} A_{t} \|_{L_{t,x}^{2}} \lesssim_{E, k} s^{-k/2} \sqrt{\epsilon}.
 \end{equation}

Similarly, we get 
\begin{equation}\label{bound:Ax}
\| P_k A_x (s) \|_{L^2_{t,x}} \lesssim 2^{-k} \sqrt{\epsilon} \alpha_k
\end{equation}
by computing as in Corollary $6.10$ in \cite{dodson2012bilinear} using Theorem $6.9$ of that work which follows verbatim as it is a frequency envelope bound.

Since $(\ref{eqn:Abdassumps_alt})$-\eqref{bound:A} implies 
\begin{equation}
\label{eqn:combdA}
\| [P_{k}, P_{\leq k - 5} A] \|_{L_{t,x}^{2}} \lesssim C(E) 2^{-k} \epsilon,
\end{equation}
we have
\begin{equation}\label{4.89}
\| P_{l} \psi_{x} \|_{L_{t}^{\infty} L_{x}^{2}} \| P_{k} \psi_{x} \|_{L_{t,x}^{4}} \| \tilde{P}_{k} \nabla \psi_{x} \|_{L_{t,x}^{4}} \| [P_{k}, P_{\leq k - 5} A] \|_{L_{t,x}^{2}} \lesssim C(E) \epsilon^{3}.
\end{equation}
Next, integrating by parts, using Hardy's inequality,
\begin{equation}\label{4.90}
\aligned
-2 \int |P_{j} \psi_{y}(t, y)|^{2} \frac{(x - y)}{|x - y|} \cdot Re[(\sum_{l = 1}^{d} \overline{(P_{\leq k - 5} A_{l}) \partial_{l} P_{k} \psi_{x}}) \nabla P_{k} \psi_{x}]  \\
- 2 \int |P_{j} \psi_{y}(t, y)|^{2} \frac{(x - y)}{|x - y|} \cdot Re[\overline{P_{k} \psi_{x}} \nabla (\sum_{l = 1}^{d} (P_{\leq k - 5} A_{l}) \partial_{l} P_{k} \psi_{x})] \\
= -2 \int |P_{j} \psi_{y}(t, y)|^{2} \frac{(x - y)}{|x - y|} \cdot Re[\overline{P_{k} \psi_{x}}  (\sum_{l = 1}^{d} (P_{\leq k - 5} \nabla A_{l}) \partial_{l} P_{k} \psi_{x})] \\
+ 2 \int |P_{j} \psi_{y}(t, y)|^{2} \frac{(x - y)}{|x - y|} \cdot Re[\overline{P_{k} \psi_{x}}  (\sum_{l = 1}^{d} (\partial_{l} P_{\leq k - 5} A_{l}) \nabla P_{k} \psi_{x})] \\
+ \int |P_{j} \psi_{x}|^{2} \frac{1}{|x - y|} |P_{k} \psi_{x}| |\nabla P_{k} \psi_{x}| |P_{\leq k - 5} A_{x}| \\
\lesssim 2^{k} \| P_{j} \psi \|_{L_{t}^{\infty} L_{x}^{2}}^{2} \| P_{k} \psi_{x} \|_{L_{t,x}^{4}}^{2} \| \nabla A \|_{L_{t,x}^{2}} \lesssim C(E) \epsilon^{\frac72}.
\endaligned
\end{equation}

\section{Two dimensional Schr{\"o}dinger map flow}
\label{sec:SchMap}

\begin{remark}
In this section, all implicit constants $\lesssim$ refer to $\lesssim_{E}$. Recall that we are also using the bootstrap condition $(\ref{8.13})$--$(\ref{8.13.2})$.
\end{remark}

\subsection{Proof of Lemma $\ref{lmass}$} We are ready to prove Lemma $\ref{lmass}$.
\begin{proof}[Proof of Lemma $\ref{lmass}$]
Computing the change of mass,
\begin{equation}\label{5.1}
\frac{d}{dt} \frac{1}{2} (P_{k} \psi_{m}, P_{k} \psi_{m})_{L^{2}} = (-2i P_{k}(\sum_{l = 1}^{d} A_{l} \partial_{l} \psi_{m}), P_{k} \psi_{m})_{L^{2}}
\end{equation}
\begin{equation}\label{5.2}
+ (P_{k}(A_{t} + \sum_{l = 1}^{d} (A_{l}^{2} - i \partial_{l} A_{l})) \psi_{m}, P_{k} \psi_{m})_{L^{2}} - (i P_{k} \sum_{l = 1}^{d} \psi_{l} Im(\bar{\psi}_{l} \psi_{m}), P_{k} \psi_{m})_{L^{2}}.
\end{equation}

Now then, by $(\ref{8.13})$--$(\ref{8.13.2})$,
\begin{equation}\label{5.3}
\aligned
& \int - (i P_{k} \sum_{l = 1}^{d} \psi_{l} Im(\bar{\psi}_{l} \psi_{m}), P_{k} \psi_{m})_{L^{2}} dt \\ 
& \lesssim \sum_{|j_{3} - k| \leq |j_{2} - k| \leq |j_{1} - k|} \| (P_{k} \psi_{x})(P_{j_{1}} \psi_{x}) \|_{L_{t,x}^{2}} \| P_{j_{2}} \psi_{x} \|_{L_{t,x}^{4}} \| P_{j_{3}} \psi_{x} \|_{L_{t,x}^{4}} \lesssim C \epsilon^{2} v_{k}^{2}.
\endaligned
\end{equation}
Next,
\begin{equation}\label{5.4}
\aligned
 & \| (P_{k} \psi_{x}) P_{k}(A_{l}^{2} \psi_{x}) \|_{L_{t,x}^{1}} +  \| (P_{k} \psi_{x}) P_{k}(\nabla \cdot A \psi_{x}) \|_{L_{t,x}^{1}}
+  \| (P_{k} \psi_{x}) P_{k}(A_{t} \psi_{x}) \|_{L_{t,x}^{1}} \\
& \hspace{1cm} \lesssim C \epsilon v_{k}^{2}.
\endaligned
\end{equation}
Indeed, by $(\ref{eqn:Abdassumps_alt})$-\eqref{bound:A} and $(\ref{eqn:bootstrap_assumptions})$,
\begin{equation}\label{5.5}
 \| (P_{k} \psi_{x}) P_{k}(A_{l}^{2} \psi_{x}) \|_{L_{t,x}^{1}} \lesssim \| (P_{k} \psi_{x}) \psi_{x} \|_{L_{t,x}^{2}} \| A_{x} \|_{L_{t,x}^{4}}^{2} \lesssim C v_{k}^{2} \epsilon,
\end{equation}
\begin{equation}\label{5.6}
 \| (P_{k} \psi_{x}) P_{k}(\nabla \cdot A \psi_{x}) \|_{L_{t,x}^{1}} \lesssim \| (P_{k} \psi_{x}) \psi_{x} \|_{L_{t,x}^{2}} \| \nabla \cdot A_{x} \|_{L_{t,x}^{2}} \lesssim C v_{k}^{2} \epsilon,
\end{equation}
\begin{equation}\label{5.7}
\| (P_{k} \psi_{x}) P_{k}(A_{t}^{(2)} \psi_{x}) \|_{L_{t,x}^{1}} \lesssim  \| (P_{k} \psi_{x}) \psi_{x} \|_{L_{t,x}^{2}} \| A_{t} \|_{L_{t,x}^{2}} \lesssim C v_{k}^{2} \epsilon,
\end{equation}
for $A_{t}^{(2)}$ as defined in \eqref{bound:At12}.  Following $(\ref{5.3})$ and using 
$$\| (e^{r \Delta} \nabla \cdot P_{k_{1}} \psi_{x}|_{s = 0})(P_{k} \psi_{x}) \|_{L_{t,x}^{2}} \lesssim 2^{-r 2^{2k_{1}}} 2^{k_{1}} 2^{-\frac{|k - k_{1}|}{2}} v_{k} v_{k_{1}}$$ 
and therefore $\| (e^{r \Delta} \nabla \cdot P_{k_{1}} \psi_{x}|_{s = 0})(P_{k} \psi_{x}) \|_{L_{r, t,x}^{2}} \lesssim  2^{-\frac{|k - k_{1}|}{2}} v_{k} v_{k_{1}}$,
\begin{equation}\label{5.7.1}
\| (P_{k} \psi_{x}) A_{t}^{(1)} \psi_{x} \|_{L_{t,x}^{1}} = \| (P_{k} \psi_{x}) (\int_{0}^{\infty} |e^{r \Delta} \nabla \cdot \psi_{x}|_{s = 0}|^{2} dr) \psi_{x} \|_{L_{t,x}^{1}} \lesssim C \epsilon^{2} v_{k}^{2}
\end{equation}
for $A_{t}^{(1)}$ as defined in \eqref{bound:At12}.

Now split
\begin{equation}\label{5.8}
A = P_{\leq k - 5} A + P_{> k - 5} A.
\end{equation}
By $(\ref{eqn:Abdassumps_alt})$-\eqref{bound:A}, computing $(-2i P_{k}(\sum_{l = 1}^{d} A_{l} \partial_{l} \psi_{m}), P_{k} \psi_{m})_{L^{2}}$ with $A$ replaced by $P_{> k - 5} A$, the contribution is bounded by
\begin{equation}\label{5.9}
\sum_{l \geq k - 5} \| P_{l} A_{x} \|_{L_{t,x}^{2}} \| (\nabla P_{\leq l + 10} \psi_{x})(P_{k} \psi_{x}) \|_{L_{t,x}^{2}} \lesssim \epsilon v_{k}^{2}.
\end{equation}
Next, integrating by parts,
\begin{equation}\label{5.9alt}
\aligned
 Re((\sum_{l = 1}^{d} \overline{(P_{\leq k - 5} A_{l}) \partial_{l} P_{k} \psi_{x}}), P_{k} \psi_{x})_{L^{2}} \lesssim \| \nabla A \|_{L_{t,x}^{2}} \| P_{k} \psi_{x} \|_{L_{t,x}^{4}}^{2} \lesssim \epsilon v_{k}^{2}.
\endaligned
\end{equation}
Finally, computing the commutator, if $\tilde{P}_{k} = P_{k - 2} + P_{k - 1} + P_{k} + P_{k + 1} + P_{k + 2}$,
\begin{equation}\label{5.10}
P_{k} ((P_{\leq k - 5} A) \psi_{x}) = P_{k}((P_{\leq k - 5} A) \tilde{P}_{k} \psi_{x}) = (P_{\leq k - 5} A) P_{k} \psi_{x} + [P_{k}, P_{\leq k - 5} A] (\tilde{P}_{k} \psi_{x}).
\end{equation}
Since $(\ref{eqn:Abdassumps_alt})$-\eqref{bound:A} implies $\| [P_{k}, P_{\leq k - 5} A] \|_{L_{t,x}^{2}} \lesssim 2^{-k} \epsilon$,
\begin{equation}\label{5.11}
\| P_{k} \psi_{x} \|_{L_{t,x}^{4}} \| \tilde{P}_{k} \nabla \psi_{x} \|_{L_{t,x}^{4}} \| [P_{k}, P_{\leq k - 5} A] \|_{L_{t,x}^{2}} \lesssim \epsilon v_{k}^{2}.
\end{equation}
\begin{remark}
The same results hold for $v_{k}$ replaced by $2^{-k \sigma} v_{k}(\sigma)$.
\end{remark}

\end{proof}

\subsection{Bilinear estimate : Proof of Lemma $\ref{lbilinear}$}

In this section we wish to prove the following lemma.
\begin{lemma}\label{l7.1}
For any $s, s' \geq 0$, $l \leq k - 10$,
\begin{equation}\label{7.1}
\| (P_{k} \psi_{x}(s))(P_{l} \psi_{x}(s')) \|_{L_{t,x}^{2}}^{2} \lesssim 2^{l - k} v_k^{2} v_l^{2} (1 + s 2^{2k})^{-3} (1 + s' 2^{2l})^{-3}.
\end{equation}
\end{lemma}
\begin{proof}
Recall from $(\ref{8.13})$--$(\ref{8.13.2})$, we have that this bound holds modulo a constant $C > 1$ from our 
\begin{equation}\label{7.2}
\| (P_{k} \psi_{x}(s))(P_{l} \psi_{x}(s')) \|_{L_{t,x}^{2}}^{2} \lesssim C 2^{l - k} v_k^{2} v_l^{2} (1 + s 2^{2k})^{-3} (1 + s' 2^{2l})^{-3}
\end{equation}
where using the $\lesssim$ sign again means that we will show that we can bootstrap up to a uniform constant on the right hand side building upon the bootstrapping bound where we have a larger constant bound.

Following Theorem \ref{t3.1}, let
\begin{equation}\label{7.3}
\aligned
F_{1} = (\partial_{t} - i \Delta) \psi_{x}(t, s'), \qquad F_{2} = (\partial_{t} - i \Delta) \psi_{x}(t, s).
\endaligned
\end{equation}
To compute $(\ref{7.3})$, note that by compatibility condition $D_t \psi_x = D_x \psi_t$, we have
\begin{equation}\label{7.4}
\aligned
\partial_{t} \psi_{x}(t, s) - i \Delta \psi_{x}(t, s) = D_{t} \psi_{x}(t, s) - i A_{t} \psi_{x}(t, s) - i \Delta \psi_{x}(t, s) = D_{x} \psi_{t} - i A_{t} \psi_{x} - i \Delta \psi_{x}.
\endaligned
\end{equation}
To control the contribution of $i A_{t} \psi_{x}$, it suffices to estimate the following term via the bootstrap assumption $(\ref{7.2})$.  For $A_{t}^{(2)}$ defined as in \eqref{bound:At12}, we have
\begin{equation}\label{7.5}
\aligned
 \| (P_{k} \psi_{x}(s)) P_{k}(A_{t}^{(2)} \psi_{x}(s)) \|_{L_{t,x}^{1}} & \lesssim \| P_{k} \psi_{x}(s) \|_{L_{t,x}^{4}} \| P_{\leq k - 5} \psi_{x}(s) \|_{L_{t,x}^{4}} \| P_{k - 5 \leq \cdot \leq k + 5} A_{t}^{(2)} \|_{L_{t,x}^{2}} \\
& + \| P_{k} \psi_{x}(s) \|_{L_{t,x}^{4}} \| P_{k - 5 \leq \cdot \leq k + 5} \psi_{x}(s) \|_{L_{t,x}^{4}} \| A_{t}^{(2)} \|_{L_{t,x}^{2}} \\
& + \| P_{k} \psi_{x}(s) \|_{L_{t,x}^{4}} \sum_{l \geq k + 5} \| P_{l} \psi_{x}(s) \|_{L_{t,x}^{4}} \| P_{l} A_{t}^{(2)}(s) \|_{L_{t,x}^{2}} \\
& \lesssim C \epsilon v_k^{2} (1 + s 2^{2k})^{-6}.
\endaligned
\end{equation}
Now by $(\ref{bound:At})$,
\begin{equation}
\aligned
& \| P_{k} \psi_{x}(s) \|_{L_{t,x}^{4}} \| P_{k - 5 \leq \cdot \leq k + 5} \psi_{x}(s) \|_{L_{t,x}^{4}} \| A_{t}^{(2)} \|_{L_{t,x}^{2}} \\ 
&+ \| P_{k} \psi_{x}(s) \|_{L_{t,x}^{4}} \sum_{l \geq k + 5} \| P_{l} \psi_{x}(s) \|_{L_{t,x}^{4}} \| P_{l} A_{t}^{(2)}(s) \|_{L_{t,x}^{2}} \\
& \lesssim C \epsilon v_k^{2} (1 + s 2^{2k})^{-6}.
\endaligned
\end{equation}
From \eqref{eqn:combdA}, \eqref{bound:At1}, and \eqref{4.76}, we have
\begin{equation}\label{7.5.1}
\aligned
  \| P_{k} \psi_{x}(s) \|_{L_{t,x}^{4}} & \| P_{\leq k - 5} \psi_{x}(s) \|_{L_{t,x}^{4}} \| P_{k - 5 \leq \cdot \leq k + 5} A_{t}^{(2)} \|_{L_{t,x}^{2}} \\
 &  \lesssim C \sqrt{\epsilon} v_k^{2} (1 + s 2^{2k})^{-6}.
 \endaligned
 \end{equation}
 Now then, from $(\ref{8.13})$--$(\ref{8.13.2})$ and following the computation in $(\ref{5.3})$, for $A_{t}^{(1)}$ defined as in \eqref{bound:At12} we have
 \begin{equation}\label{7.5.2}
 \aligned
\| (P_{k} \psi_{x}(s)) A_{t}^{(1)} \psi_{x}(s) \|_{L_{t,x}^{1}} \lesssim  \| (P_{k} \psi_{x}(s)) (\int_{s}^{\infty} |e^{r \Delta} \nabla \cdot \psi_{x}|_{s = 0}|^{2} dr) \psi_{x}(s) \|_{L_{t,x}^{1}} \\ \lesssim \sum_{k_{1}, k_{2} \leq k - 5} \int_{s}^{\infty} \| (P_{k} \psi_{x}(s))(e^{r \Delta} P_{k_{1}} \nabla \cdot \psi_{s}|_{s = 0}) \|_{L_{t,x}^{2}} \| (P_{k - 5 \leq \cdot \leq k + 5} \psi_{x}(s))(e^{r \Delta}P_{k_{2}} \psi_{x}|_{s = 0}) \|_{L_{t,x}^{2}} dr \\
+ \sum_{\max \{ k_{1}, k_{2} \} \geq k - 5} \int_{s}^{\infty} \| (P_{k} \psi_{x}(s)) \psi_{x}(s) \cdot (e^{r \Delta} P_{k_{1}} \nabla \cdot \psi_{x}|_{s = 0})(e^{r \Delta} P_{k_{2}} \nabla \cdot \psi_{x}|_{s = 0}) \|_{L_{t,x}^{1}} dr \lesssim C \epsilon^{2} v_{k}^{2}.
 \endaligned
 \end{equation}
 
\medskip

Therefore,
\begin{equation}\label{7.6}
\aligned
&-2^{l - 2k} \int |P_{l} \psi_{x}(s')|^{2} \frac{(x - y)}{|x - y|} Im[(\overline{P_{k} \psi_{x}(s)}) \nabla P_{k}(i A_{t} \psi_{x}(s))] \\
&-2^{l - 2k} \int |P_{l} \psi_{x}(s')|^{2} \frac{(x - y)}{|x - y|} Im[P_{k}(\overline{i A_{t} \psi_{x}} \nabla P_{k} \psi_{x}(s)] \\
&-2^{l - 2k + 1} \int Im[P_{l}(\overline{\psi_{x}}) P_{l}(i A_{t} \psi_{x})] \frac{(x - y)}{|x - y|} \cdot Im[(\overline{P_{k} \psi_{x})} \nabla (P_{k} \psi_{x}(s))] \\
& \lesssim 2^{l - k} C  \sqrt{\epsilon}  v_k^{2} v_l^{2}.
\endaligned
\end{equation}

Next, expand
\begin{equation}\label{7.7}
\aligned
& D_{x} \psi_{t} - i \Delta \psi_{x} = D_{x}(\psi_{t}(s) - e^{s \Delta} \psi_{t}|_{s = 0}) + D_{x}(e^{s \Delta} \psi_{t}|_{s = 0}) - i \Delta \psi_{x} \\ 
& = \partial_{x}(\psi_{t}(s) - e^{s \Delta} \psi_{t}|_{s = 0}) + D_{x}(e^{s \Delta} \psi_{t}|_{s = 0}) - i \Delta \psi_{x} + i A_{x}(\psi_{t}(s) - e^{s \Delta} \psi_{t}|_{s = 0}).
\endaligned
\end{equation}
By Lemma $5.4$ from \cite{dodsonsmith},
\begin{equation}
\| P_{k}(\psi_{t}(s) - e^{s \Delta} \psi_{t}|_{s = 0}) \|_{L_{t,x}^{2}} \lesssim v_{k} (1 + s 2^{2k})^{-3},
\end{equation}
so applying $(\ref{4.76})$, $(\ref{4.77})$, $(\ref{bound:A})$, \eqref{bound:At1} and $(\ref{bound:Ax})$ we have
\begin{equation}\label{7.8}
\aligned
& \| (P_{k} \psi_{x}(s)) P_{k}(i A_{x}(\psi_{t}(s) - e^{s \Delta} \psi_{t}|_{s = 0})) \|_{L_{t,x}^{1}} \\ 
& \lesssim \| P_{k} \psi_{x}(s) \|_{L_{t,x}^{4}} \| P_{\leq k - 5} A_{x} \|_{L_{t,x}^{4}} \sum_{k - 5 \leq l \leq k + 5} \| P_{l}(\psi_{t}(s) - e^{s \Delta} \psi_{t}|_{s = 0}) \|_{L_{t,x}^{2}} \\
& + \| P_{k} \psi_{x}(s) \|_{L_{t, x}^{4}} \| P_{k - 5 \leq \cdot \leq k + 5} A_{x} \|_{L_{t,x}^{2}} \sum_{l \leq k - 5} \| P_{l}(\psi_{t}(s) - e^{s \Delta} \psi_{t}|_{s = 0}) \|_{L_{t,x}^{4}} \\
& + 2^{k} \| P_{k} \psi_{x}(s) \|_{L_{t}^{\infty} L_{x}^{2}} \sum_{l > k + 5} \| P_{l} A_{x} \|_{L_{t,x}^{2}} \| P_{l}(\psi_{t}(s) - e^{s \Delta} \psi_{t}|_{s = 0}) \|_{L_{t,x}^{2}} \\
& \lesssim \sqrt{\epsilon} v_k^{2} (1 + s 2^{2k})^{-8}.
\endaligned
\end{equation}
\begin{remark}
Note that here we do not use the bootstrap assumption $(\ref{7.2})$ and we obtain decay in $s$.
\end{remark}
Plugging $(\ref{7.8})$ and $(\ref{4.76})$ into $(\ref{7.6})$, the contribution of $A_{x}(\psi_{t}(s) - e^{s \Delta} \psi_{t}|_{s = 0})$ is bounded by
\begin{equation}\label{7.9}
\sqrt{\epsilon} 2^{l - k} v_k^{2} v_l^{2} (1 + s 2^{2k})^{-8} (1 + s' 2^{2l})^{-8},
\end{equation}
which is acceptable for our purposes.

Next,
\begin{equation}\label{7.10}
D_{x}(e^{s \Delta} \psi_{t}|_{s = 0}) = e^{s \Delta} D_{x} \psi_{t}|_{s = 0} + [D_{x}, e^{s \Delta}] \psi_{t}|_{s = 0}.
\end{equation}
When $s = 0$,
\begin{equation}\label{7.11}
\aligned
& D_{x} \psi_{t}|_{s = 0} = i D_{x} D_{l} \psi_{l}|_{s = 0} = i D_{l} D_{l} \psi_{x}|_{s = 0} + i [D_{x}, D_{l}] \psi_{l}|_{s = 0} \\
& = (i \Delta \psi_{x} - 2A \cdot \nabla \psi_{x} - (\nabla \cdot A) \psi_{x} - i \sum_{l} A_{l}^{2} \psi_{x} - \sum_{l} \psi_{l} Im(\bar{\psi}_{l} \psi_{x}))|_{s = 0}.
\endaligned
\end{equation}
Now then, by $(\ref{7.2})$, $(\ref{eqn:psiL4epsbd})$, $(\ref{bound:A})$, and $(\ref{bound:Ax})$,
\begin{equation}\label{7.12}
\aligned
\| P_{k}(((\nabla \cdot A) \psi_{x} - i \sum_{l} A_{l}^{2} \psi_{x} - \sum_{l} \psi_{l} Im(\bar{\psi}_{l} \psi_{x}))|_{s = 0}) (P_{k} \psi_{x}(s)) \|_{L_{t,x}^{1}}  \\ \lesssim \| (\nabla \cdot A) + \sum_{l} A_{l}^{2} + |\psi_{x}|^{2} \|_{L_{t,x}^{2}} \| (P_{k} \psi_{x}(s)) \psi_{x} \|_{L_{t,x}^{2}} \lesssim C \sqrt{\epsilon} v_k^{2}.
\endaligned
\end{equation}
Therefore, examining the kernel of $e^{s \Delta} P_{k}$,
\begin{equation}\label{7.13}
\| P_{k}(e^{s \Delta} ((\nabla \cdot A) \psi_{x} - i \sum_{l} A_{l}^{2} \psi_{x} - \sum_{l} \psi_{l} Im(\bar{\psi}_{l} \psi_{x}))|_{s = 0}) (P_{k} \psi_{x}(s)) \|_{L_{t,x}^{1}} \lesssim C \sqrt{\epsilon} v_k^{2} (1 + s 2^{2k})^{-8}.
\end{equation}
The term $e^{s \Delta}(i \Delta \psi_{x}|_{s = 0})$ will be handled along with the term $-i \Delta \psi_{x}(s)$, so for now consider the term $e^{s \Delta}(-2 A \cdot \nabla \psi_{x}|_{s = 0})$. Similar to the computations in Section $5.1$, by $(\ref{bound:Ax})$,
\begin{equation}\label{7.14}
\aligned
& \| P_{k}(A_{\geq k - 5} \cdot \nabla \psi_{x}|_{s = 0})(P_{k} \psi_{x}(s)) \|_{L_{t,x}^{1}} \\ 
&\lesssim 2^{k} \| A_{\geq k - 5} \|_{L_{t,x}^{2}} \| (P_{k} \psi_{x}) \psi_{x} \|_{L_{t,x}^{2}} + \| \nabla \cdot A_{\geq k - 5} \|_{L_{t,x}^{2}} \| (P_{k} \psi_{x}) \psi_{x} \|_{L_{t,x}^{2}} \\
&\lesssim C \sqrt{\epsilon} v_k^{2}.
\endaligned
\end{equation}
Again by examining the kernel of $e^{s \Delta} P_{k}$,
\begin{equation}\label{7.15}
\| e^{s \Delta} P_{k}(A_{\geq k - 5} \cdot \nabla \psi_{x}|_{s = 0}) (P_{k} \psi_{x}(s)) \|_{L_{t,x}^{1}} \lesssim C \sqrt{\epsilon} v_k^{2} (1 + s 2^{2k})^{-8}.
\end{equation}
Also, by $(\ref{4.89})$ and the kernel of $e^{s \Delta} P_{k}$ and $e^{s \Delta} \tilde{P}_{k}$,
\begin{equation}\label{7.16}
\aligned
\| e^{s \Delta} P_{k}(e^{s \Delta} A_{\leq k - 5} \cdot \nabla \psi_{x}|_{s = 0}) & (P_{k} \psi_{x}(s)) - e^{s \Delta} \tilde{P}_{k} ( (A_{\leq k - 5} \cdot \nabla P_{k} \psi_{x})|_{s = 0}) (P_{k} \psi_{x}(s)) \|_{L_{t,x}^{1}} \\
&  \lesssim \epsilon v_k^{2} (1 + s 2^{2k})^{-8}.
\endaligned
\end{equation}

Next, adapting the control on the Duhamel integrals in Lemma $5.4$ in \cite{dodsonsmith} and applying Lemma \ref{lem:Adecay},
$$\| P_{k}[\psi_{x}(s) - e^{s \Delta} \psi_{x}(0)] \|_{L_{t,x}^{2}} \lesssim v_{k} (s^{-1/2} + 2^{k})^{-1} (1 + s 2^{2k})^{-3},$$
so
\begin{equation}\label{7.17}
\| e^{s \Delta}(A_{\leq k - 5} \cdot \nabla P_{k} \psi_{x}|_{s = 0}) (P_{k}(\psi_{x}(s) - e^{s \Delta} \psi_{x}|_{s = 0})) \|_{L_{t,x}^{1}} \lesssim \sqrt{\epsilon} v_k^{2} (1 + s 2^{2k})^{-7}.
\end{equation}

Splitting 
\begin{equation}
\aligned
& e^{s \Delta} (A_{\leq k - 5}|_{s = 0} \cdot \nabla P_{k} \psi_{k}|_{s = 0}) \\
&  = A_{\leq k - 5}|_{s = 0} e^{s \Delta} P_{k} \nabla \psi_{x}|_{s = 0} + [e^{s \Delta}, A_{\leq k - 5}|_{s = 0}] \nabla P_{k} \psi_{k}|_{s = 0},
\endaligned
\end{equation}
and following $(\ref{4.90})$,
\begin{equation}\label{7.18}
\aligned
\int |P_{l} \psi_{x}(s')|^{2} \frac{(x - y)_{m}}{|x - y|} \cdot Im[P_{k}(\overline{e^{s \Delta} \psi_{x}|_{s = 0}}) \partial_{m} ((A_{\leq k - 5}|_{s = 0} \cdot \nabla P_{k} (e^{s \Delta} \psi_{x}|_{s = 0})] \\
+ \int |P_{l} \psi_{x}(s')|^{2} \frac{(x - y)_{m}}{|x - y|} \cdot Im[A_{\leq k - 5}|_{s = 0} \cdot \nabla P_{k}(\overline{e^{s \Delta} \psi_{x}|_{s = 0}}) \partial_{m}(P_{k} (e^{s \Delta} \psi_{x}|_{s = 0})] \\
\lesssim \| P_{l} \psi_{x}(s') \|_{L_{t}^{\infty} L_{x}^{2}}^{2} \| \nabla A_{\leq k - 5}|_{s = 0} \|_{L_{t,x}^{2}} \| P_{k}(e^{s \Delta} \psi_{x}|_{s = 0}) \|_{L_{t,x}^{4}}^{2} \\ \lesssim \sqrt{\epsilon} v_l^{2} v_k^{2} (1 + s 2^{2k})^{-8} (1 + s' 2^{2l})^{-8}.
\endaligned
\end{equation}
Computing the kernel of $[e^{s \Delta}, A_{\leq k - 5}]$ using Fourier analysis, if $|\xi| \sim 2^{k}$, $|\eta| \ll |\xi|$,
\begin{equation}\label{7.19}
\aligned
& \int e^{-s|\xi|^{2}} \hat{A}(\eta) \hat{f}(\xi - \eta) d\eta - \int e^{-s|\xi - \eta|^{2}} \hat{A}(\eta) \hat{f}(\xi - \eta) d\eta \\ & \lesssim s e^{-s 2^{2k}} |\xi| \int \hat{A}(\eta) \hat{f}(\xi - \eta) |\eta| d\eta \lesssim 2^{-k} \int e^{-s 2^{2k}} \hat{A}(\eta) \hat{f}(\xi - \eta) d\eta.
\endaligned
\end{equation}
Therefore,
\begin{equation}\label{7.20}
\| (P_{k} \psi_{x}(s)) [e^{s \Delta}, A_{\leq k - 5}|_{s = 0}] (\nabla P_{k} \psi_{x}|_{s = 0}) \|_{L_{t,x}^{1}} \lesssim \sqrt{\epsilon} v_k^{2} (1 + s 2^{2k})^{-8}.
\end{equation}
The contribution of the term $[D_{x}, e^{s \Delta}] = [i A_{x}, e^{s \Delta}]$ in $(\ref{7.10})$ is similar. The main difference is that in $(\ref{7.20})$, $A_{x}$ is fixed at $s = 0$, while in $(\ref{7.10})$, $A_{x}(s)$ depends on $s$. Now then, by $(\ref{eqn:psiL4epsbd})$--$(\ref{bound:At})$,
\begin{equation}\label{7.21}
\aligned
& \| P_{\leq k - 5} (A_{x}(s) - A_{x}(0)) \|_{L_{t,x}^{2}} = \| -P_{\leq k - 5} \int_{0}^{s} Im(\bar{\psi_{x}} (D_{l} \psi_{l})) ds' \|_{L_{t,x}^{2}} \lesssim \\
& s^{1/2} \| \psi_{t} \|_{L_{s}^{2} L_{t,x}^{4}} \| \psi_{x} \|_{L_{s}^{\infty} L_{t,x}^{4}} \lesssim s^{1/2} \sqrt{\epsilon}.
\endaligned
\end{equation}
Therefore,
\begin{equation}\label{7.22}
\aligned
\| P_{k}((P_{\leq k - 5} A_{x}(s) - P_{\leq k - 5} A_{x}(0)) e^{s \Delta} D_{l} \psi_{l}|_{s = 0}) (P_{k} \psi_{x}(s))\|_{L_{t,x}^{1}} \\ \lesssim 2^{k} s^{1/2} v_k^{2}  \sqrt{\epsilon} (1 + s 2^{2k})^{-7} \lesssim \sqrt{\epsilon} v_k^{2} (1 + s 2^{2k})^{-6}.
\endaligned
\end{equation}
Finally, by the triangle inequality,
\begin{equation}\label{7.23}
\aligned
& \| P_{k}(A_{\geq k - 5}(s) e^{s \Delta} \psi_{t}|_{s = 0} - e^{s \Delta} (A_{\geq k - 5}|_{s = 0} \psi_{t}|_{s = 0}) (P_{k} \psi_{x}(s)) \|_{L_{t,x}^{1}} \\ & \leq  \| P_{k}(e^{s \Delta}(A_{\geq k - 5}|_{s = 0} \psi_{t}|_{s = 0})) (P_{k} \psi_{x}(s)) \|_{L_{t,x}^{1}} + \| P_{k}(A_{\geq k - 5}(s) e^{s \Delta} \psi_{t}|_{s = 0}) (P_{k} \psi_{x}(s)) \|_{L_{t,x}^{1}}.
\endaligned
\end{equation}
Combining estimates on the kernel of $P_{k} e^{s \Delta}$ with $(\ref{bound:Ax})$,
\begin{equation}\label{7.24}
\aligned
& \| P_{k}(e^{s \Delta}(A_{\geq k - 5}|_{s = 0} \psi_{t}|_{s = 0}))(P_{k} \psi_{x}(s)) \|_{L_{t,x}^{1}} \\
& \lesssim (1 + s 2^{2k})^{-8} \| A_{\geq k - 5}|_{s = 0} \|_{L_{t,x}^{2}} \| P_{\leq k + 5} \psi_{t} \|_{L_{t,x}^{4}} \| P_{k} \psi_{x}(s) \|_{L_{t,x}^{4}} \\ & + (1 + s 2^{2k})^{-8} \sum_{l \geq k + 5} \| P_{l}(A_{x})|_{s = 0} \|_{L_{t,x}^{2}} \| (P_{l} \psi_{t})(P_{k} \psi_{x}(s)) \|_{L_{t,x}^{2}} \\
& \lesssim (1 + s2^{2k})^{-8} 2^{-k} \sqrt{\epsilon} v_k^{2} \sum_{l \leq k + 5} \| P_{l}(D_{m} \psi_{m}) \|_{L_{t,x}^{4}} \\ &+ (1 + s 2^{2k})^{-8} \sum_{l \geq k + 5} \sqrt{\epsilon} v_l 2^{-l} \| P_{l}(D_{m} \psi_{m})(P_{k} \psi_{x}(s)) \|_{L_{t,x}^{2}} \\
& \lesssim (1 + s 2^{2k})^{-8} \epsilon^{3/2} v_k^{2} + (1 + s 2^{2k})^{-8} \sqrt{\epsilon} \sum_{l \geq k + 5} v_l \| (P_{l} \psi_{m})(P_{k} \psi_{x}(s)) \|_{L_{t,x}^{2}} \\
& + (1 + s 2^{2k})^{-8} \sqrt{\epsilon} \sum_{l \geq k + 5} 2^{\frac{k - l}{2}} v_l v_k \| \psi_{x} \|_{L_{t,x}^{4}} \lesssim \epsilon^{3/2} v_k^{2} (1 + s 2^{2k})^{-8}.
\endaligned
\end{equation}
Similarly, by examining the kernel of $P_{l} e^{s \Delta}$,

\begin{equation}\label{7.25}
\aligned
& \| P_{k}(A_{\geq k - 5}(s) (e^{s \Delta} \psi_{t}|_{s = 0}))(P_{k} \psi_{x}(s)) \|_{L_{t,x}^{1}} \\
& \lesssim  \| A_{\geq k - 5}(s) \|_{L_{t,x}^{2}} \| P_{\leq k + 5} \psi_{t}|_{s = 0} \|_{L_{t,x}^{4}} \| P_{k} \psi_{x}(s) \|_{L_{t,x}^{4}} \\ 
&+ \sum_{l \geq k + 5} (1 + s 2^{2l})^{-8} \| P_{l} A_{x}(s) \|_{L_{t,x}^{2}} \| (P_{l} \psi_{t}|_{s = 0})(P_{k} \psi_{x}(s)) \|_{L_{t,x}^{2}} \\
& \lesssim \sum_{l \geq k - 5} v_{k} v_{l} \epsilon^{2} 2^{k - l} (1 + s 2^{2l})^{-7}   + \epsilon^{3/2} v_k^{2} (1 + s 2^{2k})^{-8} \lesssim \epsilon^{3/2} v_{k}^{2} (1 + s 2^{2k})^{-6}.
\endaligned
\end{equation}
The estimates in $(\ref{7.25})$ follow from the bootstrapping assumption
\begin{equation}\label{7.25.1}
\| P_{k} \psi_{x}(s) \|_{L_{t,x}^{4}} \lesssim v_{k} (1 + s 2^{2k})^{-4},
\end{equation}
and Lemma \ref{lem:Adecay},
\begin{equation}\label{7.25.2}
\| P_{k} A_{x}(s) \|_{L_{t,x}^{4}} \lesssim v_{k} (1 + s 2^{2k})^{-\frac72}.
\end{equation}
 Then by direct computation, we can observe a slightly stronger version than \eqref{bound:Ax}
\begin{equation}\label{7.25.3}
\| P_{k} A_{x}(s) \|_{L_{t,x}^{2}} \lesssim \sqrt{\epsilon} v_{k} (1 + s 2^{2k})^{-3}.
\end{equation}
Indeed,
\begin{equation}\label{7.25.4}
\aligned
& \| P_{k} A_{x}(s) \|_{L_{t,x}^{2}} = \| P_{k} (\int_{s}^{\infty} Im(\overline{\psi_{x}(r)} (\partial_{l} \psi_{l} + i A_{l} \psi_{l})) dr \|_{L_{t,x}^{2}} \\
& \lesssim \int_{s}^{\infty} \| P_{k - 5 \leq \cdot \leq k + 5} \psi_{x}(r) \|_{L_{t,x}^{4}} \| \partial_{l} P_{\leq k - 5} \psi_{x}(r) \|_{L_{t,x}^{4}} + \| P_{k - 5 \leq \cdot \leq k + 5} \psi_{x}(r) \|_{L_{t,x}^{4}} \| A_{x} \|_{L_{t,x}^{\infty}} \| \psi_{x} \|_{L_{t,x}^{4}} dr \\
& + \int_{s}^{\infty} 2^{k} \| P_{\leq k - 5} \psi_{x}(r) \|_{L_{t,x}^{4}} \| P_{k - 5 \leq \cdot \leq k + 5} \psi_{x}(r) \|_{L_{t,x}^{4}} dr \\ 
& + \int_{s}^{\infty} \| P_{\leq k - 5} \psi_{x}(r) \|_{L_{t,x}^{4}} \| A_{x} \|_{L_{t,x}^{\infty}} \| P_{\geq k - 5} \psi_{x} \|_{L_{t,x}^{4}} + \| P_{\leq k - 5} \psi_{x}(r) \|_{L_{t,x}^{4}} \| P_{\geq k - 5} A_{x} \|_{L_{t,x}^{4}} \| \psi_{x} \|_{L_{t,x}^{\infty}} dr \\
& + \sum_{l \geq k - 5} \int_{s}^{\infty} 2^{l} \| P_{l} \psi_{x}(r) \|_{L_{t,x}^{4}}^{2} + \| P_{l} A_{x}(r) \|_{L_{t,x}^{4}}^{2} \| \psi_{x} \|_{L_{t,x}^{\infty}} dr \\
& \lesssim \int_{s}^{\infty} 2^{k} (\epsilon E)^{1/2} v_{k} (1 + r 2^{2k})^{-4} dr + \int_{s}^{\infty} r^{-1/2} (1 + r 2^{2k})^{-\frac72} v_{k} (\epsilon E)^{1/2} dr \\ & + \sum_{l \geq k - 5} \int_{s}^{\infty} 2^{l} v_{l}^{2} (1 + r 2^{2l})^{-8} + r^{-1/2} v_{l}^{2} (1 + r 2^{2l})^{-7} dr \lesssim \sqrt{\epsilon} v_{k} (1 + s 2^{2k})^{-3}.
\endaligned
\end{equation}
The proof of $(\ref{7.5.1})$ is similar, except that we replace the bound for $\psi_{x}(s)$ with the bound in $(\ref{7.33.1})$ on $\psi_{t}(s)$.\medskip

Therefore, collecting $(\ref{7.5})$, $(\ref{7.9})$, $(\ref{7.13})$, $(\ref{7.15})$, $(\ref{7.16})$, $(\ref{7.17})$, $(\ref{7.18})$, $(\ref{7.20})$, $(\ref{7.22})$, $(\ref{7.24})$, and $(\ref{7.25})$,
\begin{equation}\label{7.26}
(\partial_{t} - i \Delta) \psi_{x}(s) = F_{2} = F_{2}^{(1)} + F_{2}^{(2)},
\end{equation}
where plugging $F_{2}^{(1)}$ into $(\ref{3.12})$ yields a contribution of $2^{l - k} v_l^{2} v_k^{2} (1 + s 2^{2k})^{-3} (1 + s' 2^{2l})^{-3}$, and 
\begin{equation}\label{7.27}
F_{2}^{(2)} = \partial_{x}(\psi_{t} - e^{s \Delta} \psi_{t}|_{s = 0}) - i \Delta (\psi_{x} - e^{s \Delta} \psi_{x}|_{s = 0}).
\end{equation}
When computing the contribution of $F_{2}^{(2)}$, it is convenient to rewrite 
$$-2 A \cdot \nabla \psi - (\nabla \cdot A) \psi = -2 \nabla \cdot (A \psi) + (\nabla A) \cdot \psi.$$
By the product rule,
\begin{equation}\label{7.28}
\aligned
& \partial_{x} \int_{0}^{s} e^{(s - s') \Delta} \nabla \cdot (A \psi_{t})(s') ds' - i \Delta \int_{0}^{s} e^{(s - s') \Delta} \nabla \cdot (A \psi_{x})(s') ds' \\
&  = \nabla \cdot \int_{0}^{s} e^{(s - s') \Delta} A  [\partial_{x} \psi_{t} - i \Delta \psi_{x}] ds'
+ \int_{0}^{s} e^{(s - s') \Delta} \nabla \cdot ((\partial_{x} A) \psi_{t})(s') ds' \\
& - i \int_{0}^{s} e^{(s - s') \Delta} \nabla \cdot (\partial_{m} A \partial_{m} \psi_{x})(s') ds' - i \int_{0}^{s} e^{(s - s') \Delta} \nabla \partial_{m} \cdot ((\partial_{m} A) \psi_{x})(s') ds'.
\endaligned
\end{equation}
Examining the proof of Lemma $5.6$ in \cite{dodsonsmith},
\begin{equation}\label{7.29}
\| P_{k} [\partial_{x} \psi_{t}(s) - i \Delta \psi_{x}(s)] \|_{L_{t,x}^{2}} \lesssim v_k 2^{k} (1 + s 2^{2k})^{-4}.
\end{equation}
Therefore, making a tri--linear argument,
\begin{equation}\label{7.30}
\aligned
& \| \nabla \cdot P_{k} \int_{0}^{s} e^{(s - s') \Delta} A[\partial_{x} \psi_{t} - i \Delta \psi_{x}] ds' \|_{L_{t,x}^{4/3}} \lesssim \sum_{l \leq k} \int_{0}^{s} (1 + s 2^{2k})^{-4} \sqrt{\epsilon} v_k v_l 2^{2l} (1 + s' 2^{2l})^{-4} ds' \\
& + 2^{2k} \int_{0}^{s} (1 + (s - s') 2^{2k})^{-4} \sqrt{\epsilon} v_k (1 + s' 2^{2k})^{-4} ds' \\
& \hspace{2cm}    + \sum_{l \geq k} 2^{2k} \int_{0}^{s} (1 + (s - s') 2^{2k})^{-4} \sqrt{\epsilon} v_l v_l (1 + s' 2^{2l})^{-4} \\
& \lesssim \sqrt{\epsilon} v_k + \sqrt{\epsilon} v_k (1 + s 2^{2k})^{-4}.
\endaligned
\end{equation}  
Next, adapting the control on the Duhamel integrals in Lemma $5.4$ in \cite{dodsonsmith} and applying Lemma \ref{lem:Adecay}, we have
\begin{equation}\label{7.31}
\| P_{k}(\psi_{t}(s) - e^{s \Delta} \psi_{t}|_{s = 0}) \|_{L_{t,x}^{2}} \lesssim v_k (1 + s 2^{2k})^{-3},
\end{equation}
and
\begin{equation}\label{7.32}
\aligned
& \| P_{k}(A_{l} \psi_{l}) \|_{L_{t,x}^{2}} \lesssim \| P_{\geq k - 5} A_{x} \|_{L_{t,x}^{2}} \| P_{\leq k - 5} \psi_{x} \|_{L_{t,x}^{\infty}} \\
& + \| P_{\leq k - 5} A_{x} \|_{L_{t,x}^{4}} \| P_{k - 5 \leq \cdot \leq k + 5} \psi_{x} \|_{L_{t,x}^{4}} + \sum_{l \geq k + 5} 2^{k} \| P_{l} A_{x} \|_{L_{t,x}^{2}} \| P_{l} \psi_{x} \|_{L_{t}^{\infty} L_{x}^{2}} \lesssim \epsilon v_k,
\endaligned
\end{equation}
therefore,
\begin{equation}\label{7.33}
\| e^{s \Delta} P_{k} (A_{l} \psi_{l}|_{s = 0}) \|_{L_{t,x}^{2}} \lesssim \epsilon v_k (1 + s 2^{2k})^{-4}.
\end{equation}
As a result,
\begin{equation}\label{7.33.1}
\| P_{k} \psi_{t} \|_{L_{t}^{2} L_{x}^{\infty} + L_{t,x}^{4}} \lesssim v_k 2^{k} (1 + s 2^{2k})^{-4}.
\end{equation}
Hence,
\begin{equation}\label{7.34}
\aligned
& \| (P_{k} \psi_{x}(s)) P_{k}((\partial_{x} A) \psi_{t})(s') \|_{L_{t,x}^{1}} \\
& \lesssim \| P_{k} \psi_{x}(s) \|_{L_{t,x}^{4} \cap L_{t}^{\infty} L_{x}^{2}} \| \partial_{x} A(s') \|_{L_{t,x}^{2}} \| P_{k - 5 \leq \cdot \leq k + 5} \psi_{t}(s') \|_{L_{t}^{2} L_{x}^{\infty} + L_{t,x}^{4}} \\
& + \| P_{k} \psi_{x}(s) \|_{L_{t,x}^{4}} \| P_{k - 5 \leq \cdot \leq k + 5} \partial_{x} A_{x}(s') \|_{L_{t,x}^{2}} \| P_{\leq k - 5} \psi_{t}(s') \|_{L_{t,x}^{4}}
\\
&+ \sum_{l \geq k + 5} \| (P_{k} \psi_{x}(s)) P_{l}(e^{s \Delta} \partial_{x} \psi_{x}|_{s = 0}) \|_{L_{t,x}^{2}} \| \partial_{x} A_{x} \|_{L_{t,x}^{2}} \\
& + \sum_{l \geq k + 5} 2^{k} \| P_{k} \psi_{x}(s) \|_{L_{t}^{\infty} L_{x}^{2}} \| \partial_{x} A \|_{L_{t,x}^{2}} \| P_{l}(\psi_{t}(s') - e^{s' \Delta} (i \partial_{x} \psi_{x}|_{s' = 0}) \|_{L_{t,x}^{2}} \\
& \lesssim 2^{k} \sqrt{\epsilon} v_k^{2} (1 + s 2^{2k})^{-4} (1 + s' 2^{2k})^{-4}  + 2^{k} \sqrt{\epsilon} v_k^{2} (1 + s 2^{2k})^{-8}.
\endaligned
\end{equation}
Therefore,
\begin{equation}\label{7.35}
\aligned
\| (P_{k} \psi_{x}(s)) \int_{0}^{s} e^{(s - s') \Delta} \nabla \cdot P_{k} ((\partial_{x} A) \psi_{t})(s') ds'\|_{L_{t,x}^{1}} \\ \lesssim
\int_{0}^{s} (1 + (s - s') 2^{2k})^{-4} [2^{2k} \sqrt{\epsilon} v_k^{2} (1 + s 2^{2k})^{-4} (1 + s' 2^{2k})^{-4} \\  + 2^{2k} \sqrt{\epsilon} v_k^{2} (1 + s 2^{2k})^{-8} ds'] \lesssim \sqrt{\epsilon} v_k^{2} (1 + s 2^{2k})^{-4}.
\endaligned
\end{equation}
By $(\ref{4.2})$, it is possible to obtain a similar estimate for
\begin{equation}\label{7.36}
\| (P_{k} \psi_{x}(s)) \cdot \int_{0}^{s} e^{(s - s') \Delta} \nabla \cdot P_{k} (\partial_{m} A \partial_{m} \psi_{x})(s') ds' \|_{L_{t,x}^{1}} \lesssim \sqrt{\epsilon} v_k^{2}.
\end{equation}
Next,
\begin{equation}\label{7.37}
\aligned
 & \|(P_{k} \psi_{x}(s)) \cdot \int_{0}^{s} e^{(s - s') \Delta} \nabla \partial_{m} \cdot P_{k}((\partial_{m} A) \psi_{x})(s') ds' \|_{L_{t,x}^{1}} \\
 & \lesssim \| \partial_{x} A \|_{L_{t,x}^{2}} \int_{0}^{s} 2^{2k} (1 + (s - s')2^{2k})^{-4} \| (P_{k} \psi_{x}(s)) \psi_{x}(s') \|_{L_{t,x}^{2}} \lesssim \sqrt{\epsilon} v_k^{2}.
 \endaligned
\end{equation}

The contribution of $\int_{0}^{s} e^{(s - s') \Delta} \partial_{x}((\nabla \cdot A) \psi_{t}) ds'$ and $i \Delta \int_{0}^{s} e^{(s - s') \Delta} ((\nabla \cdot A) \psi_{x}) ds'$ can be handled in a manner identical to $(\ref{7.35})$ and $(\ref{7.37})$. Next, by $(\ref{4.76})$ and $(\ref{4.77})$,
\begin{equation}\label{7.38}
\| P_{k}(A_{x}^{2} \psi_{x} + \psi_{x}^{3})(s') \|_{L_{t,x}^{4/3}} \lesssim \epsilon v_k(1 + s' 2^{2k})^{-4} + \epsilon \sum_{l \geq k + 5} v_l (1 + s' 2^{2l})^{-4}.
\end{equation}
Therefore,
\begin{equation}\label{7.39}
\aligned
& \| (P_{k} \psi_{x}(s)) \cdot \Delta \int_{0}^{s} e^{(s - s') \Delta} P_{k}(A_{x}^{2} \psi_{x} + \psi_{x}^{3}) ds' \|_{L_{t,x}^{1}} \\ 
& \lesssim \epsilon v_k^{2} (1 + s 2^{2k})^{-4} 2^{2k} \int_{0}^{s} (1 + s' 2^{2k})^{-4} (1 + (s - s') 2^{2k})^{-4} ds' \\
& + \epsilon v_k \sum_{l \geq k + 5} v_l (1 + s 2^{2k})^{-4} 2^{2k} \int_{0}^{s} (1 + s' 2^{2l})^{-4} (1 + (s - s') 2^{2k})^{-4} ds' \lesssim \epsilon v_k^{2} (1 + s 2^{2k})^{-6}.
\endaligned
\end{equation}
By a similar calculation,
\begin{equation}\label{7.40}
\aligned
& \| (P_{k} \psi_{x}(s)) \cdot \partial_{x} \int_{0}^{s} e^{(s - s') \Delta} P_{k}((A_{x}^{2} + \psi_{x}^{2})\psi_{t} ds' \|_{L_{t,x}^{1}} \\ 
& \lesssim \epsilon v_k^{2} (1 + s 2^{2k})^{-4} 2^{2k} \int_{0}^{s} (1 + s' 2^{2k})^{-4} (1 + (s - s') 2^{2k})^{-4} ds' \\
& + \epsilon v_k \sum_{l \geq k + 5} v_l (1 + s 2^{2k})^{-4} 2^{k} 2^{l} \int_{0}^{s} (1 + s' 2^{2l})^{-4} (1 + (s - s') 2^{2k})^{-4} ds' \lesssim \epsilon v_k^{2} (1 + s 2^{2k})^{-6}.
\endaligned
\end{equation}
Therefore, under the bootstrap assumption $(\ref{7.2})$, for $l \leq k - 10$,
\begin{equation}\label{7.41}
\| (P_{k} \psi_{x}(s))(P_{l} \psi_{x}(s')) \|_{L_{t,x}^{2}}^{2} \lesssim 2^{l - k} \| P_{k} \psi_{x}(s) \|_{L_{t}^{\infty} L_{x}^{2}} \| P_{l} \psi_{x}(s') \|_{L_{t}^{\infty} L_{x}^{2}} + 2^{l - k} \sqrt{\epsilon} C v_l^{2} v_k^{2},
\end{equation}
which proves the lemma.
\end{proof}

With the bound \eqref{7.2} in hand, we can proceed to establish the decay in $s,s'$.  In particular, recalling our assumption that $\ell \leq k-10$, we claim that
\begin{equation}
\label{7.42}
\| (P_{k} \psi_{x}(s))(P_{l} \psi_{x}(0)) \|_{L_{t,x}^{2}}^{2} \lesssim 2^{l - k} v_k^{2} v_l^{2} (1 + s 2^{2k})^{-3} .
\end{equation}
for $s \geq 0$.  This follows very much in the line of reasoning as in Section $7$ of \cite{dodson2012bilinear} and the tools introduced in Lemma \ref{lem:lem5} above.  In particular, the bound we have proven suffices for $s < 2^{-2k}$.  For $s > 2^{-2k}$, we observe that splitting the time integration into $[0,(1-\delta) s]$ for some $0<\delta<1$ holds using heat kernel decay estimates.  To integrate over $[(1-\delta) s,s]$ for $0<\delta<1$, utilizing
\begin{equation}
\label{7.43}
\| (P_{k} \psi_{x}(s))(P_{l} \psi_{x}(0)) \|_{L_{t,x}^{2}}^{2} \lesssim  v_k^{2} v_l^{2}  (1 + s 2^{2k})^{-3} .
\end{equation}
 as a bootstrapping assumption, we then can extend the result to all $s$.  With this established, we can also integrate in $s'$ in a similar way.

\subsection{Proof of Lemma $\ref{lL4}$} 
\begin{proof}[Proof of Lemma $\ref{lL4}$]
Suppose $v \in V_{\Delta}^{4/3} \subset U_{\Delta}^{2}$, where
\begin{equation}\label{6.30}
\| v \|_{V_{\Delta}^{4/3}}^{2} = 1, \qquad \tilde{P}_{k} v = v.
\end{equation}
Observe that $V_{\Delta}^{4/3} \subset U_{\Delta}^{2}$. Then
\begin{equation}\label{6.31}
\aligned
& \int_{0}^{t} (A_{t} + \sum_{l = 1}^{d} (A_{l}^{2} - i \partial_{l} A_{l}) \psi_{m}, v)_{L^{2}} dt' - \int_{0}^{t} (i \sum_{l = 1}^{d} \psi_{l} Im(\bar{\psi}_{l} \psi_{m}), v)_{L^{2}} dt' \\  & \lesssim \| v \psi_{x} \|_{L_{t,x}^{2}} \| A_{t} \|_{L_{t,x}^{2}}  + \| A_{x} \|_{L_{t,x}^{4}}^{2} \| v \psi_{x} \|_{L_{t,x}^{2}} + \| \psi_{x} \|_{L_{t,x}^{4}}^{2} \| v \psi_{x} \|_{L_{t,x}^{2}} \lesssim \epsilon v_{k}
\endaligned
\end{equation}
Next, following $(\ref{6.23})$,
\begin{equation}\label{6.32}
\aligned
&  \int_{0}^{t} (\sum_{l = 1}^{d} P_{> k - 10} A_{l} \cdot \partial_{l} \psi_{x}, v) dt'  = \sum_{j > k - 10} \int_{0}^{t} (\sum_{l = 1}^{d} P_{j} A_{l} \cdot \partial_{l} \psi_{x}, v) dt' \lesssim \| \nabla A_{x} \|_{L_{t,x}^{2}} \| v \psi_{x} \|_{L_{t,x}^{2}} \\
 &  \lesssim \epsilon v_{k}.
\endaligned
\end{equation}
Now turn to the most difficult term, $\langle v, P_{\leq k - 10} A_{x} \cdot \nabla \psi_{x} \rangle$. It is here that we need the improved estimates in $(\ref{bilinear})$. In this case,
\begin{equation}\label{6.32.1}
\aligned
\langle v, (\nabla P_{k} \psi_{x}) \int_{0}^{\infty} P_{\leq k - 10} Im(\overline{\psi_{x}} (\partial_{l} + iA_{l}) \psi_{l}) ds \rangle \lesssim \sum_{j \leq k - 10} 2^{k - j} \| v (P_{j} \psi_{x}) \|_{L_{t,x}^{2}} \| (P_{k} \psi_{x})(P_{j} \psi_{x}) \|_{L_{t,x}^{2}} \\
+ \epsilon v_{k} + \beta_{k} \sum_{j \leq k} \beta_{j}^{2} + \epsilon v_{k} \sum_{j \leq k} v_{j}^{2}.
\endaligned
\end{equation}
This completes the proof of Lemma $\ref{lL4}$.

\end{proof}

\section{Proof of Theorem 2}
\label{sec:thm2pf}

We are now ready to give a proof of Theorem $\ref{t1.1}$. For the reader's convenience, we recall the theorem here.
\begin{theorem}\label{t6.1}
Let $n = 2$, $\mathcal N$ be a $2n$-dimensional compact K{\"a}hler manifold that is isometrically embedded into $\mathbb{R}^{N}$, and let $Q \in \mathcal N$ be a given point. There exists a sufficiently small constant $\epsilon_{\ast}(\| u_{0} - Q \|_{\dot{H}^{1}}) \ll 1$ such that if $u_{0} \in \mathcal H_{Q}$ satisfies
\begin{equation}\label{6.1}
\| u_{0} - Q \|_{B_{\infty, 2}^{1}} \leq \epsilon_{\ast} \ll 1,
\end{equation}
then $(\ref{1.1})$ with initial data $u_{0}$ evolves into a global unique solution $u \in C(\mathbb{R}; \mathcal H_{Q})$. Moreover, as $|t| \rightarrow \infty$ the solution $u$ converges to a constant map in the sense that
\begin{equation}\label{6.2}
\lim_{|t| \rightarrow \infty} \| u(t) - Q \|_{L_{x}^{\infty}} = 0.
\end{equation}
Furthermore, in the energy space we also have
\begin{equation}\label{6.3}
\lim_{t \rightarrow \infty} \| u(t) - \sum_{j = 1}^{n} Re(e^{it \Delta} h_{+}^{j}) - \sum_{j = 1}^{n} Im(e^{it \Delta} g_{+}^{j}) \|_{\dot{H}^{1}} = 0,
\end{equation}
for some functions $h_{+}^{j}$, $g_{+}^{j} : \mathbb{R}^{2} \rightarrow \mathbb{C}^{N}$ belonging to $\dot{H}^{1}$ with $j = 1, ..., n$.
\end{theorem}
The proof of Theorem $\ref{t1.1}$ uses perturbative arguments when $\epsilon_{\ast}$ in $(\ref{5.1})$ is small. Recall that the derivative solution to $(\ref{1.1})$ satisfies the equation,
\begin{equation}\label{6.4}
(i \partial_{t} + \Delta) \psi_{m} = -2i \sum_{l = 1}^{d} A_{l} \partial_{l} \psi_{m} + (A_{t} + \sum_{l = 1}^{d} (A_{l}^{2} - i \partial_{l} A_{l})) \psi_{m} - i \sum_{l = 1}^{d} \psi_{l} Im(\bar{\psi}_{l} \psi_{m}).
\end{equation}
The terms
\begin{equation}\label{6.5}
(A_{t} + \sum_{l = 1}^{d} (A_{l}^{2} - i \partial_{l} A_{l})) \psi_{m} - i \sum_{l = 1}^{d} \psi_{l} Im(\bar{\psi}_{l} \psi_{m}),
\end{equation}
may be treated as cubic terms. Indeed, recall from Section $\ref{sec:Gauges}$ that $A_{d + 1} \in L_{t,x}^{2}$, $A_{x} \in L_{t,x}^{4}$ and $\partial_{l} A_{l} \in L_{t,x}^{2}$. Therefore, if we somehow could choose a gauge to make the term $-2i \sum_{l = 1}^{d} A_{l} \partial_{l} \psi_{m}$ negligible, it would probably be possible to prove a result analogous to Theorem $\ref{t2.1}$ for $(\ref{6.4})$.\medskip

Because the term
\begin{equation}\label{6.6}
-2i \sum_{l = 1}^{d} A_{l} \partial_{l} \psi_{m},
\end{equation}
is quasilinear, such computations are more technically difficult, and we will utilize the bilinear estimates from \cite{ifrim2023global}, \cite{ifrim2023long}, and \cite{ifrim2023global2} for quasilinear problems.\medskip

Recall from the Theorem $\ref{t3.1}$ and Section $\ref{sec:warmup}$ that the bilinear estimate is bounded by the energy at the frequencies estimated, $\sup_{t} \| P_{k} \psi_{m}(t) \|_{L^{2}} \| P_{j} \psi_{m} \|_{L^{2}}$, and the contributions of the nonlinear terms. When $\psi_{x}$ solves a nonlinear equation, $\| P_{k} \psi_{x} \|_{L^{2}}$ is no longer a conserved quantity, Therefore,
\begin{lemma}\label{l6.2}
For any $t$,
\begin{equation}\label{6.9}
\aligned
& \| P_{k} \psi_{m}(t) \|_{L^{2}}^{2} = \| P_{k} \psi_{m}(0) \|_{L^{2}}^{2} + \int_{0}^{t} (-2i P_{k}(\sum_{l = 1}^{d} A_{l} \partial_{l} \psi_{m}), P_{k} \psi_{m})_{L^{2}} dt' \\
& + \int_{0}^{t} (P_{k}(A_{t} + \sum_{l = 1}^{d} (A_{l}^{2} - i \partial_{l} A_{l})) \psi_{m}, P_{k} \psi_{m})_{L^{2}} dt' - \int_{0}^{t} (i P_{k} \sum_{l = 1}^{d} \psi_{l} Im(\bar{\psi}_{l} \psi_{m}), P_{k} \psi_{m})_{L^{2}} dt'.
\endaligned
\end{equation}
\end{lemma}
Since the terms in the Morawetz estimate are similar, we postpone estimating the terms in $(\ref{6.9})$ until we state the terms arising in the Morawetz estimate. Take
\begin{align}
\label{6.10}
M(t) &  = \int |P_{j} \psi_{y}(t, y)|^{2} \frac{(x - y)}{|x - y|} \cdot Im[\overline{P_{k} \psi_{x}(t,x)} \nabla P_{k} \psi_{x}(t,x)]  \\
& + \int |P_{k} \psi_{y}(t, y)|^{2} \frac{(x - y)}{|x - y|} \cdot Im[P_{j} \overline{\psi_{x}(t,x)} \nabla P_{j} \psi_{x}(t,x)]. \notag
\end{align}

Compute
\begin{equation}\label{6.11}
\frac{d}{dt} M(t) = 2 \int_{(x - y) \cdot \omega = 0} |\partial_{\omega}(P_{k} \overline{\psi_{x}(t, x)} P_{j} \psi_{y}(t,y))|^{2} dx dy d\omega
\end{equation}
\begin{align}\label{6.12}
& -2 \int |P_{j} \psi_{y}(t, y)|^{2} \frac{(x - y)}{|x - y|} \cdot Re[P_{k}(\sum_{l = 1}^{d} \overline{A_{l} \partial_{l} \psi_{x}}) \nabla P_{k} \psi_{x}]  \\
& \hspace{1cm} - 2 \int |P_{j} \psi_{y}(t, y)|^{2} \frac{(x - y)}{|x - y|} \cdot Re[\overline{P_{k} \psi_{x}} \nabla P_{k}(\sum_{l = 1}^{d} A_{l} \partial_{l} \psi_{x})] \notag
\end{align}
\begin{equation}\label{6.13}
+ \int |P_{j} \psi_{y}(t, y)|^{2} \frac{(x - y)}{|x - y|} \cdot Re[P_{k}(\overline{A_{t} + \sum_{l = 1}^{d} (A_{l}^{2} - i \partial_{l} A_{l}) \psi_{x})} \nabla P_{k} \psi_{x}]
\end{equation}
\begin{equation}\label{6.14}
+ \int |P_{j} \psi_{y}(t, y)|^{2} \frac{(x - y)}{|x - y|} \cdot Re[P_{k}(\sum_{l = 1}^{d} \psi_{l} Im(\psi_{l} \psi_{x})) \nabla \overline{P_{k} \psi_{x}}]
\end{equation}
\begin{align}\label{6.12.1}
& -2 \int |P_{k} \psi_{y}(t, y)|^{2} \frac{(x - y)}{|x - y|} \cdot Re[P_{j}(\sum_{l = 1}^{d} \overline{A_{l} \partial_{l} \psi_{x}}) \nabla P_{j} \psi_{x}]  \\
& \hspace{1cm} - 2 \int |P_{k} \psi_{y}(t, y)|^{2} \frac{(x - y)}{|x - y|} \cdot Re[\overline{P_{j} \psi_{x}} \nabla P_{j}(\sum_{l = 1}^{d} A_{l} \partial_{l} \psi_{x})] \notag
\end{align}
\begin{equation}\label{6.13.1}
+ \int |P_{k} \psi_{y}(t, y)|^{2} \frac{(x - y)}{|x - y|} \cdot Re[P_{j}(\overline{A_{t} + \sum_{l = 1}^{d} (A_{l}^{2} - i \partial_{l} A_{l}) \psi_{x})} \nabla P_{j} \psi_{x}]
\end{equation}
\begin{equation}\label{6.14.1}
+ \int |P_{k} \psi_{y}(t, y)|^{2} \frac{(x - y)}{|x - y|} \cdot Re[P_{j}(\sum_{l = 1}^{d} \psi_{l} Im(\psi_{l} \psi_{x})) \nabla \overline{P_{j} \psi_{x}}].
\end{equation}

Now make the following bootstrap assumptions,
\begin{equation}\label{6.15}
\aligned
\sup_{|j - k| \geq 10} 2^{\frac{|j - k|}{2}} \| (P_{j} \psi_{x})(P_{k} \psi_{x}) \|_{L_{t,x}^{2}} \leq C \epsilon^{2}, \\
\sup_{k} \| P_{k} \psi_{x} \|_{L_{t}^{\infty} L_{x}^{2}} \leq C \epsilon, \\
\| P_{k} \psi_{x} \|_{L_{t,x}^{4}} \leq C(E) v_k.
\endaligned
\end{equation}

Now then, first observe that by $(\ref{eq:L4norm})$, under the bootstrap assumptions in $(\ref{6.15})$,
\begin{equation}\label{6.15.1}
\| \psi_{x} \|_{L_{t,x}^{4}}^{4} \lesssim C(E) \epsilon^{2}.
\end{equation}

Now then, computing the bilinear estimate, suppose without loss of generality that $j \leq k - 10$.
\begin{equation}\label{6.16}
\aligned
& (\ref{6.14}) \lesssim \\
&  2^{k} \| P_{j} \psi_{x} \|_{L_{t}^{\infty} L_{x}^{2}}^{2} \sum_{|j_{3} - k| \leq |j_{2} - k| \leq |j_{1} - k|} \| (P_{k} \psi_{x})(P_{j_{1}} \psi_{x}) \|_{L_{t,x}^{2}} \| P_{j_{2}} \psi_{x} \|_{L_{t,x}^{4}} \| P_{j_{3}} \psi_{x} \|_{L_{t,x}^{4}} \lesssim 2^{k} \epsilon^{5} C(E).
\endaligned
\end{equation}
Next,
\begin{equation}\label{6.17}
\aligned
(\ref{6.13}) & \lesssim 2^{k} \| P_{j} \psi_{x} \|_{L_{t}^{\infty} L_{x}^{2}}^{2} \| (P_{k} \psi_{x}) P_{k}(A_{l}^{2} \psi_{x}) \|_{L_{t,x}^{1}} + 2^{k} \| P_{j} \psi_{x} \|_{L_{t}^{\infty} L_{x}^{2}}^{2} \| (P_{k} \psi_{x}) P_{k}(\nabla \cdot A \psi_{x}) \|_{L_{t,x}^{1}} \\
& + 2^{k} \| P_{j} \psi_{x} \|_{L_{t}^{\infty} L_{x}^{2}}^{2} \| (P_{k} \psi_{x}) P_{k}(A_{t} \psi_{x}) \|_{L_{t,x}^{1}}.
\endaligned
\end{equation}
By $(\ref{eqn:Abdassumps_alt})$-\eqref{bound:A} and $(\ref{eqn:bootstrap_assumptions})$,
\begin{equation}\label{6.18}
2^{k} \| P_{l} \psi_{x} \|_{L_{t}^{\infty} L_{x}^{2}}^{2} \| (P_{k} \psi_{x}) P_{k}(A_{l}^{2} \psi_{x}) \|_{L_{t,x}^{1}} \lesssim \epsilon^{2} \| (P_{k} \psi_{x}) \psi_{x} \|_{L_{t,x}^{2}} \| A_{x} \|_{L_{t,x}^{4}}^{2} \lesssim 2^{k} C(E) \epsilon^{5},
\end{equation}
\begin{equation}\label{6.19}
2^{k} \| P_{l} \psi_{x} \|_{L_{t}^{\infty} L_{x}^{2}}^{2} \| (P_{k} \psi_{x}) P_{k}(\nabla \cdot A \psi_{x}) \|_{L_{t,x}^{1}} \lesssim \epsilon^{2} \| (P_{k} \psi_{x}) \psi_{x} \|_{L_{t,x}^{2}} \| \nabla \cdot A_{x} \|_{L_{t,x}^{2}} \lesssim 2^{k} C(E) \epsilon^{5}.
\end{equation}
\begin{equation}\label{6.20}
2^{k} \| P_{l} \psi_{x} \|_{L_{t}^{\infty} L_{x}^{2}}^{2} \| (P_{k} \psi_{x}) P_{k}(A_{t} \psi_{x}) \|_{L_{t,x}^{1}} \lesssim \epsilon^{2} \| (P_{k} \psi_{x}) \psi_{x} \|_{L_{t,x}^{2}} \| A_{t} \|_{L_{t,x}^{2}} \lesssim 2^{k} C(E) \epsilon^{4}.
\end{equation}
Following similar calculations,
\begin{equation}\label{6.20.1}
(\ref{6.14.1}) + (\ref{6.13.1}) \lesssim 2^{k} C(E) \epsilon^{4},
\end{equation}
and
\begin{equation}\label{6.21}
(P_{k}(A_{t} + \sum_{l = 1}^{d} (A_{l}^{2} - i \partial_{l} A_{l})) \psi_{m}, P_{k} \psi_{m})_{L^{2}} - (i P_{k} \sum_{l = 1}^{d} \psi_{l} Im(\bar{\psi}_{l} \psi_{m}), P_{k} \psi_{m})_{L^{2}} \lesssim C(E)\epsilon^{3}.
\end{equation}

Now split
\begin{equation}\label{6.22}
A = P_{\leq k - 5} A + P_{> k - 5} A.
\end{equation}
By $(\ref{eqn:Abdassumps_alt})$-\eqref{bound:A}, computing $(\ref{6.12})$ with $A$ replaced by $P_{> k - 5} A$, under the bootstrap assumption,
\begin{equation}\label{6.23}
\lesssim 2^{k} \| P_{j} \psi_{x} \|_{L_{t}^{\infty} L_{x}^{2}}^{2} \sum_{l \geq k - 5} \| P_{l} A_{x} \|_{L_{t,x}^{2}} \| (\nabla P_{\leq l + 10} \psi_{x})(P_{k} \psi_{x}) \|_{L_{t,x}^{2}} \lesssim C(E) \epsilon^{4}.
\end{equation}

Finally, computing the commutator, if $\tilde{P}_{k} = P_{k - 2} + P_{k - 1} + P_{k} + P_{k + 1} + P_{k + 2}$,
\begin{equation}\label{6.25}
P_{k} ((P_{\leq k - 5} A) \psi_{x}) = P_{k}((P_{\leq k - 5} A) \tilde{P}_{k} \psi_{x}) = (P_{\leq k - 5} A) P_{k} \psi_{x} + [P_{k}, P_{\leq k - 5} A] (\tilde{P}_{k} \psi_{x}).
\end{equation}

\begin{remark}
The estimates with $j$ and $k$ switched are similar.
\end{remark}
By similar computations,
\begin{equation}\label{6.27}
\int_{0}^{t} (-2i P_{k}(\sum_{l = 1}^{d} A_{l} \partial_{l} \psi_{m}), P_{k} \psi_{m})_{L^{2}} dt' \lesssim C(E) \epsilon^{3}.
\end{equation}
Therefore, we have proved, under the bootstrap assumption $(\ref{eqn:bootstrap_assumptions})$,
\begin{equation}\label{6.28}
\| P_{k} \psi_{x}(t) \|_{L^{2}}^{2} \leq \| P_{k} \psi_{x}(0) \|_{L^{2}}^{2} + O(C(E) \epsilon^{3}),
\end{equation}
and
\begin{equation}\label{6.29}
\sup_{|j - k| \geq 10} 2^{\frac{|j - k|}{2}} \| (P_{j} \psi_{x})(P_{k} \psi_{x}) \|_{L_{t,x}^{2}}^{2} \leq \sup_{j} \| P_{j} \psi_{x}(0) \|_{L^{2}}^{4} + O(C(E) \epsilon^{4}).
\end{equation}
Now then, under the Besov norm assumption on the initial data, $(\ref{1.14})$, the first two terms in bootstrap $(\ref{6.15})$ are confirmed.\medskip

Next, observe that by similar computations, under the bootstrap assumptions $(\ref{6.15})$, we can show
\begin{equation}\label{6.29.1}
\sup_{|j - k| \geq 10} 2^{\frac{|j - k|}{2}} \| (P_{j} \psi_{x}) (P_{k} \psi_{x}) \|_{L_{t,x}^{2}} \leq C v_k v_j ,
\end{equation}
and 
\begin{equation}
\label{6.29.2}
\sup_{k} \| P_{k} \psi_{x} \|_{L_{t}^{\infty} L_{x}^{2}} \leq C v_k.
\end{equation}

We use $(\ref{6.29.1})$ and $(\ref{6.29.2})$ to prove an improved bound on $\| P_{k} \psi_{x} \|_{L_{t,x}^{4}}$.

\section*{Acknowledgements}
B.D. gratefully acknowledges support from NSF Analysis Grant DMS--2153750. J.L.M. gratefully acknowledges support from NSF Applied Math Grant DMS--2307384.

\bibliography{biblio}
\bibliographystyle{alpha}

\end{document}